\documentclass[onefignum,onetabnum]{siamart190516}

\usepackage{amssymb,latexsym,enumerate,verbatim,amsfonts}
\usepackage{subfig}
\usepackage{lipsum}
\usepackage{graphicx,microtype}
\usepackage{color}
\usepackage{epstopdf}
\usepackage[active]{srcltx}
\usepackage{algorithmic}
\usepackage{enumitem}
\usepackage{booktabs}	
\usepackage[normalem]{ulem}		
\usepackage{xcolor}
\usepackage{cite}

\newif\ifistoreview
\istoreviewtrue
\newcommand{\R}{\mathbb{R}}

\newcommand{\lb}{\lambda}
\newcommand{\argmin}{\arg\min}

\newcommand{\st}{\textnormal{s.t.}}
\newcommand{\zz}{^\top}
\newcommand{\dist}{\textnormal{dist}}

\newcommand\bovermat[2]{%
	\makebox[0pt][l]{$\smash{\overbrace{\phantom{%
					\begin{matrix}#2\end{matrix}}}^{\text{$#1$}}}$}#2}

\newsiamremark{remark}{Remark}
\newsiamremark{hypothesis}{Hypothesis}
\crefname{hypothesis}{Hypothesis}{Hypotheses}
\newsiamthm{claim}{Claim}

\title{Linearly-constrained nonsmooth optimization for training autoencoders\thanks{Submitted on 29 March, 2021. Revised on 17 January, 2022
		\funding{This work is supported partly by 
			the National Natural Science Foundation of China (No. 12125108, 11971466 and 11991021), 
			Hong Kong Research Grants Council grant PolyU15300120,
			Key Research Program of Frontier Sciences, 
			Chinese Academy of Sciences (No. ZDBS-LY-7022) and the CAS AMSS-PolyU Joint Laboratory in Applied Mathematics.}}}

\author{Wei Liu\thanks{State Key Laboratory of Scientific and Engineering Computing, Academy of { Mathematics} and
		Systems Science, Chinese Academy of Sciences, and University of Chinese Academy of Sciences,
		China (\email{liuwei175@lsec.cc.ac.cn}, \email{liuxin@lsec.cc.ac.cn}).}
	\and Xin Liu\footnotemark[2]
	\and Xiaojun Chen\thanks{Department of Applied { Mathematics}, The Hong Kong Polytechnic University, Hong Kong, China
		(\email{xiaojun.chen@polyu.edu.hk}).}}

\headers{nonsmooth optimization for training autoencoders}{WEI LIU, XIN LIU, AND XIAOJUN CHEN}

\ifpdf
\hypersetup{
	pdftitle={Linearly-constrained nonsmooth optimization for training autoencoders},
	pdfauthor={Wei Liu and Xin Liu and Xiaojun Chen}
}
\fi

\begin{document}
	
		\maketitle
		\begin{abstract}	
			A regularized minimization model with $l_1$-norm penalty (RP) is introduced for training the autoencoders
			that belong to a class of two-layer neural networks.
			We show that the RP can act as an exact penalty model
			which shares the same global minimizers, local minimizers, and
			{d(irectional)-}stationary points with
			the  original regularized model under mild conditions. We construct a bounded box region that contains
			at least one global minimizer of the RP, and
			propose a linearly constrained regularized minimization model with $l_1$-norm
			penalty (LRP) for training autoencoders. A smoothing proximal gradient algorithm
			is designed to solve the LRP. Convergence of the algorithm to a generalized d-stationary point of the RP and LRP is delivered. Comprehensive numerical experiments convincingly
			illustrate the efficiency as well as the robustness of the proposed
			algorithm.
		\end{abstract}
		
		
		
		\section{Introduction}\label{sec:review}
		
		{A deep neural network (DNN)~\cite{mcculloch1943logical} aims to solve a finite-sum minimization problem
			\begin{equation}\label{eq:dnn}
				\min_{{W_\ell,b_\ell,\ell=1\ldots,L}}\frac{1}{N}\sum_{n=1}^N\psi_n({\varphi_{n,L}}(W_1,\ldots,W_L,b_1,\ldots,b_L)).
			\end{equation}
			Here  $\varphi_{n,L}(W_1,\ldots,W_{L},b_1,\ldots,b_{L})=\sigma_{L}(W_{L}\sigma_{L-1}(\cdots\sigma_1(W_1x_n+b_1)+\cdots)+b_{L})$ denotes the outputs of the $L$-th hidden layer, and $\psi_n$ denotes the loss function measuring the output  $\varphi_{n,L}(W_1,\ldots,W_L,b_1,\ldots,b_L)$ and its corresponding true output for  $n=1,\ldots,N$, where $\{x_n\}_{n=1}^N$ is the data set, $W_{\ell} \in \R^{N_{\ell}\times N_{\ell-1}}$, $b_{\ell}\in \R^{N_{\ell}}$, $\sigma_{\ell}:\R^{N_{\ell}} \mapsto \R^{N_{\ell}}$ ($\ell=1,\ldots,L$) are
			the weight matrices, the bias vectors and the activation functions, respectively.

			A broad class of methods, based on stochastic gradient descent (SGD),
			are proposed to solve~\eqref{eq:dnn}, such as the vanilla SGD~\cite{cramir1946mathematical}, the Adadelta~\cite{zeiler2012adadelta}, and the Adam~\cite{kingma2014adam}.} In SGD methods, the gradient of the objective function is calculated by {the chain rule}, which is {applicable to} smooth activation functions, such as sigmoid, hyperbolic, and softmax functions~\cite{goodfellow2016deep}.
		{However if a nonsmooth activation function is used, such as the rectified linear unit (ReLU) or the leaky ReLU~\cite{maas2013rectifier},
			the subgradient of the objective function in~\eqref{eq:dnn} is difficult to calculate.
			At least the chain rule is no longer useful (see~\cite[Theorem 10.6]{clarke1990optimization}).
			As shown by recent studies, such nonsmooth activation functions have some advantages over the aforementioned smooth ones{\color{black}, as they can pursue the sparsity of the network~\cite{glorot2011deep}. The readers may refer to Glorot et al. \cite{glorot2011deep} and Jarrett et al. \cite{jarrett2009best} for the numerical comparisons between the DNN with smooth activation functions and those with nonsmooth ones. 
				Due to excellent numerical behavior, the ReLU activation function} has been widely used since 2010~\cite{agarap2018deep,dahl2013improving,nair2010rectified,
					sun2019optimization,xu2021artificial}.
			In practice, the SGD based approaches are still used to
			tackle the problem with nonsmooth activation functions.
			The exactness in calculating the subgradient of a nonsmooth function
			is usually neglected in SGD methods. {\color{black} Gradients} in
			a neighborhood are often used to approximate the one at
			a nonsmooth point. Certainly, such approximation may cause
			theoretical and numerical troubles in some cases.
			Hence, it is worthwhile to develop algorithms for
			solving	 problem~\eqref{eq:dnn} with nonsmooth activation functions
			and deal with the nonsmoothness appropriately.
		}

		In \cite{carreira2014distributed}, Carreira-Perpi\~n\'an and Wang {reformulate} problem (\ref{eq:dnn}) as the following constrained optimization problem {with $u_{n,0}=x_n$  for all $n=1,2,\ldots,N$,}
		\begin{equation}\label{eq:dnn2}
			\begin{aligned}
				\min_{{W_\ell,b_\ell, u_{n,\ell}}\atop{\ell=1,\ldots, L, n=1,\ldots,N}}\, & \, \frac{1}{N}\sum_{n=1}^N\psi_n(u_{n,L})\\
				{\rm s.t.}\quad \quad  \, & \, u_{n,\ell}=\sigma_{\ell}(W_{\ell}u_{n,\ell-1}+b_{\ell}),  \quad n=1,\ldots,N, \, \ell=1,\ldots,L,
			\end{aligned}
		\end{equation}
		and propose a method of auxiliary coordinates to solve (\ref{eq:dnn2}).
		Moreover, an alternating direction method of multipliers (ADMM)~\cite{taylor2016training}
		and a block coordinate descent method (BCD)~\cite{lau2018proximal}
		are proposed to solve the constrained model
		and its $l_2$-norm penalty problem, respectively.
		However, these methods are less efficient than SGD based approaches,
		and lack of theoretical guarantee {(see~\cite{zeng2019global})}.
		
		More recently, Cui et al. \cite{cui2020multicomposite} use an $l_1$-norm penalty method to replace the constraints in (\ref{eq:dnn2}) by adding $\sum^N_{n=1}\sum^{L}_{\ell=1}    \|u_{n,\ell}-\sigma_{\ell}(W_{\ell}u_{n,\ell-1}+b_{\ell})\|_1$ in the objective function.
		They provide an exact penalty analysis and establish the convergence of the sequence generated by their proposed algorithm to a directional stationary point, {\color{black} which will be defined in \eqref{eq:direcs}}.
		To the best of our knowledge, this is the first mathematically rigorous method
		for training deep neural networks with nonsmooth activation functions.
		However, some assumptions imposed in their theoretical analysis
		are restrictive for some applications. For instance, the ReLU does not satisfies
		the assumptions on activation functions  in \cite[Corollary 2.2]{cui2020multicomposite}.
		Moreover, the boundedness assumption on the sequences of iterates imposed for convergence analysis
		is not natural since the solution set of (\ref{eq:dnn2}) is unbounded.
		For example, suppose that $\{\bar{u}_{n,\ell},\bar{W}_{\ell},\bar{b}_{\ell}\}_{\ell=1,n=1}^{L,N}$ is a global minimizer of model (\ref{eq:dnn2}){, it} is easy to verify that $\{{\hat{u}_{n,\ell},\hat{W}_{\ell},
			\hat{b}_{\ell}}\}_{\ell=1,n=1}^{L,N}$
		is also a global minimizer, where
		\begin{eqnarray*}
			&\hat{u}_{n,1}=t\,\bar{u}_{n,1},
			\hat{W}_{1}=t\bar{W}_{1},
			\hat{b}_{1}=t\bar{b}_{1},
			\hat{u}_{n,2}=\bar{u}_{n,2},
			\hat{W}_{2}=\frac{1}{t}\bar{W}_{2},
			\hat{b}_{2}=\frac{1}{t}\bar{b}_{2},
			&\\
			&\hat{u}_{n,\ell}=\bar{u}_{n,\ell},\hat{W}_{\ell}=\bar{W}_{\ell},\hat{b}_{\ell}=\bar{b}_{\ell}, \quad \mbox{for\,}\,\,
			n=1,\ldots,N, \ell = 3,...,L&
		\end{eqnarray*}
		for any $t> 0$.
		Let $t$ tend to infinity, if $\bar{W}_1\neq 0$ or $\bar{b}_1 \neq 0$, then the norm of
		$\{{\hat{u}_{n,\ell},\hat{W}_{\ell},
			\hat{b}_{\ell}}\}_{\ell=1,n=1}^{L,N}$ tends to infinity.
		
		{To overcome the unboundness of the solution set of (\ref{eq:dnn2}), in this paper we consider the regularized model of problem (\ref{eq:dnn2}) in \cite{carreira2014distributed}, which adds the regularization term
			$\|W_\ell\|_F^2$ in the objective function. 	Motivated by the ideas of the exact $l_1$-norm penalty and directional stationarity in \cite{cui2020multicomposite},  we design a deterministic algorithm
			for training {\color{black} the autoencoder, a special two-layer network,} 
			using ReLU, with guaranteed convergence, and achieve competitive performances comparing with the
			SGD based approaches in solving large-scale problems.
			{\color{black}
				Our proposed model can be generalized to problem (\ref{eq:dnn2}) with certain regularizing term (see \eqref{eq:dnn21}--\eqref{eq:dnn23} in the conclusion part.). In fact, 
				the number of layers does not affect the validity of our
				theoretical analysis on the model.
				The reason we focus on the autoencoder in this paper is that a large number of layers does increase lots of
				tedious notations as well as rapidly increasing number of variables which requires further development on the algorithm to maintain the  comparability with existing approaches. Such development is out of the main scope of this paper.
			}
			

		}
		
		\subsection{Regularized Autoencoders}
		
		Training an autoencoder using ReLU as the activation function
		can be formulated as the following
		nonsmooth nonconvex finite-sum {minimization} problem.
		\begin{equation}\label{eq:auto}
			\min_{W,b}\frac{1}{N}\sum_{n=1}^N\|\sigma(W\zz \sigma(Wx_n+b_1)+b_2)-x_n\|_2^2,
		\end{equation}
		where $\{x_n\in\R^{N_0}\}_{n=1}^N$ is the given data,
		$W\in\R^{N_1\times N_0}$ is the weight matrix,
		$b_1\in\R^{N_1}$ and $b_2\in\R^{N_0}$ are the bias vectors. {\color{black} For convenience, we use
			$X=(x_1,x_2,\ldots,x_N)\in\R^{N_0\times N}$ to denote the data matrix,
			and denote
			$b={(b_1\zz,b_2\zz )\zz} \in\R^{N_1+N_0}$ as the combination
			of two bias vectors.}
		{\color{black}
			Here, we select $W\zz$ as the weight matrix of the second layer,
			which is the transpose of that of the first layer. The
			consequent model \eqref{eq:auto} is called the autoencoder with
			tied weight which has been widely used in practice (see in \cite{goodfellow2016deep, hinton2006reducing}). However, there
			exists autoencoder without tied weight, namely, the weight matrices
			of the two layers take $W_1$ and $W_2$, respectively (see in \cite{ng2011sparse}). Nevertheless, 
			Li and Nguyen \cite{li2019auto} have shown that by using the tied weight, the training speed is increasing and the numerical performance is comparable than that without tied weight.
			Then, it becomes uncommon to consider the general case.
			On the other hand, our new model, algorithm and theorectical analysis can be generalized to the autoencoder without tied weight easily.}
		
		\begin{figure}[htbp!]
			\centering
			\setcounter{subfigure}{0}
			\subfloat{\includegraphics[width=120mm]{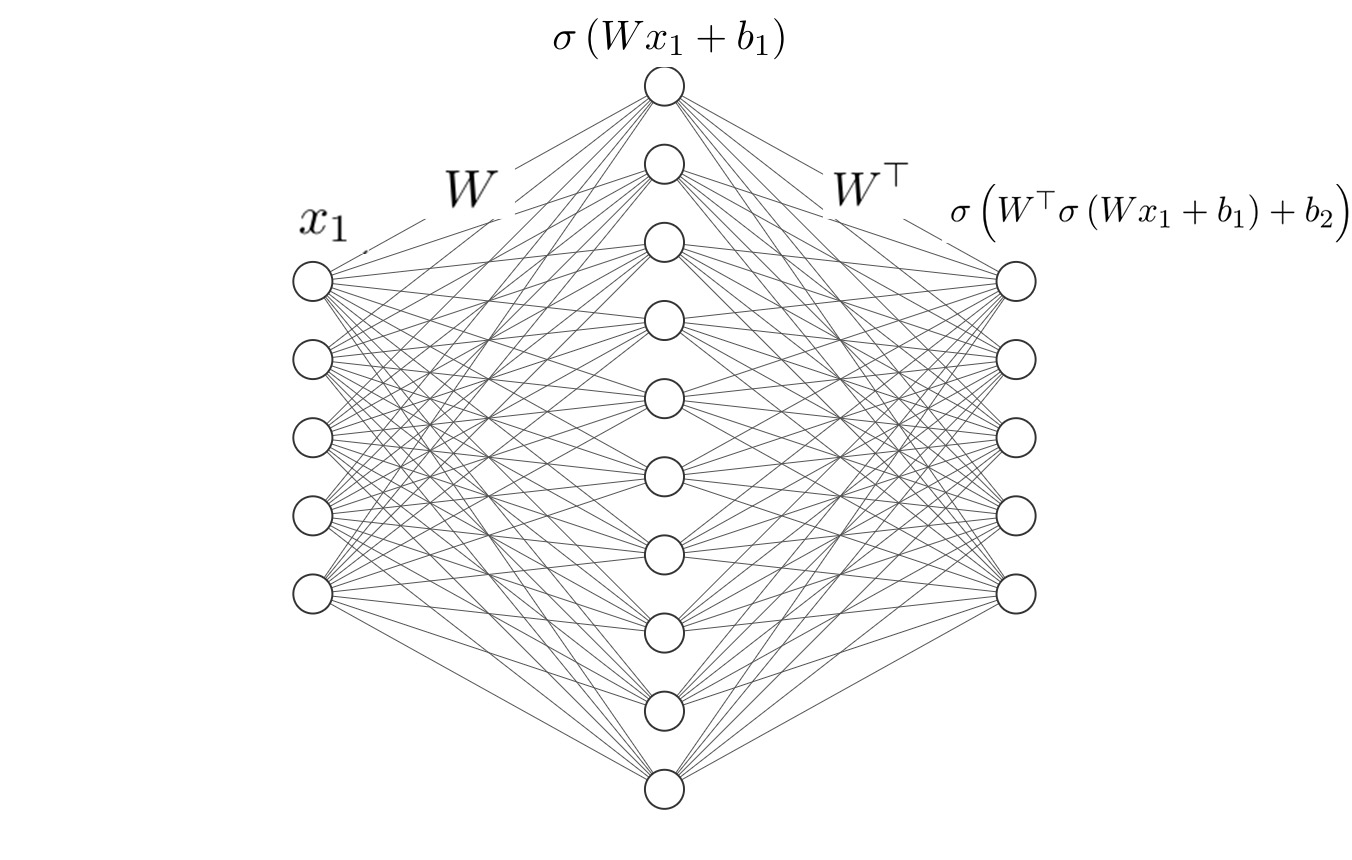}}
			\caption{Illustration of the network of an autoencoder}\label{fig:NN2}	\vspace{-4mm}
		\end{figure}
		
		In this paper, we focus on the ReLU, i.e. $\sigma(y)= y_{+} := \max\{0,y\}$.
		An autoencoder aims to learn a {prediction function} $\sigma(W\zz\sigma(Wx_n+b_1)+b_2)$ for
		the given data $\{x_n\in\R^{N_0}\}_{n=1}^N$ without any label, since $\{x_n\}_{n=1}^N$ is also regarded as the true
		value of the output layer. Hence, the autoencoder is classified
		as an unsupervised learning tool. 
	In recent years, autoencoders have been widely used in denoising, dimensionality reduction, and feature learning (e.g.,~\cite{bourlard1988auto,le1987modeles, wen2017new}). Besides, {autoencoders can
		be used as a preprocessing tool before training a DNN (e.g.,~\cite{ hinton2006reducing,pasa2014pre}).}

	{
		In practice, directly solving~\eqref{eq:auto} may lead to
		overfitting or ill-condition.
		To conquer these issues, the authors of~\cite{goodfellow2016deep} {introduce} two regularization terms to guarantee the model's robustness.
		The first class of regularizers {is}
		the $l_F$-norm term $\|W\|_F^2$, {\color{black} called weight decay,}
		which can effectively avoid the overfitting phenomenon~\cite{krogh1992simple}.
		The second class is the $l_1$-norm that can pursue the sparsity \cite{goodfellow2016deep, ng2011sparse}.}
	In this paper, we use the $l_F$-norm for the weight matrix and the $l_1$-norm for the auxiliary vectors.   To present our optimization model in $R^{N_2}$ with $N_2=N_0N_1+N_1+N_0+N_1N$,  we introduce a vector variable
	\begin{equation}\label{eq:defz}
		z=(\mathrm{vec}(W)\zz, b\zz,\mathrm{vec}(V)\zz)\zz\in\R^{N_2},
	\end{equation}
	where  $V=(v_1,v_2,\ldots,v_N)\in\R^{N_1\times N}$ is an auxiliary variable with $v_n=(Wx_n+b_1)_{+}$ for all $n=1,\ldots,N$, and $\mathrm{vec}(Y) \in \R^{lm}$ denotes the columnwise vectorization of the matrix $Y\in\R^{l\times m}$.
	Let
	$$ \mathcal{F}(z)=  \frac{1}{N}\sum_{n=1}^N\|(W\zz v_n+b_2)_{+}-x_n\|_2^2 \quad {\rm and} \quad   \mathcal{R}(z)=\lb_1\sum_{n=1}^Ne\zz v_n+\lb_2\|W\|_F^2$$
	denote the fidelity term and regularization term, respectively, where $e=(1,\ldots,1)\zz \in\R^{N_1}$ and $\lb_1,\lb_2>0$.
	We consider the following Regularized (R) minimization  model
	for the autoencoders
	\begin{equation}
		\tag{R}\label{eq:reguauto2}
		\begin{aligned}
			\min_{ z}\,\,&\mathcal{F}(z)+\mathcal{R}(z)\\
			\st 
			\,\,& z \in \Omega_1:= \left\{z : v_n= (Wx_n+b_1)_{+}, \, n=1,\ldots,N \right\}.
		\end{aligned}
	\end{equation}
	{\color{black} We would like to mention that 
		the equivalent form of problem \eqref{eq:reguauto2}, namely 
		\eqref{eq:auto} with regularizer $\mathcal{R}$, 
		has been widely used in autoencoders (see in \cite{goodfellow2016deep, wen2017new}). }
	
	\subsection{Our Focuses and Motivation}\label{sec:focus}
	
	{{The feasible set $\Omega_1$ of problem \eqref{eq:reguauto2} is nonconvex and the standard constraint qualifications may fail due to the nonsmooth equality constraints in~\eqref{eq:reguauto2}.}
		Hence, we 
		introduce the following
		Regularized minimization model with $l_1$-norm Penalty (RP) for
		the autoencoders.}
	\begin{equation}\tag{RP}\label{eq:reguauto4}
		\begin{aligned}
			\min_{ z}\,\,&\mathcal{O}(z):=\mathcal{F}(z)+\mathcal{R}(z)+\mathcal{P}(z)\\
			\st 
			\,\,& z\in \Omega_2: = \left\{z : v_n\ge (Wx_n+b_1)_{+}, \, n=1,\ldots,N \right\},
		\end{aligned}
	\end{equation}
	where $
	\mathcal{P}(z):=\beta\sum_{n=1}^Ne\zz \left(v_n-(W x_n+b_{1})_{+}\right)
	$ is the penalty term.

	Compared with the $l_1$-norm penalty term $\sum_{n=1}^N\left\|v_n-(Wx_n+b_{1})_{+}\right\|_1$ proposed in  \cite{cui2020multicomposite}, the subdifferential of $\mathcal{P}(z)$
	enjoys an explicit expression.
	In addition, the feasible set $\Omega_2$ of~\eqref{eq:reguauto4} is convex and the slater-type constraints qualification 
	holds~\cite[Section 6.3, Proposition 6.3.1]{clarke1990optimization}.
	However, the solution set of (RP) may be unbounded as that of the model in \cite{cui2020multicomposite}. To overcome the unboundness and ensure the sequence generated by the algorithm is bounded, we introduce
	a convex set
	$$
	\Omega_{3}:=\left\{z: \|b\|_{\infty} \leq \alpha\right\},
	$$
	where
	\begin{equation}\label{def:alpha}
		\alpha=\max\left\{\frac{\theta}{\lb_1}+\sqrt{\frac{N_1N_0\theta}{\lb_2}}\|X\|_1,\frac{\theta\sqrt{N_1N_0\theta}}{\lb_1\sqrt{\lb_2}}+\sqrt{N\theta}+\|X\|_1\right\},
		\,\,\theta> \frac{1}{N}\|X\|_F^2.
	\end{equation}
	
	We will show that~\eqref{eq:reguauto4}  has a global solution in $\Omega_3$.
	Hence, it suffices to
	solve~\eqref{eq:reguauto4}  restricted to $\Omega_3$. Note that $v_n\geq (Wx_n+b_1)_{+}$ can be represented by $v_n\geq Wx_n+b_1$ and $v_n\geq 0$.  Let  $\nu=2(NN_1+N_0+N_1)$,
	
	{		\color{black}
		\begin{align*}
			&A=
			\left[
			\begin{array}{ccc}
				\bovermat{W:N_0N_1}{X\zz\otimes I_{N_1}} & \bovermat{b:N_1+N_0}{e_{N}\otimes [I_{N_1} \,\,\,\textbf{0}]}  & \bovermat{V:NN_1}{-I_{N_1N}}\\
				\textbf{0} & \textbf{0}& -I_{N_{1} N}\\
				\textbf{0} & I_{N_{1}+N_0} & \textbf{0} \\
				\textbf{0} & -I_{N_{1}+N_0}  & \textbf{0}
			\end{array}
			\right]\in\R^{\nu\times N_2},  \quad c=
			\left[
			\begin{array}{c}
				\textbf{0}\\
				\textbf{0} \\
				\alpha e_{N_1+N_0}\\
				\alpha e_{N_1+N_0}
			\end{array}
			\right]\in\R^{\nu},
	\end{align*}}
	where $\otimes$ represents the Kronecker product, {\color{black} $e_{N}\in\R^{N}$, $e_{N_1+N_0}\in\R^{N_1+N_0}$ denote the vector whose elements are all one.}
	We consider the following Linearly constrained Regularized minimization
	model with $l_1$-norm Penalty (\ref{eq:reguauto7})
	\begin{equation}\tag{LRP}\label{eq:reguauto7}
		\begin{aligned}
			\min_{ z}\,\,&\,\, \mathcal{O}( z)\\
			\st\,\, & \,\,    z\in \mathcal{Z}:=\Omega_2\cap\Omega_3=\{z:A z\leq c\}.
		\end{aligned}
	\end{equation}
	
	\subsection{Contribution}
	{We propose a partial penalty model~\eqref{eq:reguauto4}
		and establish the equivalence between  the models~\eqref{eq:reguauto4}  and~\eqref{eq:reguauto2}  regarding  global minimizers, local minimizers, and directional stationary points under some mild conditions.
		Moreover, we show that the solution set of \eqref{eq:reguauto7} is bounded and contains at least one of global minimizers of~\eqref{eq:reguauto4}, and provide conditions such that \eqref{eq:reguauto7} and  \eqref{eq:reguauto4} have the same local minimizers and directional stationary points in ${\cal Z}$. 
		
		We propose a smoothing proximal gradient algorithm for solving~\eqref{eq:reguauto7}, whose subproblem at each iteration is a  structured strongly convex quadratic program. We develop a splitting algorithm for solving the subproblem by using the special structure, which is faster than {the} ``quadprog'' \cite{turlach2019package} and  {the} ``CVX'' \cite{grant2014cvx}.
		We prove that the sequence generated
		by our algorithm converges to a generalized directional stationary point of~\eqref{eq:reguauto7} without assuming the boundness of
		sequences or existence of accumulation points.
		
		The numerical experiments demonstrate
		that our algorithm, equipped {with} adaptively selected stepsize and smoothing parameters, outperforms the popular SGD methods (e.g., Adam, Adadelda, and vanilla SGD) in acquiring better and more robust solutions
		to a group of randomly  generated data sets and one real data set for autoencoders.
		More specifically, compared with SGD methods, our algorithm achieves lower training error and objective function values, and obtains sparser solutions to testing problems.
	}
	
	\subsection{Notations and Organizations}
	Let $\mathcal{B}_{\epsilon}(y)$ be the closed Euclidean ball in $\R^m$ centered at $y$ and radius $\epsilon$. 
	The $m\times m$ identity matrix is denoted by $I_{m}$. 
	Given a nonempty closed set \(\Omega\) 
	and a point \(y^{*}\), we use $\dist(y^{*},\Omega)=\inf_{y\in\Omega}||y-y^*||_2$ to denote the distance from $y^*$ to $\Omega$. We use $\mathrm{co}(\Omega)$ to represent convex hull of $\Omega$.
	
	The rest of this paper is organized as follows. In~Section \ref{sec:pre}, we give theoretical results for the relationship among the three models~\eqref{eq:reguauto2},~\eqref{eq:reguauto4} and~\eqref{eq:reguauto7}.
	In~Section \ref{sec:smooth}, we propose a smoothing proximal gradient algorithm for solving \eqref{eq:reguauto7} and present the global convergence
	of the algorithm. In~Section \ref{sec:num},
	{we illustrate the performance of our proposed algorithm
		through comprehensive numerical experiments.}
	Concluding remarks are given in the last section.
		
	\section{Model Analysis}\label{sec:pre}
	
	In this section,
	we aim to theoretically investigate the relationship among problems  \eqref{eq:reguauto2},~\eqref{eq:reguauto4}  and~\eqref{eq:reguauto7} for autoencoders.
	
	\subsection{Preliminaries}\label{sec:pre1}
	In this subsection, we present some preliminary definitions.
	Let $\mathrm{Proj}_{\Omega}(y^*)=\arg \min \left\{\|y-y^*\|_{2}: y \in \Omega\right\}$ denote the orthogonal projection of a vector $y^* \in \R^m$ onto a convex set $\Omega\subseteq \R^m$.
	
	The Clarke subdifferential \cite[Section 1.2]{clarke1990optimization} of a locally Lipschitz continuous function $f: \R^m \to \R$ at $y^{*}$ is defined by
	$
	\partial f\left(y^{*}\right)=\operatorname{co}\left\{\lim _{y \rightarrow y^{*}} \nabla f(y): f \text { is smooth at } y\right\}$.
	
	We use $f^{\prime}\left(y ; d \right)$ to denote the directional derivative of a directional differentiable function $f$ at $y $ along the direction $d$, i.e.,
	\begin{equation}
		\label{eq:dsta}
		f^{\prime}\left(y ; d \right)=\lim _{ t \downarrow 0} \frac{f(y+t d)-f(y)}{t}.
	\end{equation}
%

A function $f$ is said to be regular~\cite[Definition 2.3.4]{clarke1990optimization} at $\bar{ y}\in\R^m$ provided that if
for all $d$, the directional derivative $f^{\prime}(\bar{ y} ; d)$ exists, and $$f^{\prime}(\bar{ y} ; d)=f^{\circ}(\bar{ y} ; d),$$
where $f^{\circ}(\bar{ y} ; d):=\limsup _{{y \rightarrow \bar{ y}}\atop{t \downarrow 0}} \frac{f(y+t d)-f(y)}{t}$ is the generalized directional derivative at $\bar{y}$ along the direction $d$~\cite{clarke1990optimization}. 

It is known that if $f$ is piecewise smooth and  Lipschitz continuous in a neighborhood of $y$, then $f$ is semismooth and directional differentiable at $y$ \cite{mifflin}.
The objective functions of \eqref{eq:reguauto2},~\eqref{eq:reguauto4}  and~\eqref{eq:reguauto7}
are locally Lipschitz and piecewise smooth. Hence they are semismooth functions and
directional differentiable.


%
%

Let $\mathcal{T}_{\Omega}(\bar{y})=\{d:d=\lim _{y\in\Omega, y \rightarrow \bar{y},\tau\downarrow 0} \frac{y-\bar{y}}{\tau}\}$ be the tangent cone of a set $\Omega$ at $\bar{y}$.

{\color{black}
	\begin{definition}\label{eq:direcs}
		We call $\bar{z}\in\Omega_1$, $\bar{z}\in\Omega_2$, $\bar{z}\in\mathcal{Z}$  a d(irectional)-stationary point \cite{cui2020multicomposite} of problems~\eqref{eq:reguauto2}, \eqref{eq:reguauto4}  and \eqref{eq:reguauto7}, respectively, if
		\begin{equation}\label{dR}
			\mathcal{F}^{\prime}\left(\bar{z} ; d \right)+ \nabla \mathcal{R}(\bar{z})\zz d \geq 0, \quad \forall d\in\mathcal{T}_{\Omega_1}(\bar{z}),
		\end{equation}
		\begin{equation}\label{dRP}
			\mathcal{O}^{\prime}\left(\bar{z} ; d \right)\geq 0, \quad \forall d\in\mathcal{T}_{\Omega_2}(\bar{z}),
		\end{equation}
		\begin{equation}\label{dLRP}
			\mathcal{O}^{\prime}\left(\bar{z} ; d \right)\geq 0, \quad \forall d\in\mathcal{T}_{\mathcal {Z}}(\bar{z}).
		\end{equation}
\end{definition}}

We call $\bar{z}\in\Omega_1$, $\bar{z}\in\Omega_2$, $\bar{z}\in\mathcal{Z}$  a generalized d(irectional)-stationary point of problems~\eqref{eq:reguauto2}, \eqref{eq:reguauto4}  and \eqref{eq:reguauto7}, respectively, if \eqref{dR}--\eqref{dLRP} hold with $\mathcal{F}^{\circ}(\bar{z} ; d )$ and $\mathcal{O}^{\circ}(\bar{z} ; d )$ instead of $\mathcal{F}^{\prime}(\bar{z} ; d )$ and $\mathcal{O}^{\prime}(\bar{z} ; d )$.

We call $\bar{z}\in\mathcal{Z}$ a generalized
KKT point of \eqref{eq:reguauto7} if there exists a nonnegative vector $\bar{\gamma} \in \R^\nu$ such that
\begin{equation}\label{eq:kktxi}
	\begin{aligned}
		0 \in \partial {\cal O}(\bar{z})+A^{\top} \bar{\gamma},\,\quad
		\bar{\gamma}^{\top}\left(A\bar{z}-c\right)=0,\,\quad A \bar{z}-c \leq 0.
	\end{aligned}
\end{equation}

\subsection{Global and Local Solutions}
Let $\mathcal{S}^*$,  $\mathcal{S}$ and ${\cal Z}^*$ be  the global solution sets of \eqref{eq:reguauto2},~\eqref{eq:reguauto4}  and~\eqref{eq:reguauto7}, respectively.
In this subsection, we
prove $\mathcal{S}^*$,  $\mathcal{S}$ and ${\cal Z}^*$  are not empty, and
${\cal Z}^*\subset \mathcal{S} =\mathcal{S}^*$.

We define a level set of problem~\eqref{eq:reguauto4} using $\theta$ defined in (\ref{def:alpha}) as follows
$$\Omega_\theta=\{ z \in \Omega_2 :   \mathcal{O}(z)\leq \theta\}.$$
Obviously, $0\in \Omega_\theta$ since ${\cal O}(0)=\frac{1}{N}\|X\|_F^2<\theta$ and $0\in \Omega_2$.
\begin{theorem}\label{thm:nonem}   For any $z\in\Omega_\theta$,
	the following statements hold.
	\begin{description}
		\item{(a)} $\|W\|_F^2\le \frac{\theta}{\lambda_2},
		\|V\|_1\le \frac{\theta}{\lambda_1}$ and $\|b_{+}\|_{\infty}\leq\alpha$.
		\item{(b)}	
		$\bar{z}=\mathrm{Proj}_{\Omega_3}(z)\in\mathcal{Z}$ and $\mathcal{O}(\bar{ z})=\mathcal{O}(z).$
	\end{description}
	Moreover, the solution set ${\cal Z}^*$ of~\eqref{eq:reguauto7} is not empty and bounded, and  ${\cal Z}^*
	\subset {\cal S}$.
\end{theorem}
\begin{proof}
	(a) 	The first two inequalities are from $V\ge 0$, $\mathcal{O}(z)\leq \theta $, $\mathcal{F}(z)\geq 0$,
	$\mathcal{P}(z)\geq 0$ and $\mathcal{R}(z)\geq 0$, which imply
	\begin{equation}
		\lb_1\sum_{n=1}^Ne\zz v_n\leq \theta \,\,\, {\rm and} \,\,\,  \lb_2\|W\|_F^2\leq \theta.\label{eq:set1-1}
	\end{equation}
	Now we prove $\|b_{+}\|_{\infty}\leq\alpha$.	From the Cauchy inequality and~\eqref{eq:set1-1},
	we have
	\begin{equation}\label{eq:equv12}
		\begin{aligned} \sum_{j=1}^{N_1}\sum_{s=1}^{N_0}|W_{j,s}|\leq \sqrt{N_1N_0}\|W\|_F\leq\sqrt{\frac{N_1N_0\theta}{\lb_2}}.
		\end{aligned}
	\end{equation}
	For $n=1,\ldots,N$, combining~\eqref{eq:set1-1} with $ z\in\Omega_2$, we obtain that
	\begin{equation}\label{eq:set1}
		\begin{aligned}
			\frac{\theta}{\lb_1}\geq e\zz v_n\geq e\zz (Wx_n+b_1)_{+}\geq (W_{j,\cdot}x_n+b_{1,j})_{+}\geq W_{j,\cdot}x_n+b_{1,j}
		\end{aligned}
	\end{equation}
	for all $j=1,\ldots,N_1$. On the other hand,~\eqref{eq:equv12} yields $\|W_{j,\cdot}\|_{\infty}\leq \sqrt{\frac{N_1N_0\theta}{\lb_2}}$, which implies
	\begin{equation}\label{eq:set2}
		|W_{j,\cdot}x_n|\leq \sqrt{\frac{N_1N_0\theta}{\lb_2}}\|X\|_1.
	\end{equation}
	Together with~\eqref{eq:set1}, we can conclude that $b_{1,j}$ satisfies
	\begin{equation}\label{eq:setb1}
		b_{1,j}\leq \frac{\theta}{\lb_1}+\sqrt{\frac{N_1N_0\theta}{\lb_2}}\|X\|_1, \quad \forall j=1,\ldots,N_1.
	\end{equation}
	
	From ${\cal F}(z) \le \theta $, we have
	\begin{equation}\label{eq:set3}
		\sqrt{N\theta}\geq (W_{\cdot,j}\zz v_n+b_{2,j})_{+}-X_{j,n}\geq W_{\cdot,j}\zz v_n+b_{2,j}-X_{j,n}
	\end{equation}
	for all $n=1,\ldots,N$ and $j=1,\ldots,N_0$.
	From $\|v_n\|_1\leq \frac{\theta}{\lb_1}$ and $\|W_{\cdot,j}\|_1\leq \sqrt{\frac{N_1N_0\theta}{\lb_2}}$, we find
	\begin{equation}\label{eq:set4}
		|W_{\cdot,j}\zz v_n|\leq \frac{\theta\sqrt{N_1N_0\theta}}{\lb_1\sqrt{\lb_2}}.
	\end{equation}
	Together with~\eqref{eq:set3}, we obtain that
	\begin{equation}\label{eq:setb2}
		b_{2,j}\leq \frac{\theta\sqrt{N_1N_0\theta}}{\lb_1\sqrt{\lb_2}}+\sqrt{N\theta}+\|X\|_1,
		\quad \forall j=1,\ldots, N_0.
	\end{equation}
	Combining~\eqref{eq:setb1} and~\eqref{eq:setb2}, we finally arrive at the assertion that $\|b_{+}\|_{\infty}\leq\alpha$.
	
	$($b$)$  Let $\bar{z}=(\mathrm{vec}(\bar{W})\zz, \bar{b}\zz,\mathrm{vec}(\bar{V})\zz)\zz$ with
	$\bar{W}=W$, $\bar{V}=V$ and
	\begin{equation}\label{eq:setcon}
		\bar{b}_{1,j_1}=\left\{
		\begin{aligned}
			&b_{1,j_1} \,\,\,\text{if } {b}_{1,j_1} \geq -\alpha,\\
			&-\alpha \,\,\,\text{otherwise,}
		\end{aligned}\right.\quad \quad
		\bar{b}_{2,j_2}=\left\{
		\begin{aligned}
			&b_{2,j_2} \,\,\,\text{if } {b}_{2,j_2}\geq -\alpha,\\
			&-\alpha\,\,\,\text{otherwise}
		\end{aligned}\right.
	\end{equation}
	for all $j_1=1,\ldots,N_1$, $j_2=1,\ldots,N_0$. By part (a), we have $\|\bar{ b}\|\le \alpha$. Hence
	$\bar{z}\in \Omega_3.$
	
	By~\eqref{eq:set2} and~\eqref{eq:set4}, we have
	$${b}_{1,j_1}+{W}_{j_1,\cdot}x_n\leq \bar{b}_{1,j_1} +\bar{ W}_{j_1,\cdot}x_n\leq -\alpha+ \sqrt{\frac{N_1N_0\theta}{\lb_2}}\|X\|_1\leq 0,
	\,\, {\rm if}\,\, 	\bar{b}_{1,j_1}=-\alpha,$$
	$${b}_{2,j_2}+{W}_{\cdot,j_2}\zz v_n \leq \bar{b}_{2,j_2}+ \bar{W}_{\cdot,j_2}\zz \bar{v}_n\leq -\alpha+ \frac{\theta\sqrt{N_1N_0\theta}}{\lb_1\sqrt{\lb_2}}\leq 0, \,\,\,\quad\quad {\rm if} \,\,
	\bar{b}_{2,j_2}=-\alpha,$$
	which together with (\ref{eq:setcon}) implies that for all $n=1,\ldots,N$, it holds
	\begin{equation}\label{eq:setb1b2}	
		(\bar{W} x_n+\bar{b}_{1})_{+}=(Wx_n+b_1)_{+}\,\,  {\rm and} \,\, (\bar{W}\zz \bar{v}_n+\bar{b}_{2})_{+}=(W\zz v_n+b_2)_{+}.
	\end{equation}
	Combining with $\bar{W}=W$ and $\bar{V}=V$, we have
	$\mathcal{O}(\bar{ z})=\mathcal{O}(z).$
	Moreover~\eqref{eq:setb1b2}, $\bar{W}=W$ and $\bar{V}=V$ yield $\bar{ z}\in \Omega_2$.
	Hence by the definition of $\bar{z}$,  $\bar{ z}=\mathrm{Proj}_{\Omega_3}(z)\in \Omega_2\cap  \Omega_3={\cal Z}$.
	
	Now we prove the last statement. By parts (a) and  (b), the set ${\cal Z}\cap\Omega_\theta$
	is a bounded closed set. Hence, there exists $ z^*\in {\cal Z}^*$  such that
	$
	\mathcal{O}(z^*)=\min_{ z\in\mathcal{Z}}\mathcal{O}(z) \le \theta.
	$
	Assume on contradiction that $z^*\in {\cal Z}^*$, but $z^*\not\in \mathcal{S}$. Then there exists $\widetilde{ z}\in\Omega_2$ such that $\mathcal{O}(\widetilde{ z})<\mathcal{O}(z^*)\leq\theta$.
	As we have proved in $($b$)$, $\bar{ z}=$Proj$_{\Omega_3}(\widetilde{ z})\in \mathcal{Z}$ and
	$
	\mathcal{O}(\bar{ z})=\mathcal{O}(\widetilde{ z}),
	$
	which implies $\mathcal{O}(\bar{ z})<\mathcal{O}(z^*)$. This is a contradiction.
	Hence ${\cal Z}^*\subset {\cal S}$.
\end{proof}

The following theorem shows that~\eqref{eq:reguauto4} is an exact penalty formulation of~\eqref{eq:reguauto2} regarding global minimizers if the penalty parameter $\beta$ in ${\cal P}$ is larger than a computable number.

\begin{theorem}
	The following statements hold.
	\begin{description}
		\item{(a)} The functions  $\mathcal{F}$ and $\mathcal{R}$ are Lipschitz continuous over $\Omega_{\theta}$.
		\item{(b)} Let  $L_{\mathcal{F}}$ and $L_{\mathcal{R}}$ be Lipschitz modulus of $\mathcal{F}$ and $\mathcal{R}$ over $\Omega_{\theta}$, respectively. Suppose $\beta>L_{\mathcal{F}}$+$L_{\mathcal{R}}$.
		If \(\bar{ z}\in\Omega_{\theta}\) is a global minimizer  of~\eqref{eq:reguauto2},
		then \(\bar{ z}\) is also a global minimizer of~\eqref{eq:reguauto4}.
		\item{(c)} Let $\delta:=3 \theta+\frac{2 N\theta^{3}}{ \lb_1^2\lambda_{2}}$ and $\Omega_\delta=\{ z \in \Omega_2 :   \mathcal{O}(z)\leq \delta\}$. Let $L_{\mathcal{F}}$ and $L_{\mathcal{R}}$ be Lipschitz modulus of $\mathcal{F}$ and $\mathcal{R}$ over $\Omega_{\delta}$ respectively. Suppose $\beta>L_{\mathcal{F}}$+$L_{\mathcal{R}}$. 	If \(\bar{ z}\in\Omega_{\theta}\) is a global minimizer of~\eqref{eq:reguauto4},
		then \(\bar{ z}\) is also a global minimizer of~\eqref{eq:reguauto2}.
	\end{description}
\end{theorem}
\begin{proof}
	(a) From Theorem \ref{thm:nonem} (a), it is clear that $\mathcal{R}$ is Lipschitz continuous over  $\Omega_{\theta}$.  From  Theorem \ref{thm:nonem} (a)-(b), the set ${\cal Z}\cap \Omega_\theta$ is bounded.  Suppose that $L_{\mathcal{F}}$ is the Lipschitz constant of $\mathcal{F}$ over ${\cal Z}\cap \Omega_{\theta}$. Let $z_1, z_2\in\Omega_{\theta}$. It follows~Theorem \ref{thm:nonem} (b) that {$\mathcal{F}(\mathrm{Proj}_{\Omega_3}(z_1))=\mathcal{F}(z_1)$, $\mathcal{F}(\mathrm{Proj}_{\Omega_3}(z_2))=\mathcal{F}(z_2)$, and $\mathrm{Proj}_{\Omega_3}(z_1), \mathrm{Proj}_{\Omega_3}(z_2)\in{\cal Z}$.  From ${\cal Z}=\Omega_2\cap \Omega_3$,  we have $\mathrm{Proj}_{\Omega_3}(z_1), \mathrm{Proj}_{\Omega_3}(z_2)\in\Omega_{\theta}$. Hence, it holds that
		\begin{eqnarray*}
			\|\mathcal{F}(z_1)-\mathcal{F}(z_2)\|_2&=& \|\mathcal{F}(\mathrm{Proj}_{\Omega_3}(z_1))-\mathcal{F}(\mathrm{Proj}_{\Omega_3}(z_2))\|_2\\
			&\leq& L_{\mathcal{F}}\|\mathrm{Proj}_{\Omega_3}(z_1)-\mathrm{Proj}_{\Omega_3}(z_2)\|_2\\
			&\leq& L_{\mathcal{F}}\|z_1-z_2\|_2,
	\end{eqnarray*}}
	where the last inequality is from that $\Omega_3$ is a convex set and the projection is Lipschitz continuous with Lipschitz constant  1. Hence we derive that $\mathcal{F}$ is Lipschitz continuous over $\Omega_{\theta}$ with the Lipschitz constant $L_{\mathcal{F}}$.
	
	(b)
	We first prove that
	$
	\beta \dist(z,\Omega_1)\leq {\cal P}(z)$
	for all $z\in\Omega_2$.
	
	For  $z\in \Omega_2$, 		let $\widetilde{ z}=(\mathrm{vec}(W)\zz, b\zz,\mathrm{vec}(\widetilde{V})\zz)\zz$ with
	$\widetilde{v}_n=({W}x_n+{b}_1)_{+}$ for all $n=1,2,\ldots,N$.
	Then, we have $\widetilde{ z}\in \Omega_1$, $\widetilde{v}_n\leq {v}_n$, and
	$$ \dist({ z},\Omega_1)\leq \|z-\widetilde{ z}\|_2\leq \|{\rm vec}(V-\widetilde{ V})\|_2\leq\sum_{n=1}^N\|{v}_n-({W}x_n+{b}_1)_{+}\|_1=\frac{1}{\beta} {\cal P}(z),$$
	where the last inequality comes from the definition of $\widetilde{ v}_n$ and $\|\cdot\|_2\le \|\cdot\|_1$,  and the equality is from $z\in \Omega_2$.
	
	Since $\beta > L_{\mathcal{F}}+L_{\mathcal{R}}$, ${\cal P}(z)=0$ for all $z\in \Omega_1$,
	$\Omega_1\subset \Omega_2$, and $\Omega_\theta\subset \Omega_2$, we have
	\begin{eqnarray*}
		\min_{z\in \Omega_\theta} \mathcal{F}(z)+\mathcal{R}(z) + \mathcal{P}(z)
		&\ge&\min_{z\in \Omega_1\cap\Omega_\theta }\mathcal{F}(z)+\mathcal{R}(z) \\
		&=&\min_{z\in \Omega_1\cap\Omega_\theta} \mathcal{F}(z)+\mathcal{R}(z) + \mathcal{P}(z)\\
		&\ge&\min_{z\in \Omega_\theta} \mathcal{F}(z)+\mathcal{R}(z) + \mathcal{P}(z).
	\end{eqnarray*}
	Hence we obtain the statement (b).
	
	(c) Let $\bar{ z}\in\mathcal{S}$ and $\zeta_2(V)=\frac{1}{N}\sum_{n=1}^N\left\|\left(\bar{W}\zz v_n+\bar{b}_2\right)_{+}-x_n\right\|_2^2+\lb_1\sum_{n=1}^Ne\zz v_n.
	$ By the definition of $\bar{V}$,  $\bar{V}$ is a global minimizer of
	$$
	\min_{V\in \Omega_4}\zeta_1(V):=\zeta_2(V)+\beta\sum_{n=1}^N\left\|v_n-(\bar{W}x_n+\bar{b}_1)_{+}\right\|_1,
	$$
	where $\Omega_4=\left\{V:v_n\geq(\bar{W}x_n+\bar{b}_1)_{+}, n=1,\ldots,N, \zeta_1(V)\leq \delta\right\}$.

	It follows from $\mathcal{O}(\bar{ z})\leq \mathcal{O}(0)< \theta$ that
	$
	\sum_{n=1}^N\left\|(\bar{ W}\zz \bar{v}_n+\bar{ b}_2)_{+}-x_n\right\|_2^2\leq N\theta,\,\,
	\lb_2\|\bar{W}\|_F^2\leq \theta,
	\lb_1\sum_{n=1}^Ne\zz \bar{v}_n \leq \theta$, and $
	\beta\sum_{n=1}^Ne\zz r_n\leq \theta,
	$ where $r_n=\bar{v}_n-(\bar{W} x_n+\bar{b}_{1})_{+}\geq 0$ for $n=1,2,\ldots,N$.
	Let
	$$
	\Omega_5=\left\{V:v_n=(\bar{W}x_n+\bar{b}_1)_{+},\,  n=1,\ldots,N,\, \zeta_2(V)\leq \delta\right\}
	$$
	and
	$\hat{v}_n=(\bar{W}x_n+\bar{b}_1)_{+}$ for all $n=1,\ldots,N$.
	We show that $\hat{V}\in\Omega_5$ as follows.
	
	\begin{align*}
		\zeta_2(\hat{V})&= \frac{1}{N}\sum_{n=1}^N\sum_{j=1}^{N_0}\left|\left(\bar{W}_{\cdot,j}\zz \left(\bar{v}_n-r_n\right)+\bar{b}_{2,j}\right)_{+}-X_{j,n}\right|^2+\lb_1\sum_{n=1}^Ne\zz \left(\bar{v}_n-r_n\right)\\
		&\leq \theta+\frac{1}{N}\sum_{n=1}^N\sum_{j=1}^{N_0}\left(\left|X_{j,n}-\left(\bar{W}_{\cdot,j}\zz \bar{v}_n+\bar{b}_{2,j}\right)_{+}\right|+(\bar{W}_{\cdot,j}\zz r_n)_{+}\right)^2\\
		&\leq \theta+\frac{2}{N}\sum_{n=1}^N\sum_{j=1}^{N_0}\left(\left|\left(\bar{W}_{\cdot,j}\zz \bar{v}_n+\bar{b}_{2,j}\right)_{+}-X_{j,n}\right|^2+(\bar{W}_{\cdot,j}\zz r_n)^2_{+}\right)\\ 
		&\leq \theta+2\theta+\frac{2}{N}\sum_{n=1}^N\sum_{j=1}^{N_0}(\bar{W}_{\cdot,j}\zz r_n)^2_{+}\leq 3\theta+\frac{2}{N}\sum_{n=1}^N\sum_{j=1}^{N_0}(\bar{W}_{\cdot,j}\zz r_n)^2\\
		&
		\leq 3\theta+\frac{2}{N}\|\bar{W}\|_F^2\left(\sum_{n=1}^Ne\zz r_n\right)^2 \leq \delta,
	\end{align*}
	where the first inequality comes from $ |(a_1+a_2)_{+}+a_3|\leq|(a_1)_{+}+a_3|+(a_2)_{+}$ with $a_1,a_2,a_3\in\R$, the second last inequality uses the fact $r_n\geq 0$,
	and
	the last inequality is from $\lb_2\|\bar{W}\|_F^2\leq \theta$, $\beta\sum_{n=1}^Ne\zz r_n\leq \theta$, $\beta> L_{\mathcal{R}}\geq \lb_1$, and the definition of $\delta$.
	Hence $\Omega_5$ is nonempty. Obviously $\Omega_5\subset \Omega_4$ and $\{z: W=\bar{W}, b=\bar{b}, V\in \Omega_5\} \subset \Omega_1.$
	
	On the other hand, it is clear that $L_{\mathcal{F}}+L_{\mathcal{R}}$ is also a Lipschitz constant of $\zeta_2(V)$ over $\Omega_4$.
	Besides, we have $\sum_{n=1}^N\|v_n-(\bar{W}x_n+\bar{b}_1)_{+}\|_1\geq \operatorname{dist}\left(z, \Omega_5\right)$ for all $z \in \Omega_4$, which is resulted from~\cite[Proposition 4]{hoffman2003approximate}.
	Together with $\beta>L_{\mathcal{F}}+L_{\mathcal{R}}$
	and~\cite[Lemma 3.1]{chen2016penalty}, we obtain $\bar{V}$ is also a global minimizer of
	$
	\min_{V\in\Omega_5}\zeta_2(V).
	$
	Hence $\bar{ z}\in\Omega_1$. From $\Omega_1\subset \Omega_2$, we obtain $\bar{ z}\in {\cal S}^*.$  We complete the proof.
\end{proof}


%

The above two theorems show that the solution sets $\mathcal{S}^*$,  $\mathcal{S}$  of problems~\eqref{eq:reguauto2} and~\eqref{eq:reguauto4}  are the same and contain the solution set
${\cal Z}^*$ of~\eqref{eq:reguauto7} that is bounded. The following example shows that
the solution sets $\mathcal{S}^*$ and $\mathcal{S}$ are unbounded for some data set $X$.

{\bf Example 2.1}	Let $ z^*$ be a global minimizer of problem~\eqref{eq:reguauto2} with
\begin{equation*}
	\begin{matrix}
		X=(x_1,x_2)=\begin{bmatrix}
			0&0\\1&2
		\end{bmatrix}\in\R^{2\times 2},\,\,W=\begin{bmatrix}
			w_1,w_2
		\end{bmatrix}\in\R^{1\times 2},\,\, b_1 \in \R,  \,\, b_2=\begin{bmatrix}
			b_{2,1}\\b_{2,2}
		\end{bmatrix}\in\R^{2}.
	\end{matrix}
\end{equation*}

We set ${\widehat{z}=(\mathrm{vec}(W^*)\zz, \hat{b}\zz,\mathrm{vec}(V^*)\zz)\zz\in\R^{N_2}}$
with  $\hat{b}_1=b^*_1$,  $\hat{b}_{2,2}=b_{2,2}^*$ and $\hat{b}_{2,1}\leq\min\{-w_1^*(w^*_2+b^*_1)_{+}, -w^*_1(2w^*_2+b^*_1)_{+}\}$.

From $\hat{b}_2 \in \argmin_{b_2}\sum_{i=n}^{2}\left\|((W^*)\zz (W^*x_n+b_1^*)_++b_2)_{+}-x_n\right\|_2^2$,
we can verify that $\widehat{z}$ is also a global minimizer of~\eqref{eq:reguauto2}.  Hence $\mathcal{S}^*$ is unbounded, and  the solution set $\mathcal{S}$ of problem~\eqref{eq:reguauto4} with any large $\beta>0 $
is also unbounded.

By the similar argument, we can claim the following relationships among the local minimizers of~\eqref{eq:reguauto2},~\eqref{eq:reguauto4} and~\eqref{eq:reguauto7}.
\begin{corollary}
	Let  $L_{\mathcal{F}}$ and $L_{\mathcal{R}}$ be Lipschitz modulus of $\mathcal{F}$ and $\mathcal{R}$ over $\Omega_{\theta}$ respectively. Suppose $\beta>L_{\mathcal{F}}$+$L_{\mathcal{R}}$.
	If $\bar{ z}\in\Omega_{\theta}$ is a local minimizer of~\eqref{eq:reguauto2} or \eqref{eq:reguauto7},
	then $\bar{ z}$ is also a local minimizer of~\eqref{eq:reguauto4}.
	If $\bar{ z}\in\Omega_{\theta}\cap \Omega_3$
	is a local minimizer  of~\eqref{eq:reguauto4},
	then $\bar{ z}$ is also a local minimizer of~\eqref{eq:reguauto7}.
\end{corollary}

\subsection{Stationary Points}

In this subsection, we investigate the relationships among the stationary points of problems~\eqref{eq:reguauto2},~\eqref{eq:reguauto4}  and~\eqref{eq:reguauto7}.

From ${\cal Z}=\Omega_2\cap \Omega_3$, we have $\{z:z\in\mathcal{Z},\,\mathcal{O}(z)\leq \theta\}\subset \Omega_\theta.$

\begin{theorem}\label{thm:dsta}
	Let $L_{\mathcal{F}}$ and $L_{\mathcal{R}}$
	be the Lipschitz modulus of $\mathcal{F}$ and $\mathcal{R}$ over $\Omega_\theta.$
	Suppose $\beta>L_{\mathcal{F}}+L_{\mathcal{R}}$. If $\bar{z}\in{\cal Z} $ with $\mathcal{O}(\bar{z})< \theta$
	is a d-stationary point of~\eqref{eq:reguauto7}, then $\bar{z}\in\Omega_1$
	is a d-stationary point of \eqref{eq:reguauto4} and \eqref{eq:reguauto2}.
\end{theorem}
\begin{proof}
	Firstly, we show $\bar{z}\in\Omega_1$.
	
	Assume on contradiction that $\bar{z}\notin\Omega_1$,
	we construct $\widetilde{ z}=(\mathrm{vec}(\bar{W})\zz, \bar{b}\zz,\mathrm{vec}(\widetilde{V})\zz)\zz$
	with   $\widetilde{v}_n=(\bar{W}x_n+\bar{b}_{1})_{+}$ for {all} $n=1,\ldots,N$.
	It
	{then follows from} $\bar{z}\in \mathcal{Z}$ that $\bar{ z}\in\Omega_2$, which further implies that  $\bar{v}_n\geq(\bar{W}x_n+\bar{b}_{1})_{+}$ for all $n=1,\ldots,N$.
	Hence, we have $\widetilde{V}\le \bar{V}$ and $\widetilde{V}\neq \bar{V}.$
	
	Since $\mathcal{O}(\bar{z})< \theta$ and $\mathcal{O}$ is locally Lipschitz continuous, there exists $t_1\in(0,1]$
	such that $\mathcal{O}(\bar{ z}+t (\widetilde{ z}-\bar{z}))\leq\theta$ for all $0<t<t_1$. Together with $\|\cdot\|_2\leq\|\cdot\|_1$ and the definition of $\bar{ z}$ and $\widetilde{ z}$, we have
	\begin{align*}
		&\mathcal{O}(\bar{ z}+t (\widetilde{ z}-\bar{z}) )-\mathcal{O}(\bar{ z})= \mathcal{F}(\bar{ z}+t (\widetilde{ z}-\bar{z}))-\mathcal{F}(\bar{ z})\\&\quad\quad\quad+\mathcal{R}(\bar{ z}+t (\widetilde{ z}-\bar{z}))-\mathcal{R}(\bar{ z})+\beta\sum_{n=1}^Ne\zz  (\bar{ v}_n+t (\widetilde{ v}_n-\bar{v}_n))-\beta\sum_{n=1}^Ne\zz  \bar{ v}_n\\
		\leq& t(L_{\mathcal{F}}+L_{\mathcal{R}})\| \bar{z}-\widetilde{ z}\|_2+t\beta\sum_{n=1}^Ne\zz  (\widetilde{ v}_n-\bar{v}_n)
		\leq t(\beta-(L_{\mathcal{F}}+L_{\mathcal{R}}))\sum_{n=1}^Ne\zz  (\widetilde{ v}_n-\bar{v}_n).
	\end{align*}
	
	{Together} with $\beta>L_{\mathcal{F}}+L_{\mathcal{R}}$, $\widetilde{V}\neq \bar{V}$, and $\widetilde{v}_n\leq \bar{v}_n$ for {all} $n=1,\ldots,N$,
	{we arrive at }
	\begin{align}\label{eq:ozzleq}
		\mathcal{O}^{\prime}(\bar{z}; \widetilde{ z}-\bar{z})
		\leq(\beta-(L_{\mathcal{F}}+L_{\mathcal{R}}))\sum_{n=1}^Ne\zz  (\widetilde{ v}_n-\bar{v}_n)<0.
	\end{align}
	On the other hand, it holds that $\mathcal{T}_{\mathcal {Z}}(\bar{z})=\{d : (Ad)_i\le 0, \, i\in {\cal A}\} $ where
	${\cal A}=\{i \in \{1,\ldots, \nu\} : (A\bar{z})_i =c_i\}$. Since $\widetilde{W}=\bar{W}, \widetilde{b}=\bar{b}, \widetilde{V}\le \bar{V}$, we have $(A(\widetilde{z}-\bar{z}))_i \leq 0$, for all
	$i \in {\cal A}$, i.e. $\widetilde{ z}-\bar{z}\in\mathcal{T}_{\mathcal{Z}}(\bar{z})$. This together with~\eqref{eq:ozzleq}  contradicts to that $\bar{z}$ is a d-stationary point of~\eqref{eq:reguauto7} in (\ref{dLRP}), which means  $O^{\prime}\left(\bar{z} ; d\right)\geq 0 $ for all $d\in\mathcal{T}_{\mathcal{Z}}(\bar{z})$.
	{Hence}, we have $ \bar{z}\in\Omega_1$, which implies $\bar{v}_n=(\bar{W}x_n+\bar{b}_{1})_{+}$ {holds for all} $n=1,\ldots,N$.
	
	Secondly, we prove that $\bar{z}$ is a d-stationary point of~\eqref{eq:reguauto4}.
	For any $d\in\mathcal{T}_{\Omega_2}(\bar{z})$, there exists $t_2\in(0,1]$ such that $\bar{z}+t d\in\Omega_2$ and $\mathcal{O}(\bar{z}+t d)<\theta$ for all $0\leq t\leq t_2$, since $\Omega_2$ is a convex set. 
	Together with~Theorem \ref{thm:nonem} (b), we have $\mathrm{Proj}_{\Omega_3}(\bar{z}+t d)\in\mathcal{Z}$ for all $0\leq t\leq t_2$, and
	\begin{equation}\label{eq:oprimed}
		\mathcal{O}^{\prime}(\bar{z}; d)
		=\lim _{ t \downarrow 0} \frac{\mathcal{O}(\bar{z}+t d )-\mathcal{O}(\bar{z})}{t}=\lim _{ t \downarrow 0,t\leq t_2} \frac{\mathcal{O}(\mathrm{Proj}_{\Omega_3}(\bar{z}+t d ))-\mathcal{O}(\bar{z})}{t}.
	\end{equation}
	
	Define a function $\zeta:\R_{+}\mapsto\R^{N_2}$ satisfying $\zeta(t)=\mathrm{Proj}_{\Omega_3}(\bar{z}+t d)$, then $\zeta(t)$ is a piecewise linear function with respect to $t$, due to the explicit formula of Proj$_{\Omega_3}(z)$ for $z \in \Omega_\theta$ (cf. (\ref{eq:setcon})). Hence, there exists $t_3\in(0,t_2]$ such that for all $0<t<t_3$, we have $\zeta(t)=(1-\frac{t}{t_3})\zeta(0)+\frac{t}{t_3}\zeta(t_3)$. 
	Together with~\eqref{eq:oprimed}, $t_3\leq t_2$, $\zeta(0)=\bar{z}$, and $\mathrm{Proj}_{\Omega_3}(\bar{z}+t_3 d)-\bar{z}\in\mathcal{T}_{\mathcal{Z}}(\bar{z})$, we arrive at
	\begin{equation*}
		\begin{aligned}
			&\mathcal{O}^{\prime}(\bar{z}; d)
			=\lim _{ t \downarrow 0} \frac{\mathcal{O}(\bar{z}+\frac{t}{t_3}(\zeta(t_3)-\bar{z}))-\mathcal{O}(\bar{z})}{t}=
			\mathcal{O}^{\prime}\left(\bar{z}; \frac{1}{t_3}(\zeta(t_3)-\bar{z})\right)\geq 0
		\end{aligned}
	\end{equation*}
	for all $d\in\mathcal{T}_{\Omega_2}(\bar{z})$. Hence $\bar{z}$ is a d-stationary point of~\eqref{eq:reguauto4}.

	Finally, we prove  that $\bar{z}$ is a d-stationary point of~\eqref{eq:reguauto2}.
	Since  the difference between the objective functions of~\eqref{eq:reguauto4} and \eqref{eq:reguauto2} is the term ${\cal P}$,
	we only need to prove $\mathcal{P}^{\prime}(\bar{z}; d) = 0$ for all $d\in\mathcal{T}_{\Omega_1}(\bar{z})$, which together with $\Omega_1\subset\Omega_2$ and $\mathcal{O}^{\prime}(\bar{z}; d)\geq 0$ for all $d\in\mathcal{T}_{\Omega_2}(\bar{z})$ yields
	(\ref{dR}).

	For a fixed $d \in {\cal T}_{\Omega_1}(\bar{z}),$  by the definition of ${\cal T}_{\Omega_1}(\bar{z})$, let $\{\tau_k\}$ be  a  sequence of positive numbers with $\tau_k\le \tau_{k-1}$ converging to zero, and $\{z^{(k)}\} \subset \Omega_1 $ a sequence converging to $\bar{ z}$  such that $d=\lim_{ k\rightarrow\infty}d^{(k)}$ with $d^{(k)}=\frac{z^{(k)}-\bar{z}}{\tau_k}$. From $z^{(k)}=
	\bar{z} +\tau_k d^{(k)} \in \Omega_1, $ we have ${\cal P}(\bar{z} +\tau_kd^{(k)})=0$.
	Note that $\bar{z}\in \Omega_1$ implies ${\cal P}(\bar{z})=0$. Hence from the Lipschitz continuity and directional differentiability of ${\cal P}$, we obtain
	\begin{align*}
		\mathcal{P}^{\prime}(\bar{z}; d)&=\lim_{ t \downarrow 0 }\frac{\mathcal{P}(\bar{z}+td)-\mathcal{P}(\bar{z})}{t}=\lim_{ \tau_k \downarrow 0 }\frac{\mathcal{P}(\bar{z}+\tau_kd) - {\cal P}(\bar{z})}{\tau_k}\\
		&=\lim_{ \tau_k \downarrow 0 }\frac{\mathcal{P}(\bar{z}+\tau_kd) - {\cal P}(\bar{z} +\tau_kd^{(k)})}{\tau_k}=0.
	\end{align*}
	Since $d \in {\cal T}_{\Omega_1}(\bar{z})$ is arbitrarily chosen, we complete the proof.
\end{proof}
\begin{theorem}\label{thm:ggdsta}
	Let $L_{\mathcal{F}}$ and $L_{\mathcal{R}}$
	be the Lipschitz modulus of $\mathcal{F}$ and $\mathcal{R}$ over $\Omega_\theta.$
	Suppose $\beta>L_{\mathcal{F}}+L_{\mathcal{R}}$. If $\bar{z}\in\mathcal{Z}$ with $\mathcal{O}(\bar{z})< \theta$
	is a generalized KKT point of~\eqref{eq:reguauto7}, then $\bar{z}\in\Omega_1$ is a generalized d-stationary point of \eqref{eq:reguauto7}.
	In additional, if $\bar{z} \in \mathrm{int}(\Omega_3)$, then $\bar{z}$ is a generalized d-stationary point of \eqref{eq:reguauto4}. Furthermore, if $\mathcal{P}$ is regular at $\bar{ z}$, then $\bar{ z}$
	is a generalized d-stationary point of \eqref{eq:reguauto2}.
\end{theorem}
\begin{proof}			
	By the definition of generalized KKT point of \eqref{eq:reguauto7}, \cite[Proposition 2.1.2]{clarke1990optimization} and $\mathcal{T}_{\mathcal {Z}}(\bar{z})=\{d : (Ad)_i\le 0, \, i\in {\cal A}\} $ where
	${\cal A}=\{i \in \{1,\ldots, \nu\} : (A\bar{z})_i =c_i\}$, we have
	$$0\le -(Ad)^{\top}\bar{\gamma}\le \max_{\xi \in \partial {\cal O}(\bar{z})} \xi^\top d=\mathcal{O}^{\circ}\left(\bar{z} ; d \right), \quad \forall d \in \mathcal{T}_{\mathcal {Z}}(\bar{z}),$$
	which implies that $\bar{ z}$ is a generalized d-stationary point of \eqref{eq:reguauto7}.

	Now, we prove that $\bar{z}\in\Omega_1$.
	
	Assume on contradiction that $\bar{z}\notin\Omega_1$,
	we construct the same $\widetilde{ z}$ as that in the proof of Theorem \ref{thm:dsta}.
	Since $\mathcal{O}(\bar{z})< \theta$ and $\mathcal{O}$ is locally Lipschitz continuous, there exists $\epsilon>0$ such that for all $z\in\mathcal{B}_{\epsilon}(\bar{ z})$, it holds that $\mathcal{O}(z)< \theta$. Furthermore, for any $z\in\mathcal{B}_{\epsilon}(\bar{ z})$, there exists $t_1\in(0,1]$
	such that $\mathcal{O}(z+t (\widetilde{ z}-\bar{z}))\leq\theta$ for all $0<t<t_1$. Together with $\|\cdot\|_2\leq\|\cdot\|_1$ and the definition of $z$ and $\widetilde{ z}$, we also have
	\begin{align*}
		&\mathcal{O}(z+t (\widetilde{ z}-\bar{z}) )-\mathcal{O}(z)
		\leq t(\beta-(L_{\mathcal{F}}+L_{\mathcal{R}}))\sum_{n=1}^Ne\zz  (\widetilde{ v}_n-\bar{v}_n).
	\end{align*}
	Using a similar method as that in the proof of Theorem \ref{thm:dsta}, we have $ \bar{z}\in\Omega_1$, which implies $\bar{v}_n=(\bar{W}x_n+\bar{b}_{1})_{+}$ {holds for all} $n=1,\ldots,N$.

	Since $\bar{z} \in \mathrm{int}({\Omega_3})$ implies
	${\cal T}_{\mathcal{Z}}(\bar{z})={\cal T}_{\Omega_2}(\bar{z})$, we obtain that $\bar{z}$ is a generalized d-stationary point of~\eqref{eq:reguauto4}.  
	
	Finally, for all $d\in\mathcal{T}_{\Omega_1}(\bar{ z})$, we have 
	\begin{align*}
		&(\mathcal{F}+\mathcal{R})^{\circ}(\bar{ z};d)=(\mathcal{F}+\mathcal{R})^{\circ}(\bar{ z};d)+\mathcal{P}^{\prime}(\bar{ z};d)=(\mathcal{F}+\mathcal{R})^{\circ}(\bar{ z};d)+\mathcal{P}^{\circ}(\bar{ z};d)\\\geq&\mathcal{O}^{\circ}(\bar{ z};d)\geq 0,
	\end{align*}
	where the first equality comes from $\mathcal{P}^{\prime}(\bar{z}, d)=0$ (see the last part of the proof of Theorem \ref{thm:dsta}), the second equality comes from $\mathcal{P}$ being regular at $\bar{ z}$, and the last inequality comes from $d\in\mathcal{T}_{\Omega_1}(\bar{ z})\subset\mathcal{T}_{\Omega_2}(\bar{ z})$. Hence, $\bar{ z}$
	is a generalized d-stationary point of \eqref{eq:reguauto2}.
\end{proof}

We end this section by summarizing our results for the relationship of problems \eqref{eq:reguauto2},~\eqref{eq:reguauto4} and~\eqref{eq:reguauto7} with $\beta > L_{\cal F}+L_{\cal R}$ 
in the following diagram, where $\bar{ z}\in\Omega_\theta$, $L_{\mathcal{F}}$ and $L_{\mathcal{R}}$
are the Lipschitz modulus of $\mathcal{F}$ and $\mathcal{R}$ over $\Omega_\delta$, respectively. 
\begin{equation*}\label{eq:relation}
	\footnotesize
	\boxed{\begin{aligned}
			&R: \,\:\quad \text{global minimizer}\,\quad\,\text{local minimizer}\,\quad\, \text{d-stationary point}\,\quad\, \text{ generalized d-stationary point }\\
			&\quad\quad\quad\quad\quad\,\Downarrow\Uparrow\quad\quad\quad\quad\quad\,\,\,\,\quad\quad\Downarrow\hspace{0.7in}\Uparrow\hspace{1.22in}\Uparrow\text{$\mathcal{P}$ is regular at $\bar{ z}$}\\
			&RP:\:\,\: \text{global minimizer}\,\quad\, \text{local minimizer}\,\quad\,  \text{d-stationary point}\,\quad\,\text{ generalized d-stationary point } \\
			&\quad\quad\,\,\tiny{\text{$\bar{z}\in\Omega_3$ }}\Downarrow\Uparrow\quad\quad\quad\,\quad\quad\,\tiny{\text{$\bar{z}\in\Omega_3$ }}\Downarrow\Uparrow \quad \quad \quad\quad\quad\quad\,\,\,\,\Uparrow\tiny{\text{$\mathcal{O}(\bar{z})<\theta$ }} \quad\quad\quad\quad\quad\quad\quad\:\:\, \Uparrow\tiny{\text{$\bar{z}\in\mathrm{int}(\Omega_3),\mathcal{O}(\bar{z})<\theta$ }}\\
			&LRP: \text{global minimizer}\,\quad\, \text{local minimizer}\,\quad\,  \text{d-stationary point}\,\quad\,\text{ generalized d-stationary point }
	\end{aligned}}
\end{equation*}
\section{A Smoothing Proximal Gradient Algorithm}\label{sec:smooth}

In this section, we propose a smoothing proximal gradient algorithm (SPG) for solving problem~\eqref{eq:reguauto7}.
The proposed SPG introduces a smoothing function of the objective function  of \eqref{eq:reguauto7} and solves a strongly convex quadratic program over its feasible set ${\cal Z}=\{z:A z\leq c\}$ at each iteration.   	In the rest of this section, we first present
the algorithm framework and then establish convergence results of the algorithm.

\subsection{Algorithm Framework}

\begin{definition}\label{def:smooth}\cite{chen2012smoothing}
	Let $f : \R^m \mapsto \R$ be a continuous function. We call $\widetilde{f}: \R^m \times \R_{+} \mapsto \R$ a smoothing function of $f$, if for all fixed {$\mu>0$}, $\widetilde{f}(\cdot,\mu)$ is continuously differentiable, and $\lim_{y\rightarrow\bar{y},\mu\downarrow 0} \widetilde{f}(y,\mu) = f(\bar{y})$.
\end{definition}

In this paper, we adopt the following smoothing function $\widetilde{\sigma}(y,\mu):\R^m\times\R _{+}\mapsto\R^m$ for the ReLU activation function $\sigma=(y)_+$ as follows.
\begin{equation*}
	\widetilde{\sigma}_i(y,\mu) =\left\{\begin{aligned}
		&0 \quad\quad\,\,\,\,\text{ if }y_i<0,\\
		&\frac{y_i^2}{2\mu}\quad\,\,\,\,\,\text{ if }0\leq y_i\leq\mu,\\
		&y_i-\frac{\mu}{2} \,\,\,\,\text{ if } y_i>\mu
	\end{aligned}\right.
\end{equation*}
for all $i=1,\ldots,m$, where $y_i$ is the $i$-th element of $y\in \R^m$. Then, we obtain that  $\nabla\widetilde{\sigma}_i(y,\mu)=\min\left\{\max\left\{\frac{y_i}{\mu},0\right\},1\right\}$, and $\widetilde{\sigma}(y,\mu_1)<\widetilde{\sigma}(y,\mu_2)$ with $\mu_1>\mu_2$.

We construct a smoothing function of $\mathcal{O}( z)$ over {$\mathcal{Z}$} for $\mu>0$,
\begin{equation}\label{eq:smooth}
	\begin{aligned}
		\widetilde{\mathcal{O}}( z,\mu):&=\widetilde{\mathcal{H}}( z,\mu)+\mathcal{R}(z),
	\end{aligned}
\end{equation}
where $\widetilde{\mathcal{H}}( z,\mu):=\widetilde{\mathcal{F}}( z,\mu)+\widetilde{P}( z,\mu)$, and
$$\widetilde{\mathcal{F}}( z,\mu)=\frac{1}{N}\sum_{n=1}^N\|(W\zz v_n+b_2)_{+}\|_2^2+\frac{1}{N}\|X\|_F^2-\frac{2}{N}\sum_{n=1}^Nx_n\zz \widetilde{\sigma}(W\zz v_n+b_2,\mu),$$
$$\widetilde{P}( z,\mu)=\beta\sum_{n=1}^Ne\zz \left(v_n-\widetilde{\sigma}(Wx_n+b_{1},\mu)\right)$$
{are the smoothing functions of $\mathcal{F}(z)$ and $\mathcal{P}(z)$, respectively. Here we use the smoothness of $\sum_{n=1}^N\|(W\zz v_n+b_2)_{+}\|_2^2$.  
	It is clear that $\widetilde{\sigma}(Wx_n+b_{1},\mu) \le (Wx_n+b_{1})_+$ and  $\widetilde{\mathcal{O}}( z,\mu_1)>\widetilde{\mathcal{O}}( z,\mu_2)$ for $\mu_1>\mu_2$ and  {$ z\in\mathcal{Z}$. In addition, for all $ z\in\mathcal{Z}$ and $\mu>0$,} we have
	\begin{equation}\label{eq:obound}
		0\leq \mathcal{O}( z)\leq \widetilde{\mathcal{O}}( z,\mu)\leq \mathcal{O}( z)+(\|X\|_1+N_1N\beta)\mu.
	\end{equation}
	{The function $\mathcal{R}$ is a convex quadratic function and the eigenvalues of the Hessian matrix of $\mathcal{R}$ are in $\{0, 2\lb_2\}$.
		It is clear that $\widetilde{\mathcal{P}}(\cdot,\mu)$, $\widetilde{\mathcal{F}}(\cdot,\mu)$, $\nabla_z\widetilde{\mathcal{P}}(\cdot,\mu)$ and $\nabla_z\widetilde{\mathcal{F}}(\cdot,\mu)$ are
		locally Lipschitz continuous for any fixed $\mu>0$. Moreover, $\mu\widetilde{\mathcal{P}}$, $\mu\widetilde{\mathcal{F}}$, $\mu\nabla_z\widetilde{\mathcal{P}}$ and $\mu\nabla_z\widetilde{\mathcal{F}}$ are piecewise quadratic functions with respect to $\mu$.
		
		By the proof of Theorem \ref{thm:nonem}, the set $\Omega_\theta\cap {\cal Z}$ is bounded and $\|z\|_{\infty}\leq \max\{\alpha,\eta\}$ holds for any $z\in  \Omega_\theta\cap {\cal Z}$, where $\eta:=\max\{\sqrt{\frac{N_1N_0\theta}{\lb_2}}, \frac{\theta}{\lb_1}\}$. 
		Let $L_{\widetilde{\mathcal{H}}}$ and $L_{\nabla\widetilde{\mathcal{H}}}$ be Lipschitz modulus of 
		$\mu\widetilde{\mathcal{H}}$ over $\Omega_\theta\cap {\cal Z}\times (0,1)$, and $\mu \nabla_z\widetilde{\mathcal{H}}$
		over  $\{z:\|z\|_{\infty}\leq \max\{\alpha,2\eta\}\}\times (0,1),$ respectively.
		
		Our smoothing proximal gradient algorithm is presented in Algorithm \ref{alg:autoencoder1}.}

	\begin{algorithm}
		\caption{A smoothing proximal gradient algorithm (SPG)}
		\label{alg:autoencoder1}
		\begin{algorithmic}[1]
			\STATE{Initialization: choose $z^{(0)}\in\mathcal{Z}$, $0<\mu^{(0)}<1$, $0<\tau_1<1$, $\tau_2>0$, $\tau_3\geq 1$, and $L^{(0)}\geq 1$. Set $k:=0$}.
			\WHILE{{a termination criterion is not met,} }
			\STATE\label{line3}
			Set $z^{(k+1)}$ be the unique minimizer of the strongly convex quadratic program  \begin{equation}\label{eq:updateWBV1}
				\min_{ z\in\mathcal{Z}}\left\langle\nabla_z \widetilde{\mathcal{H}}( z^{(k)},\mu^{(k)}), z- z^{(k)}\right\rangle+\mathcal{R}( z)+\frac{L^{(k)}}{2}\| z- z^{(k)}\|_2^2.
			\end{equation}
			
			\STATE\label{line5}{
				Update the smoothing and proximal parameters $\mu^{(k+1)}$ and  $L^{(k+1)}$ by
				\begin{equation}\label{eq:mu}
					\left\{\begin{aligned}
						&(\mu^{(k+1)},\,L^{{(k+1)}}):=(\mu^{(k)}, \,L^{(k)}), \,\mbox{ if\, } \widetilde{\mathcal{O}}( z^{(k+1)},\mu^{(k)})- \widetilde{\mathcal{O}}( z^{(k)},\mu^{(k)})<-\tau_2\frac{\mu^{(k)}}{L^{(k)}}, \\ &(\mu^{(k+1)},\,L^{{(k+1)}}):=(\tau_1\mu^{(k)},\,\tau_3 L^{(k)}), \,\mbox{ otherwise. } \end{aligned}\right.
				\end{equation}
				Set $k:=k+1$.
			}
			\ENDWHILE
		\end{algorithmic}
	\end{algorithm}
	
	\subsection{Convergence Analysis}
	
	The following lemma will be used for the convergence results of SPG.
	
	\begin{lemma}\label{lem:alg1}
		Let $\{ z^{(k)}\}$ and $\{\mu^{(k)}\}$ be the sequences generated by~Algorithm \ref{alg:autoencoder1} with $\mathcal{O}(z^{(0)})<\theta$, $\tau_1\tau_3\geq 1$ and $\mu^{(0)}L^{(0)}$
		satisfying
		\begin{equation}\label{eq:LLL}
			\mu^{(0)}L^{(0)}\geq \max\left\{6\lb_2N_1N_0+\frac{2}{\eta}(N_2 L_{\widetilde{\mathcal{H}}}+\lb_1N_1N),\,\, 8\lb_2+L_{\nabla\widetilde{\mathcal{H}}}\right\}.
		\end{equation}
		Then, the following statements hold.
		
		(a) The sequence \(\left\{ \widetilde{\mathcal{O}}( z^{(k)},\mu^{(k)})\right\}\) is non-increasing, and $\{z^{(k)}\}\subset \Omega_\theta\cap {\cal Z} $;
		
		(b) If
		$\widetilde{\mathcal{O}}( z^{(k+1)},\mu^{(k)})- \widetilde{\mathcal{O}}( z^{(k)},\mu^{(k)})\geq -\tau_2\frac{\mu^{(k)}}{L^{(k)}}$, then there exists a nonnegative vector $\gamma^{{(k+1)}}\in\R^\nu$ such that
		\begin{equation}\label{eq:kktregu}		
			\begin{aligned}
				&\left\|\nabla_z\widetilde{\mathcal{O}}( z^{(k)},\mu^{(k)})+ A\zz \gamma^{{(k+1)}}\right\|_{2}\leq  2\sqrt{\tau_2}(\mu^{(k)})^{1/2}, \text{ and}\\
				&A z^{(k)} \leq c, \,\,\,  -\tau_2\frac{(\mu^{(k)})^2}{8\lb_2+L_{\nabla\widetilde{\mathcal{H}}}}\leq (\gamma^{{(k+1)}})\zz(A z^{(k)}-c)\leq 0.
			\end{aligned}
		\end{equation}
	\end{lemma}
	
	The proof is given in~Section \ref{sec:bound}.
	It is worth mentioning that  $z^{(0)}=0$ satisfies
	the condition on the initial guess in~Lemma \ref{lem:alg1}. Now we present
	our main convergence theorem as follows.
	
	\begin{theorem}\label{thm:alg1}
		Under assumptions of Lemma  \ref{lem:alg1},	
		the following statements hold.\\
		(a) $\lim_{k\rightarrow\infty}\mu^{(k)}=0$;\\
		(b) $\{\mathcal{O}(z^{(k)})\}$ and $\{\widetilde{\mathcal{O}}( z^{(k)},\mu^{(k)})\}$
		are convergent. Moreover,  we have
		\begin{equation}\label{eq:ooo}
			\lim_{k\rightarrow \infty}\widetilde{\mathcal{O}}( z^{(k)},\mu^{(k)})=\lim_{k\rightarrow \infty}\mathcal{O}( z^{(k)}).
		\end{equation}
	\end{theorem}
	
	\begin{proof}
		$($a$)$
		Assume on contradiction that there exists $k_0>0$
		such that whenever $k\geq k_0$,
		$\widetilde{\mathcal{O}}( z^{(k+1)},\mu^{(k)})- \widetilde{\mathcal{O}}( z^{(k)},\mu^{(k)})<-\tau_2\frac{\mu^{(k)}}{L^{(k)}}$ holds. Hence, it follows from the updating formula \eqref{eq:mu} of $\mu$ and $L$ that $
		\mu^{(k)}=\mu^{(k_0)}$ and $L^{(k)}=L^{(k_0)}$ for all $k\geq k_0.$
		Together with  the inequality $\widetilde{\mathcal{O}}( z^{(k+1)},\mu^{(k)})\geq\widetilde{\mathcal{O}}( z^{(k+1)},\mu^{(k+1)})$,
		we have
		\begin{equation}\label{eq:de}
			\widetilde{\mathcal{O}}( z^{(k+1)},\mu^{(k+1)})- \widetilde{\mathcal{O}}( z^{(k)},\mu^{(k)})<-\tau_2\frac{\mu^{(k)}}{L^{(k)}}=-\tau_2\frac{\mu^{(k_0)}}{L^{(k_0)}}
		\end{equation}
		for all $k\geq k_0$. 
		
		Denote $l=\lceil \frac{L^{(k_0)}}{\tau_2\mu^{(k_0)}}\widetilde{\mathcal{O}}( z^{{(k_0)}},\mu^{{(k_0)}}) \rceil$. It then follows from \eqref{eq:obound} and
		\eqref{eq:de} that
		\begin{align*}
			&-\widetilde{\mathcal{O}}( z^{{(k_0)}},\mu^{{(k_0)}})
			\leq\widetilde{\mathcal{O}}(z^{{(k_0+l)}},\mu^{{(k_0+l)}})- \widetilde{\mathcal{O}}( z^{(k_0)},\mu^{(k_0)})\\
			=&\sum_{k=k_0}^{k_0+l-1}\left(\widetilde{\mathcal{O}}(z^{{(k+1)}},\mu^{{(k+1)}})- \widetilde{\mathcal{O}}( z^{(k)},\mu^{(k)})\right)<-\tau_2\frac{\mu^{(k_0)}}{L^{(k_0)}} l\leq -\widetilde{\mathcal{O}}( z^{{(k_0)}},\mu^{{(k_0)}}),
		\end{align*}
		which leads to a contradiction. {Hence,
			the second situation in~\eqref{eq:mu} happens infinite times and we prove the assertion
			(a).
			
			$($b$)$
			Recall the nonnegativity of $\widetilde{\mathcal{O}}(z,\mu)$
			and the monotonical non-increasing of} \(\left\{\widetilde{\mathcal{O}}( z^{(k)},\mu^{(k)})\right\}\), {we can conclude that $\{\widetilde{\mathcal{O}}( z^{(k)},\mu^{(k)})\}$ is convergent.
			Together with } $\lim_{k\rightarrow\infty}\mu^{(k)}=0$, we have $\lim_{k\rightarrow \infty}\widetilde{\mathcal{O}}( z^{(k)},\mu^{(k)})=\lim_{k\rightarrow \infty}\mathcal{O}( z^{(k)})$, which completes the proof.
	\end{proof}
	
	
	

	\begin{theorem}\label{thm:gdsta}
		Suppose the assumptions of  Lemma \ref{lem:alg1} hold. Let ${\cal K}=\{k: \mu^{(k+1)}=\tau_1\mu^{(k)}, k\ge 0\}$. Then $\{ z^{(k)}: k\in {\cal K}  \}$ is bounded and
		any accumulation point $ z^{*}$ of $\{ z^{(k)}: k\in {\cal K}  \}$
		is a generalized d-stationary point of~\eqref{eq:reguauto7}.
		Moreover, if ${\cal O}$ is regular at $ z^{*}$, then $ z^{*}$
		is a d-stationary point of~\eqref{eq:reguauto7}.
	\end{theorem}
	\begin{proof}
		Let  $\{\gamma^{(k+1)} \} \subset\R^\nu_+$
		{satisfy}~\eqref{eq:kktregu} for $k\in {\cal K}$. By the structure of the matrix $A$ and Lipschitz continuity of ${\cal O}$,  $\{\gamma^{(k+1)} \}$ is bounded.
		{Due to the fact that $\lim_{k\rightarrow\infty}\mu^{(k)}=0$,
			it holds that ${\cal K}$ has infinitely many elements.  Since $z^*$ is an accumulation point of $\{z^{(k)}:k\in{\cal K}\}$, there exist subsequences $\{z^{(j_k)}\}$ of $\left\{z^{(k)}:k\in {\cal K}\right\}$ and  $\{\gamma^{(j_k+1)}\}$ of $\left\{\gamma^{(k)}: k\in {\cal K}\right\}$ such that  $\lim_{k \rightarrow \infty} z^{(j_k)}=z^*$ and $\lim_{k \rightarrow \infty} \gamma^{(j_k+1)}=\gamma^*$.
			By taking $k$ from the both sides of~\eqref{eq:kktregu} to infinity, we obtain that}
		\begin{equation}\label{eq:infty1}
			\begin{aligned}
				0=\liminf_{k\in{\cal K}, k\rightarrow\infty}\left\|\nabla_z\widetilde{\mathcal{O}}( z^{(k)},\mu^{(k)})+ A\zz \gamma^{{(k+1)}}\right\|,  \quad (Az^*-c)^\top \gamma^*=0, \quad A z^* -c \leq 0.
			\end{aligned}
		\end{equation}

		Since $\mathcal{O}(z)$ is a finite-sum composite max function, we have the gradient sub-consistency
		\cite{BCS,chen2012smoothing} that
		\begin{equation}\label{eq:gf}
			{\rm co}\left\{\lim _{k\rightarrow\infty  } \nabla_z \widetilde{\mathcal{O}}(z^{(j_k)},\mu^{(j_k)})\right\}\subset \partial \mathcal{O}(z^*).
		\end{equation}
		Hence, we have $0\in \partial Q(z^*) +A^\top \gamma^*$ and hence $z^*$ is a generalized KKT point of  \eqref{eq:reguauto7}.

		It follows Lemma \ref{thm:ggdsta} that $z^*$ is a  generalized d-stationary point of \eqref{eq:reguauto7}.
		
		If ${\cal O}$ is regular at $z^*$, then $\mathcal{O}^{\circ}(z^* ; d)= \mathcal{O}'(z^* ; d)$ for all $d \in {\cal T}_{\cal Z}(z^*)$. In this case $z^*$  is a d-stationary point of \eqref{eq:reguauto7}.
	\end{proof}

		\section{Numerical Experiments}\label{sec:num}
		
		In this section, we evaluate the numerical
		performance of our proposed SPG method for training
		the autoencoders. We first introduce the implementation details
		including the default settings for the parameters and the test problems.
		Then we report the numerical comparison
		among our SPG with a few state-of-the-art
		SGD based approaches, including the Adam~\cite{kingma2014adam}, the Adamax~\cite{kingma2014adam}, the Adadelata~\cite{zeiler2012adadelta}, the Adagrad~\cite{duchi2011adaptive}, the AdagradDecay~\cite{duchi2011adaptive}, and the Vanilla SGD~\cite{cramir1946mathematical}. {\color{black} Our numerical experiments will use both synthetic datasets and real dataset.}
		All the numerical experiments in this section are performed on a workstation with one Intel (R) Xeon (R) Silver 4110CPU (at 2.10GHz $\times$ 32) and 384GB of RAM running MATLAB R2018b under Ubuntu 18.10.}

	
	\subsection{Implementation Details}
	
	We first introduce the model parameters involved in
	\eqref{eq:reguauto7}. We set the regularization
	parameter $\lb_2=0.1$. If the sparsity of $V$ is pursued,
	we further set $\lb_1=0.0001$ in~\eqref{eq:reguauto7}.
	The penalty parameter $\beta$ takes the constant value $\frac{1}{N}$.
	
	We explain how to choose the algorithm parameters of SPG.
	The constants $\tau_{1}$ and $\tau_{2}$ take the values $0.5$ and $0.001$, 
	respectively. The initial value of the smoothing parameter is set as $\mu^{(0)}=0.001$. Although the global convergence of Algorithm 3.1 {requires sufficiently large parameters $\tau_{3}$
		and $L^{(0)}$, they can take much smaller values for better performance in practice.} Empirically, we set $\tau_{3}=1.1$ and $L^{(0)}=L_{*}:=\max\{1, \sqrt{N_{0} N_{1} / N}, \beta, N_{0} / 30\},$ unless
	otherwise stated.

	{Recall~\eqref{eq:kktregu}, we set $\mu^{(k)}\leq \varepsilon$ as the stopping criterion of~Algorithm \ref{alg:autoencoder1} and the tolerance $\varepsilon$ takes $10^{-7}$ 
		in the experiments unless otherwise stated.
		Besides, the maximum number of iterations in SPG is set as $4000$.
		Next, we describe the default initial guess in the following.
		The {\color{black} matrix} $W^{(0)}$ are randomly generated
		by $W^{(0)} = \mathrm{randn}(N_1,N_0)/N$, where $\mathrm{randn}(n, p)$ 
		stands for an $n \times p$ randomly generated matrix 
		under the standard Gaussian distribution. 
		Then, we set
		{$b^{(0)}=0$} and {$v^{(0)}_n=(W^{(0)}x_n)_{+}$} for all $n=1,2,\ldots,N$.}
	
	{\color{black}	For solving the quadratic programming subproblem~\eqref{eq:updateWBV1}, 
		we propose a new splitting and alternating method, abbreviated as SAMQP,
		to significantly increase the efficiency by exploiting the special structure. The details descriptions, as well as the comparison with exiting QP solvers, are put in Section \ref{sec:alg2}.}

	The MATLAB codes of the SGD based approaches including Adam, Adamax, Adadelata, Adagrad, AdagradDecay, and Vanilla SGD are downloaded from
	the Library~\cite{JMLR:v18:17-632}. {\color{black} These approaches directly solve problem \eqref{eq:auto}.} All of
	these algorithms are run under their defaulting settings. 
	The batch-size is set to $\max\{N/100,10\}$.

	There are two classes of test problems.
	The first class of data sets are generated randomly.
	Let
	$N$ and $N_{\mathrm{test}}$ be the numbers of training and test samplings,
	respectively. We use parameter $\epsilon_0>0$ to control the
	noise level. We construct the data sets by the following two ways, where $\mathrm{rand}(n, p)$ 
	stands for an $n \times p$ randomly generated matrix under uniform distribution in $[0,1]$.
	\begin{itemize}
		\item Data type 1: 
		we generate the data matrix $X_{\mathrm{all}}=(x_1,x_2,\ldots,x_{N+N_{\mathrm{test}}})$ by setting $x_i\sim\mathcal{N}(\vartheta, \Sigma_0\zz\Sigma_0)+\epsilon_0\mathcal{N}(0, 1)$ for all $i=1,2,\ldots,N+N_{\mathrm{test}}$, where $\vartheta= 0.5+\mathrm{randn}(N_0, 1)$ and $\Sigma_0= \mathrm{randn}(N_0, 1)$. We then set all negative elements of $X_{\mathrm{all}}$ to be zero. The first $N$
		and the last $N_{\mathrm{test}}$ columns of $X_{\mathrm{all}}$ are
		selected to be the training and test sets, respectively.		
		\item Data type 2: 
		we generate the data matrix $X_{\mathrm{all}}$ by $X_{\mathrm{all}}=\mathrm{rand}(N_0,N+N_{\mathrm{test}})+\epsilon_0\mathrm{randn}(N_0,N+N_{\mathrm{test}})$. We then set all negative elements of $X_{\mathrm{all}}$ to be zero. The first $N$
		and the last $N_{\mathrm{test}}$ columns of $X_{\mathrm{all}}$ are
		selected to be the training and test sets, respectively.		 
	\end{itemize}
	
	{In the numerical experiments, we will frequently use the following nine combinations of $(N,N_0,N_1)$ to determine the size of the randomly generated data sets. For convenience, we simply call these combinations
		``E.g. 1 to 9" as follows. Other combinations will be stated otherwise.
		\begin{itemize} 
			\item[(1)] E.g. 1: $N=50$, $N_1=50$, $N_0=25$;\hspace{0.158in} (2) E.g. 2: $N=50$, $N_1=100$, $N_0=25$;
			\item[(3)] E.g. 3: $N=50$, $N_1=100$, $N_0=40$;\quad (4) E.g. 4: $N=50$, $N_1=10$, $N_0=5$;
			\item[(5)] E.g. 5: $N=75$, $N_1=10$, $N_0=5$;\hspace{0.23in} (6) E.g. 6: $N=100$, $N_1=10$, $N_0=5$;
			\item[(7)] E.g. 7: $N=100$, $N_1=100$, $N_0=25$;\hspace{0.023in} (8) E.g. 8: $N=150$, $N_1=10$, $N_0=5$;
			\item[(9)] E.g. 9: $N=150$, $N_1=20$, $N_0=10$.
	\end{itemize}}

	The second class of test problems are selected from the MNIST datasets~\cite{lecun1998mnist,lecun1998gradient} {consisting
		of $10$-classes handwritten digits
		with the size $28\times 28$, namely, $N_0 = 784$.
		In practice, we randomly pick up data entries
		from each class of MNIST under uniform distribution.}
	
	
	We record the following four measurements, the function value of~\eqref{eq:reguauto7} (``FVal''), 
	the average feasibility violation (``FeasVi"), the training error (``TrainErr")  and the test error (``TestErr") of~\eqref{eq:reguauto2}, which are 
	{denoted} by $\mathcal{O}(z)$, 
	$\frac{1}{NN_1}\sum_{n=1}^N \|v_n-(W x_n+b_{1})_{+}\|_1$, $\mathcal{F}(z)$, and $\frac{1}{N_{\mathrm{test}}}\sum_{n=N_{\mathrm{test}}+1}^{N+N_{\mathrm{test}}}\left\|(W\zz v_n+b_2)_{+}-x_n\right\|_2^2$, respectively. We use ``Noise"  and ``Time" to represent the value of $\epsilon_0$ and the CPU  Time in {seconds}, respectively.
	
	\subsection{Properties of SPG}
	In this subsection, we investigate the numerical performance of SPG
	in solving problems with randomly generated data sets.
	We first study the convergence properties of SPG.
	{The test problem is generated by data type 1 with parameter combination E.g. 6.
		and $\epsilon_0=0.05$.}
	The penalty parameter $\beta$ takes its default setting.
	The numerical results of SPG with randomly initial guess is present in Figure \ref{fig:spgnorm1}. We can learn from Figure \ref{fig:spgnorm1} that (i)
	{all of the training error, the test error, the function value of \eqref{eq:reguauto7}}
	decrease in a same order; 
	(ii)
	the feasibility reduces to its tolerance rapidly; (iii) the smoothing parameter sequence
	$\{\mu^{(k)}\}$ converges almost linearly to zero.
	
	\begin{figure}[htbp!]
		\vspace{-5mm}
		\centering
		\setcounter{subfigure}{0}
		\subfloat{\includegraphics[width=43mm]{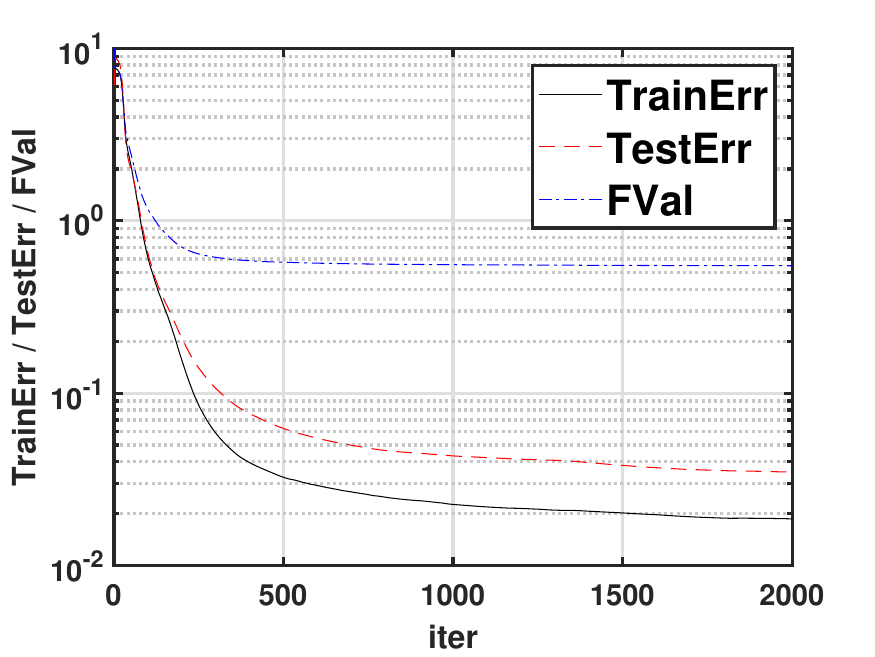}}
		\subfloat{\includegraphics[width=43mm]{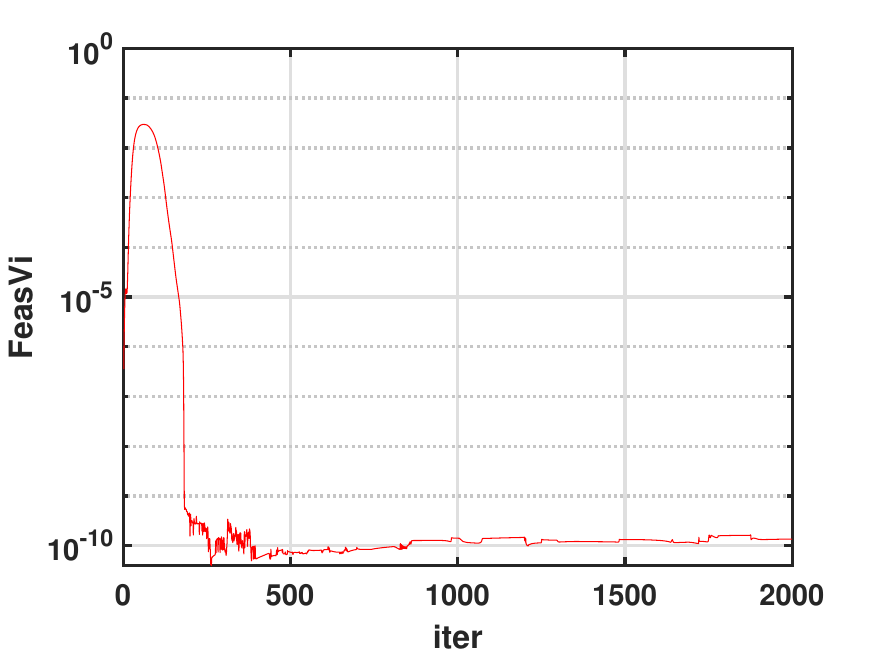}}
		\subfloat{\includegraphics[width=43mm]{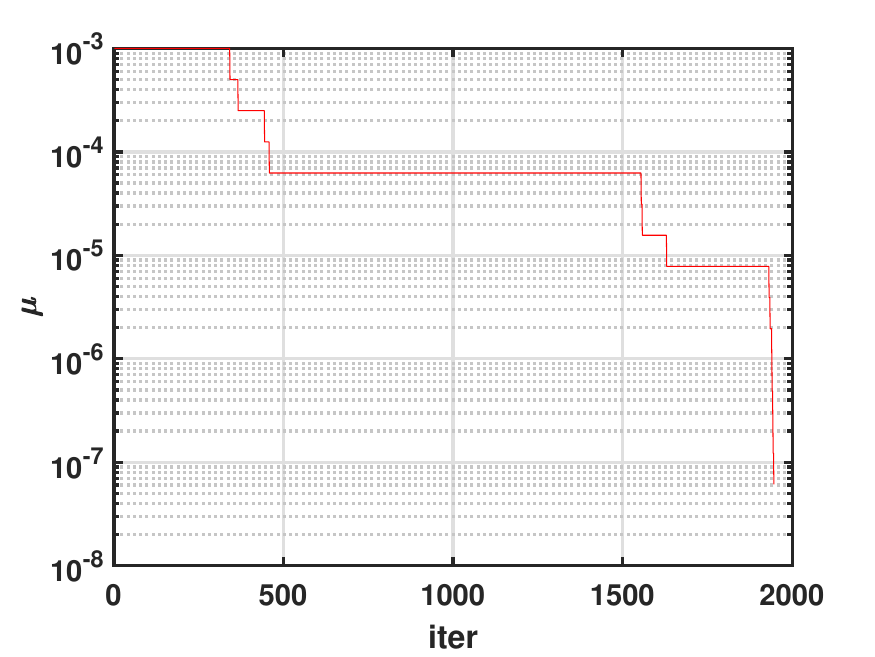}}\\s
		\caption{{Algorithm performance} of SPG}\label{fig:spgnorm1}		\vspace{-5mm}
	\end{figure}
	Secondly, we compare SPG with different choices of $L^{(0)}$ on a group of randomly generated data sets with data type 1, $N_{\mathrm{test}}=0$ and $\epsilon_0=0.05$.
	{We select three different $\beta$ and four parameter combinations.}
	The numerical results are shown in~Figure \ref{fig:Lu1}.
	We can learn from Figure \ref{fig:Lu1} that SPG may diverge
	if $L^{(0)}$
	is not sufficiently large, particularly
	if $\beta$ is large. When $\beta$ is small, the performance of SPG
	is not very sensitive to the choice of $L^{(0)}$. We also find that
	bigger $L^{(0)}$ usually leads to slow convergence. Hence,  we can conclude
	that a suitably selected $L^{(0)}$, such as our default setting, is
	important to SPG.
	
	\begin{figure}[hbp!]
		\vspace{-5mm}
		\centering
		\setcounter{subfigure}{0}
		\subfloat[\scriptsize $\beta=1/N$, E.g. 1]{\includegraphics[width=42mm]{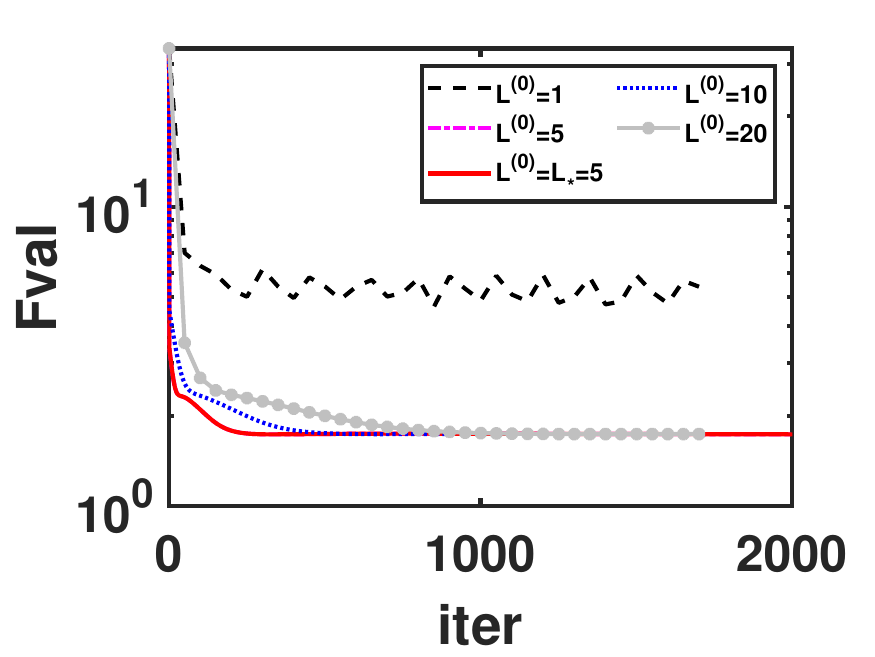}}
		\subfloat[\scriptsize $\beta=1$, E.g. 1]{\includegraphics[width=42mm]{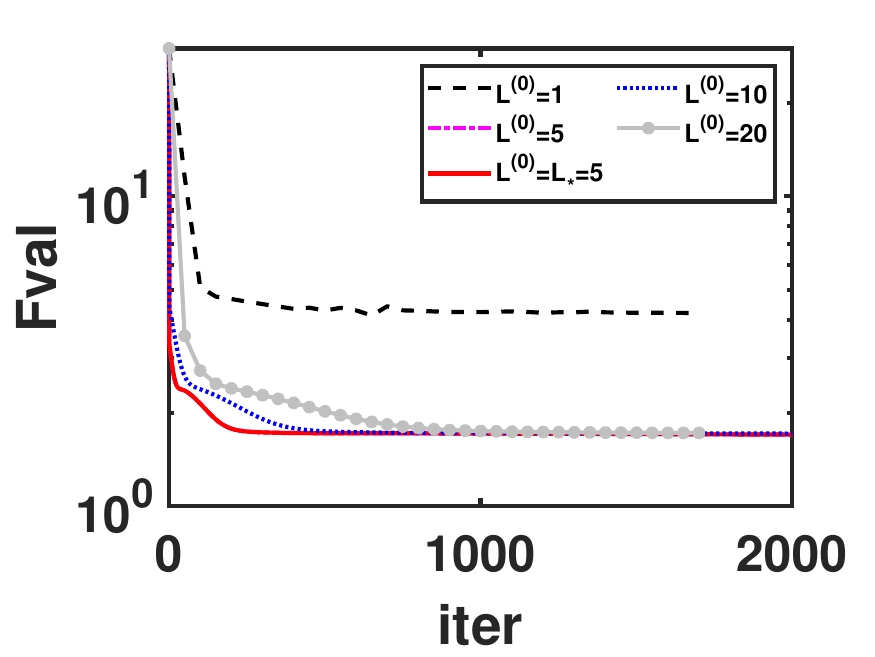}}
		\subfloat[\scriptsize $\beta=10$, E.g. 1]{\includegraphics[width=42mm]{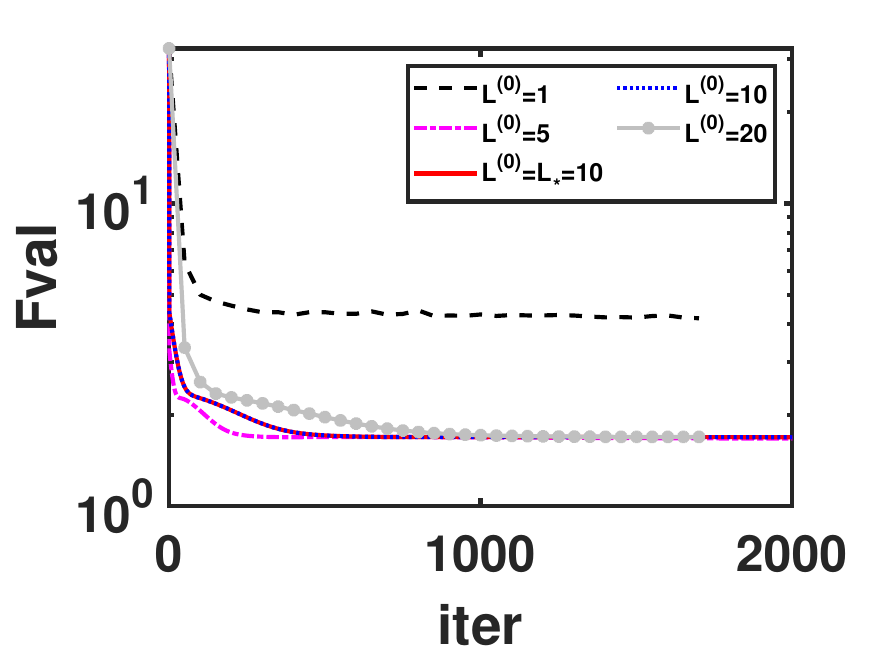}}\\
		\subfloat[\scriptsize $\beta=1/N$, E.g. 2]{\includegraphics[width=42mm]{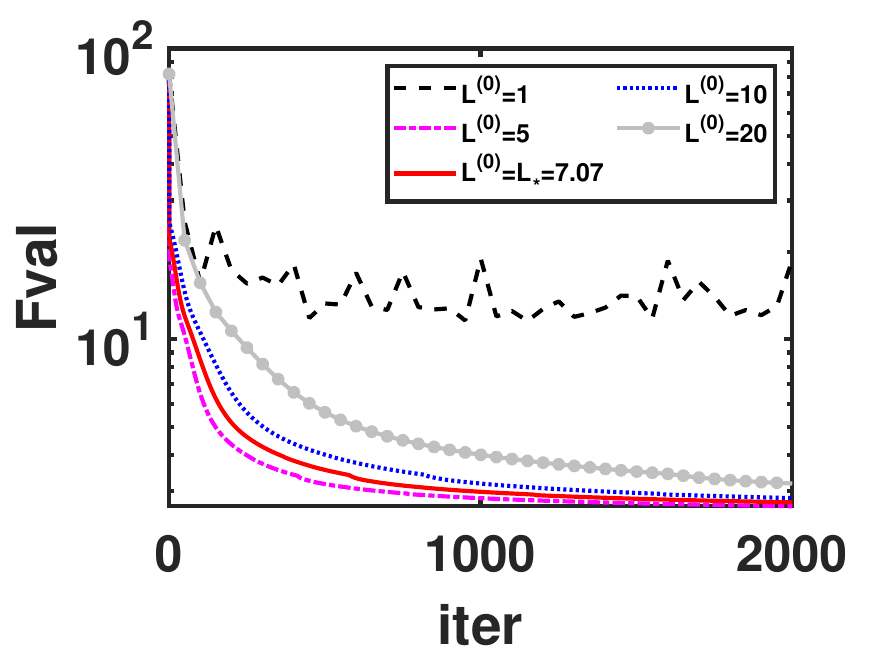}}	
		\subfloat[\scriptsize $\beta=1$, E.g. 2]{\includegraphics[width=42mm]{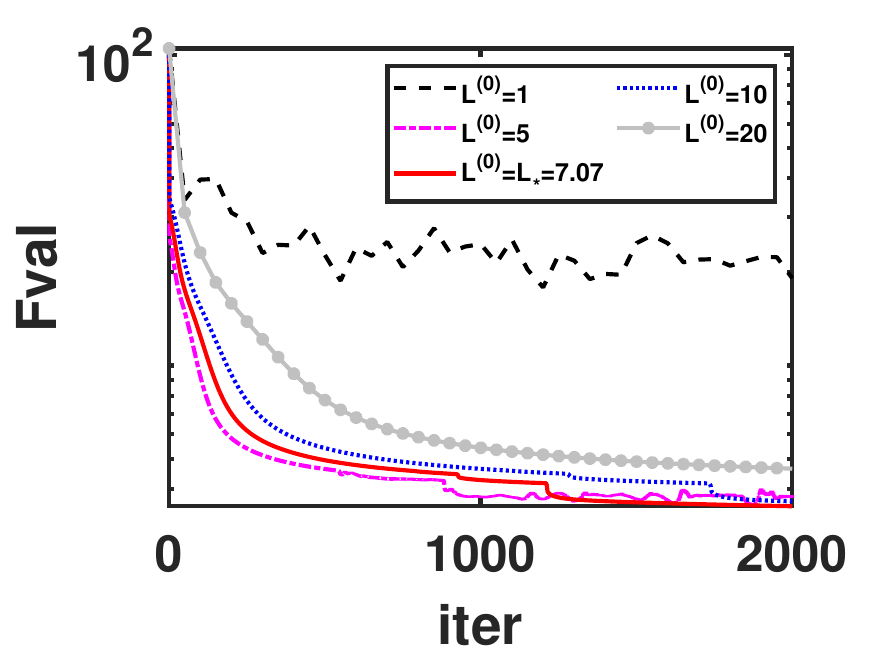}}	
		\subfloat[\scriptsize $\beta=10$, E.g. 2]{\includegraphics[width=42mm]{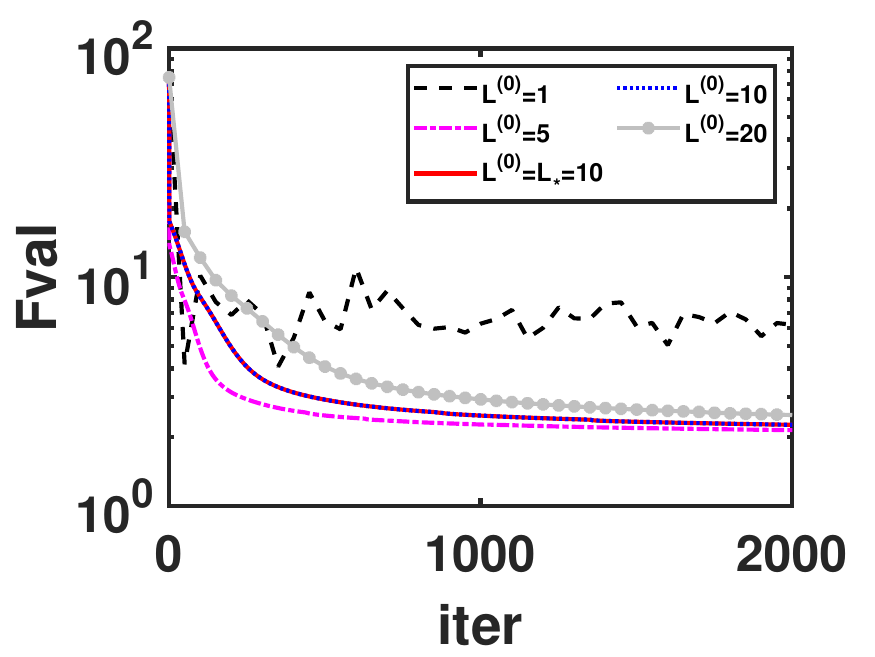}}	\\
		\subfloat[\scriptsize $\beta=1/N$, E.g. 3]{\includegraphics[width=42mm]{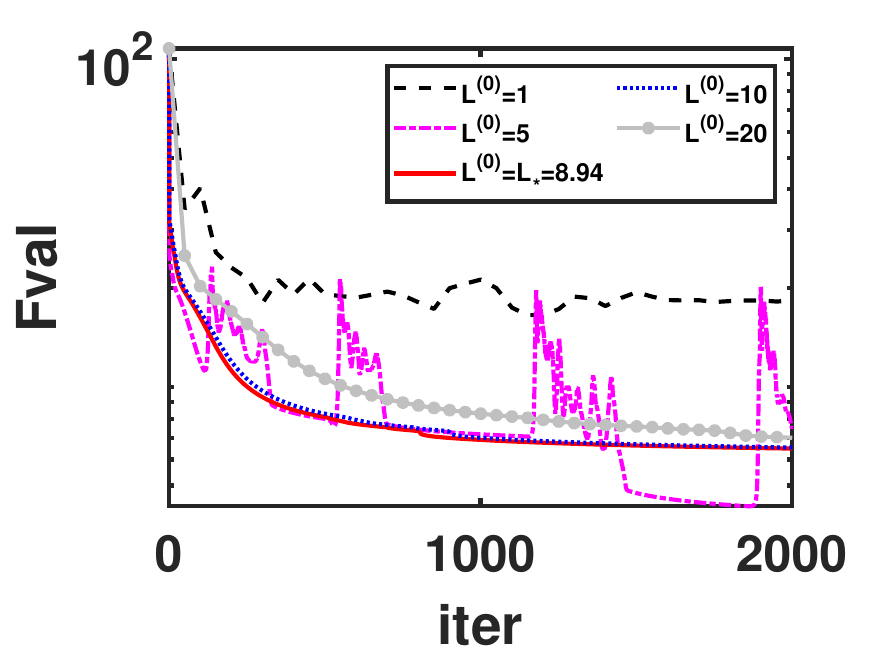}}	
		\subfloat[\scriptsize $\beta=1$, E.g. 3]{\includegraphics[width=42mm]{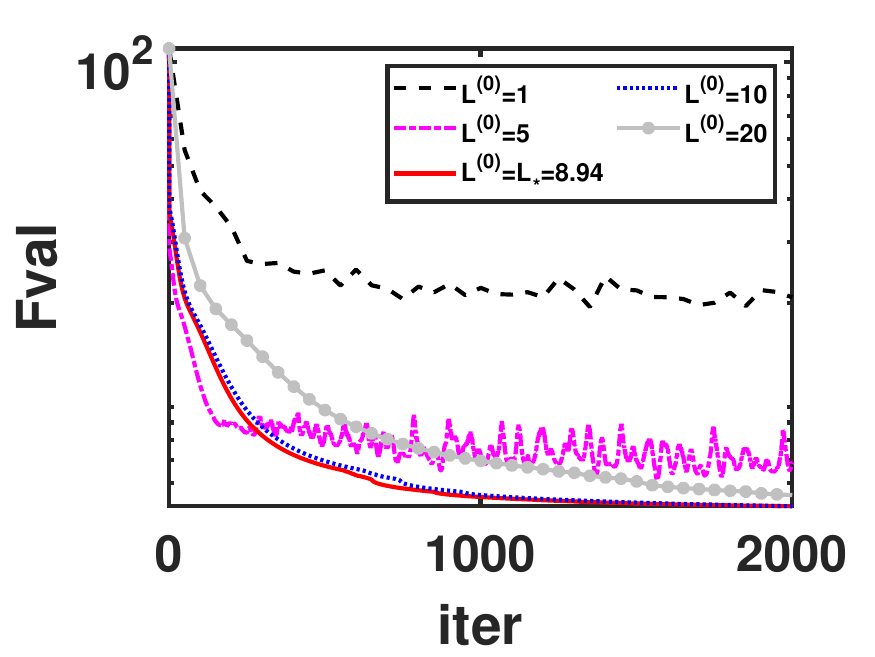}}	
		\subfloat[\scriptsize $\beta=10$, E.g. 3]{\includegraphics[width=42mm]{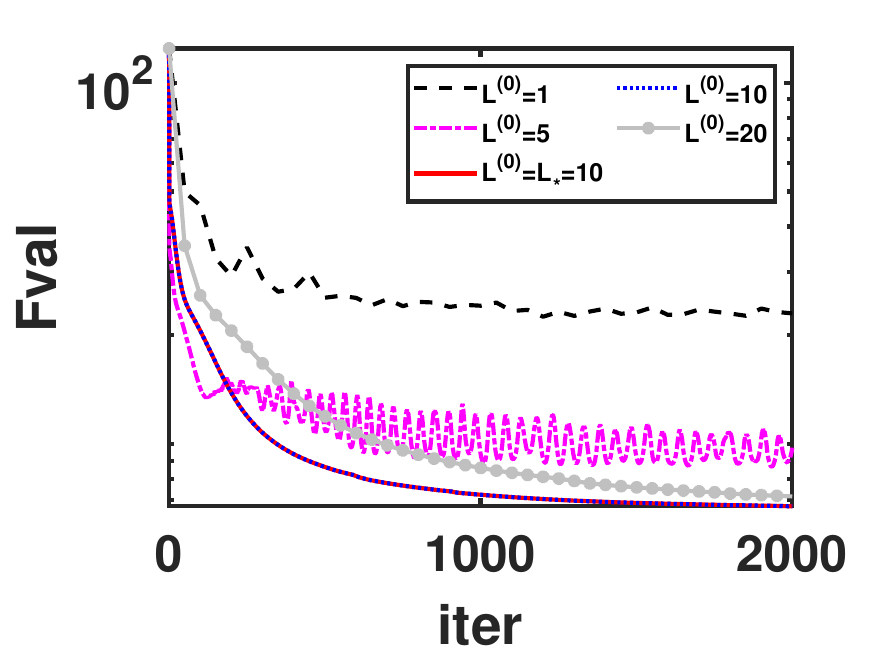}}	\\
		\subfloat[\scriptsize $\beta=1/N$, E.g. 7]{\includegraphics[width=42mm]{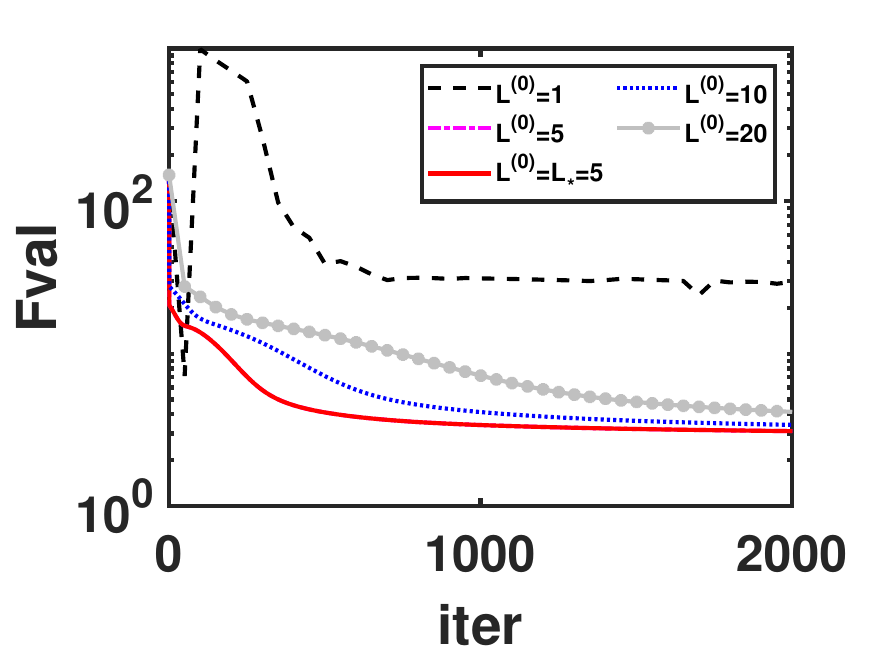}}	
		\subfloat[\scriptsize $\beta=1$, E.g. 7]{\includegraphics[width=42mm]{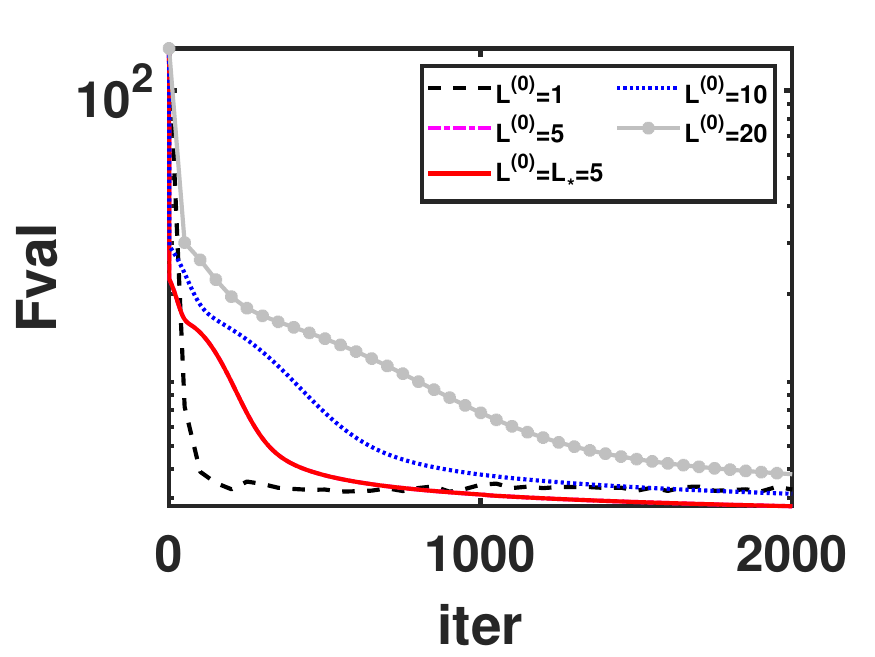}}	
		\subfloat[\scriptsize $\beta=10$, E.g. 7]{\includegraphics[width=42mm]{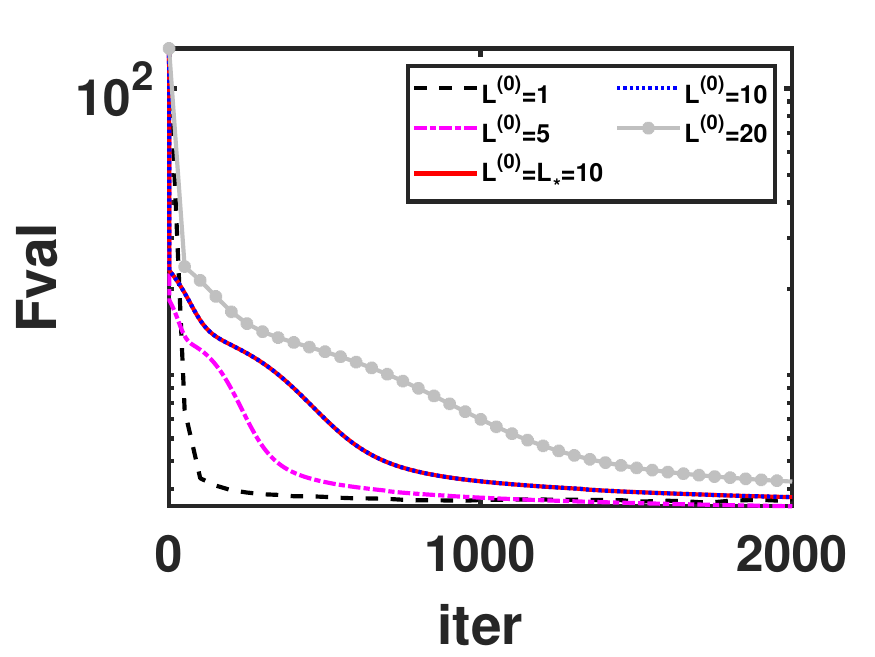}}	\\
		\caption{Comparison of SPG with varying $L^{(0)}$.}\label{fig:Lu1}
		\vspace{-5mm}
	\end{figure}
	
	Finally, we compare SPG with different choices of $\beta$ on a group of randomly generated data sets with $N_{\mathrm{test}}=0$, $\epsilon_0=0.05$ and data type 1. We choose $\beta$ from
	the set $\{\frac{1}{N}, \frac{10}{N}, \frac{1}{10N}, 1, 10\}$. 
	We record how TrainErr, the FVal and the FeasVi decrease
	through the iteration. {The numerical results with parameter combinations E.g. 4 and E.g. 8}
	are illustrated in Figures \ref{fig:beta1} and \ref{fig:beta2}, respectively.
	We can learn from these two figures that the
	bigger $\beta$ always leads to slower convergence.
	
	\begin{figure}[htbp!]
		\vspace{-5mm}
		\centering
		\setcounter{subfigure}{0}
		\subfloat{\includegraphics[width=42mm]{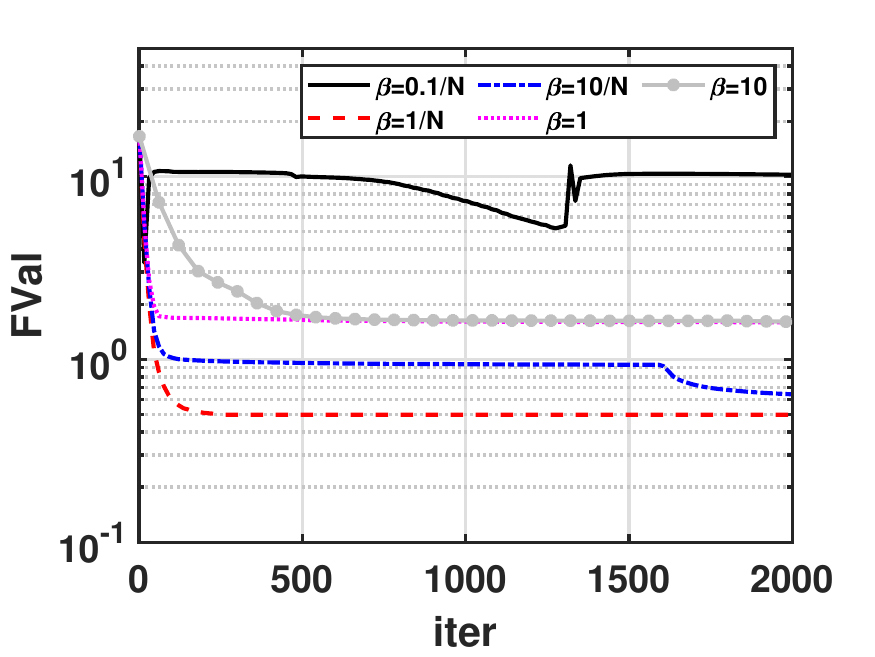}}
		\subfloat{\includegraphics[width=42mm]{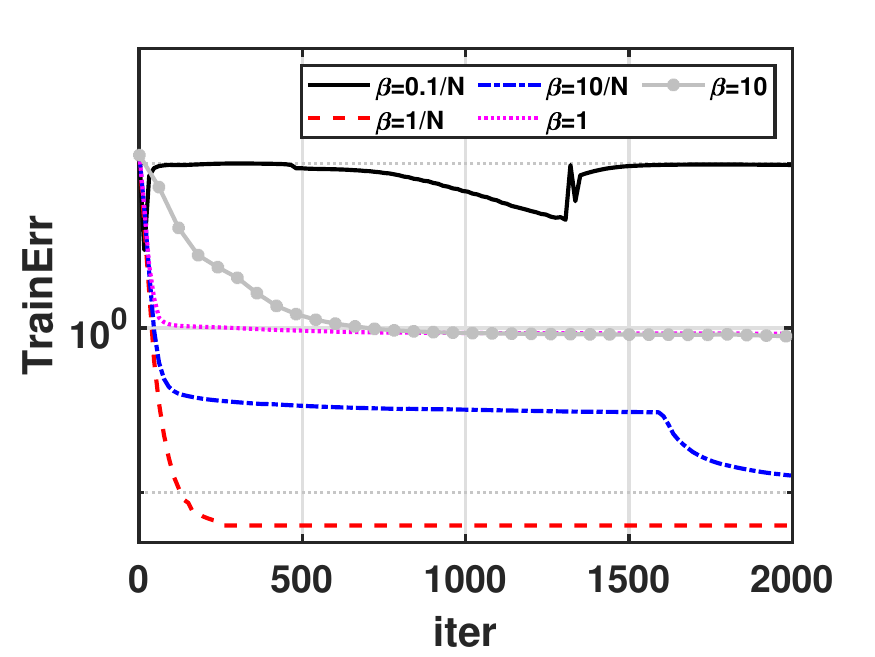}}
		\subfloat{\includegraphics[width=42mm]{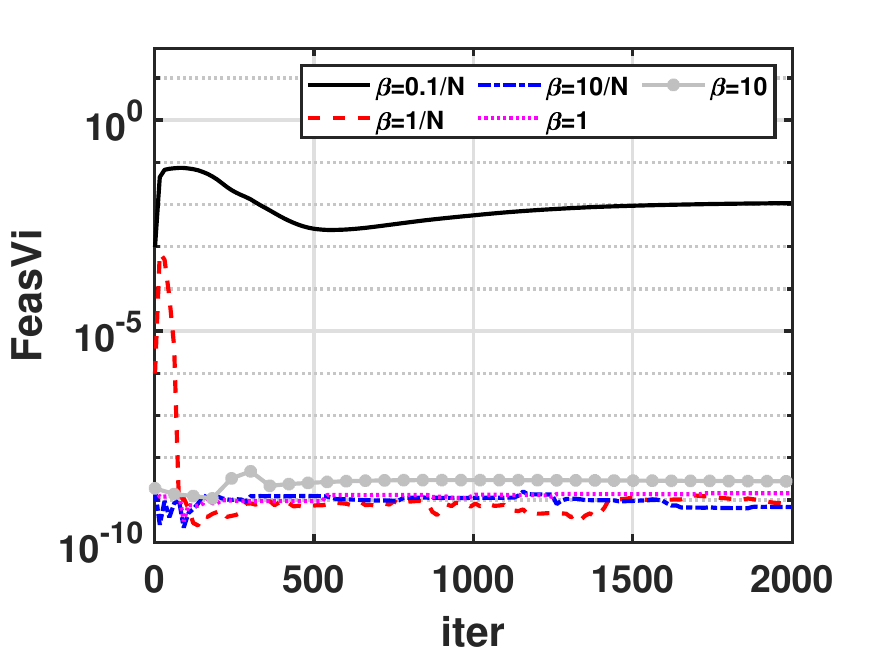}}
		\caption{Comparison of SPG with varying $\beta$, {\color{black} $L^{(0)}=L_{\star}$}.}\label{fig:beta1}
		\vspace{-5mm}
	\end{figure}
	
	\begin{figure}[htbp!]
		\vspace{-5mm}
		\centering
		\setcounter{subfigure}{0}
		\subfloat{\includegraphics[width=42mm]{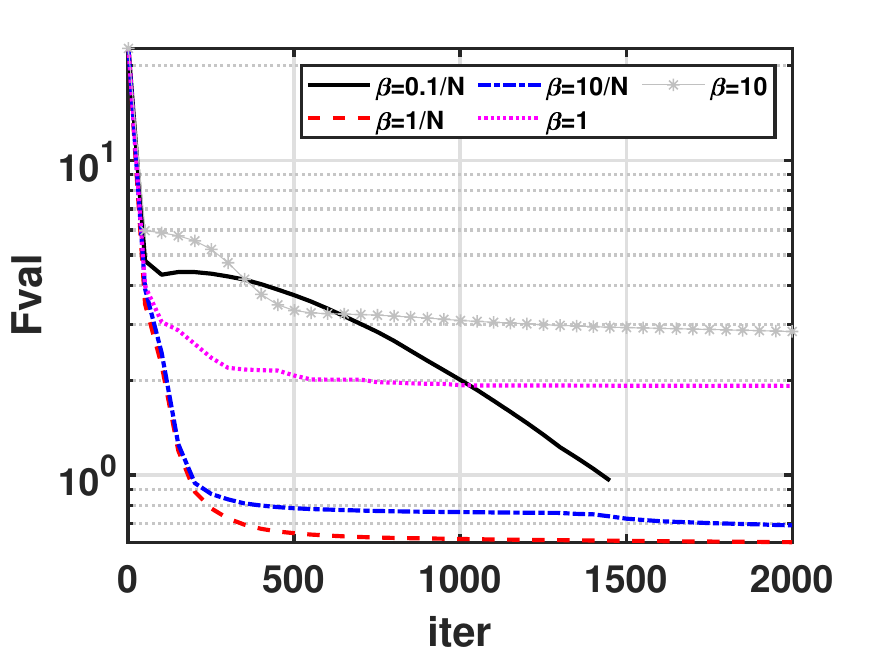}}
		\subfloat{\includegraphics[width=42mm]{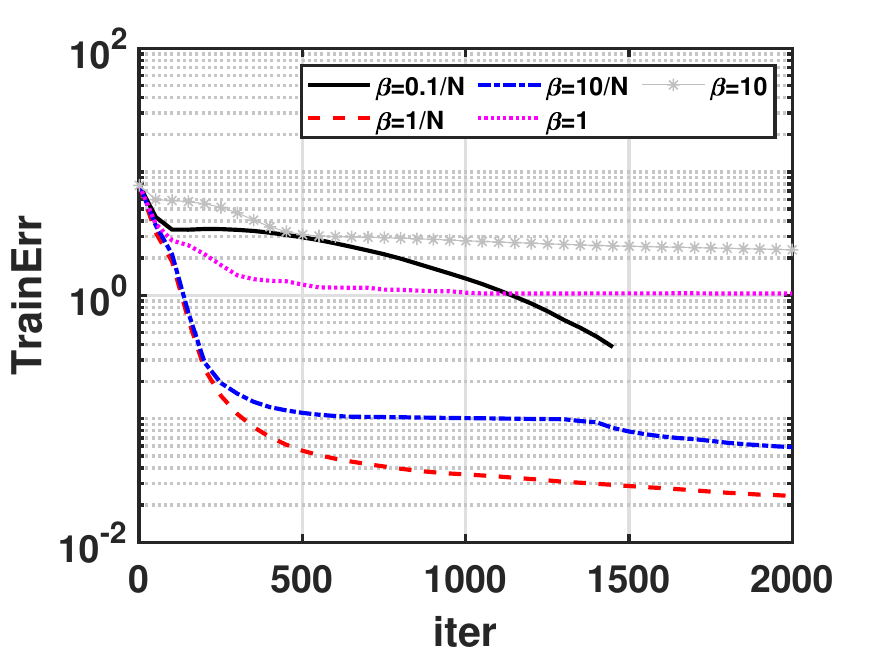}}
		\subfloat{\includegraphics[width=42mm]{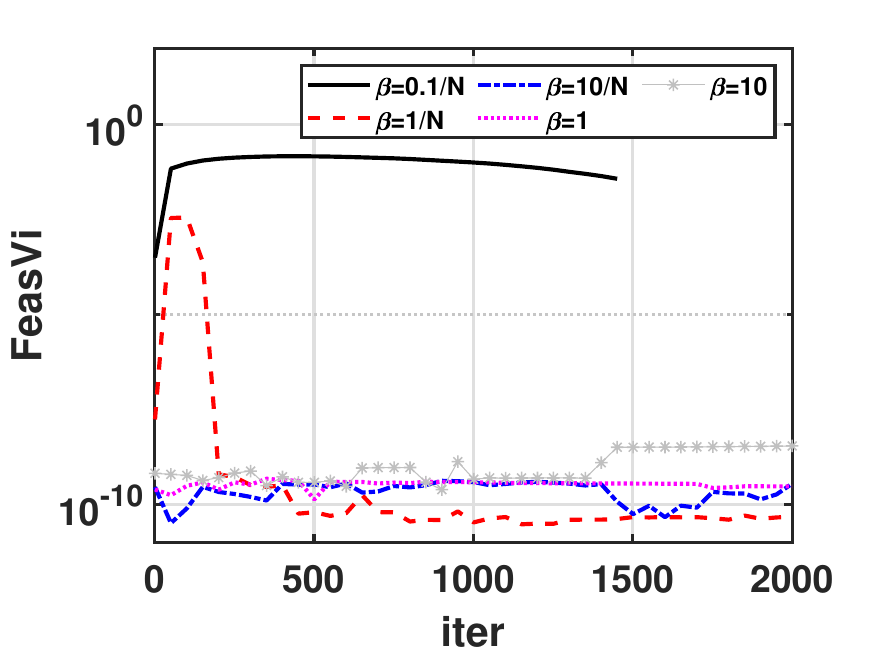}}\\
		\caption{Comparison of SPG with varying $\beta$, {\color{black} $L^{(0)}=L_{\star}$}.}\label{fig:beta2}
		\vspace{-5mm}
	\end{figure}
	
	\subsection{Comparison with Other Methods}
	
	In this subsection, we compare SPG with the existing SGD-based approaches.
	
	We choose two groups of data sets
	randomly generated by the two data types described in Subsection 4.1, and the
	numerical results are demonstrated in {Figures \ref{fig:diffalg1} and \ref{fig:diffalg2},} respectively.
	Here, all algorithms start from the same random initial guess. We set $\epsilon_0=0.05$ and $N_{\mathrm{test}}=30$. {The chosen parameter combinations are given in the subtitles of 
		{these two figures}.}

	\begin{figure}[htbp!]
		\vspace{-5mm}
		\centering
		\setcounter{subfigure}{0}
		\subfloat[E.g. 5]{\includegraphics[width=42mm]{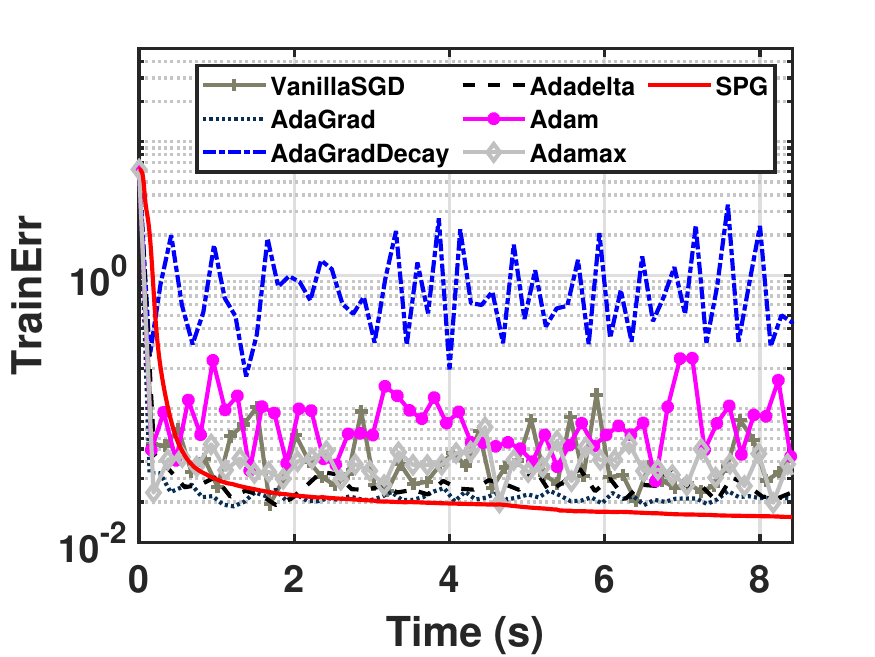}}
		\subfloat[E.g. 6]{\includegraphics[width=42mm]{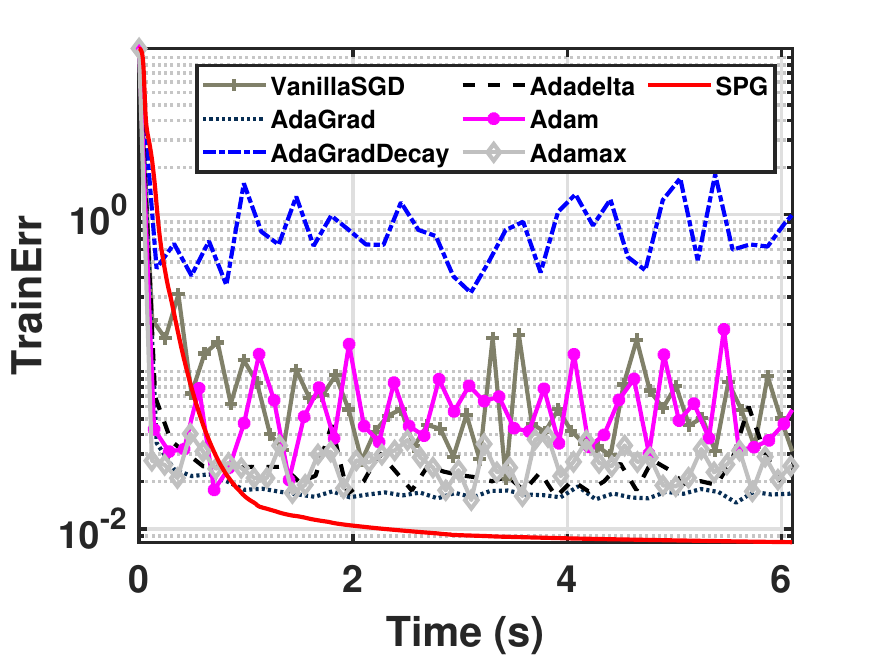}}
		\subfloat[E.g. 9]{\includegraphics[width=42mm]{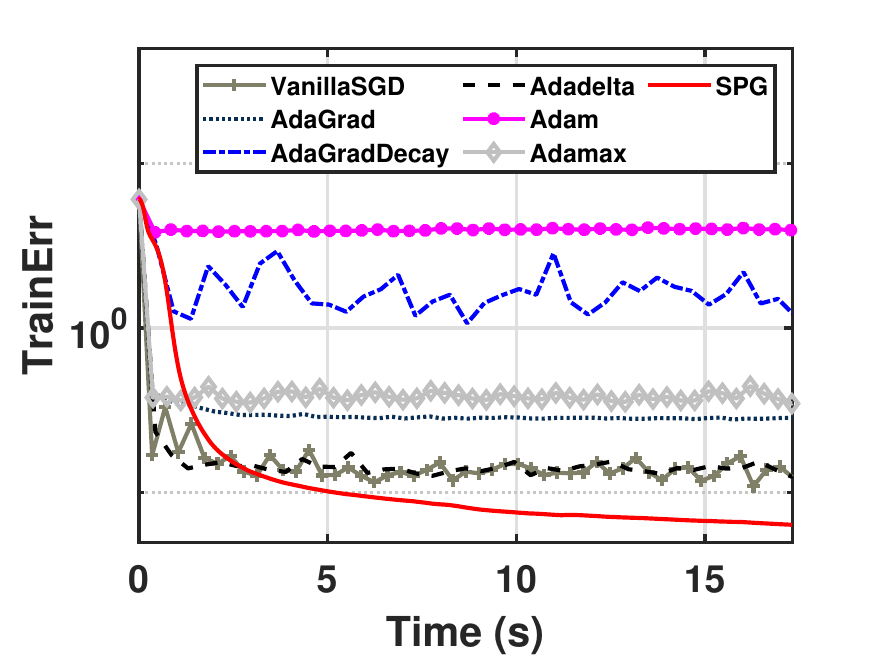}}
		\\
		\subfloat[E.g. 5]{\includegraphics[width=42mm]{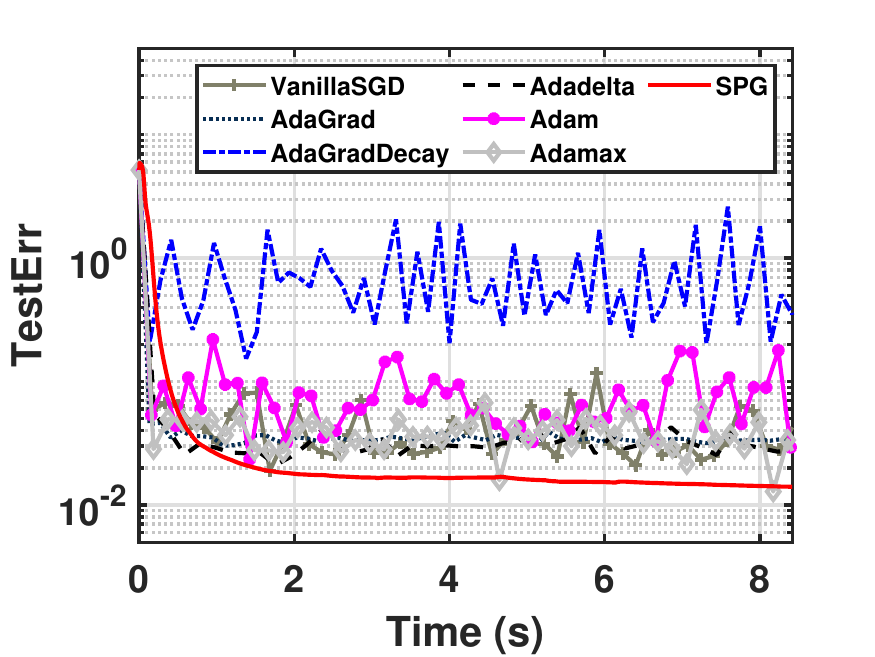}}
		\subfloat[E.g. 6]{\includegraphics[width=42mm]{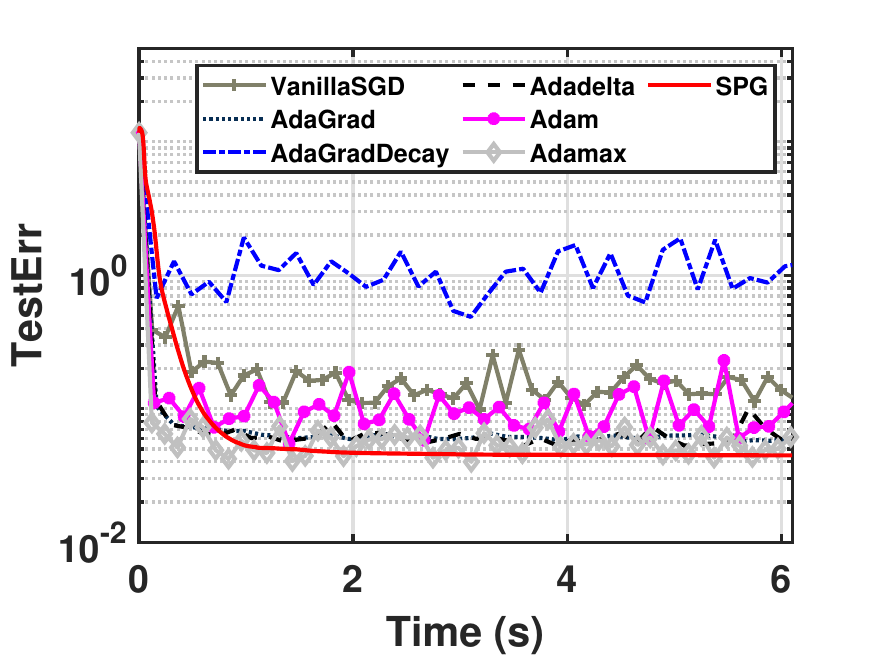}}
		\subfloat[E.g. 9]{\includegraphics[width=42mm]{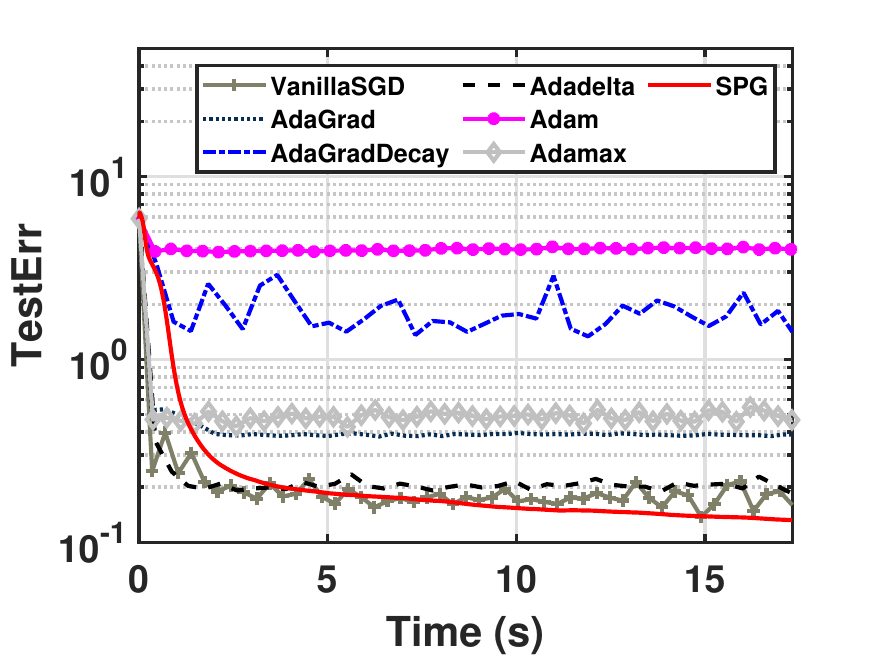}}
		\caption{Comparison among SPG and SGD-based approaches with data type 1.}\label{fig:diffalg1}
		\vspace{-5mm}
	\end{figure}
	
	\begin{figure}[htbp!]
		\vspace{-5mm}
		\centering
		\setcounter{subfigure}{0}
		\subfloat[E.g. 5]{\includegraphics[width=42mm]{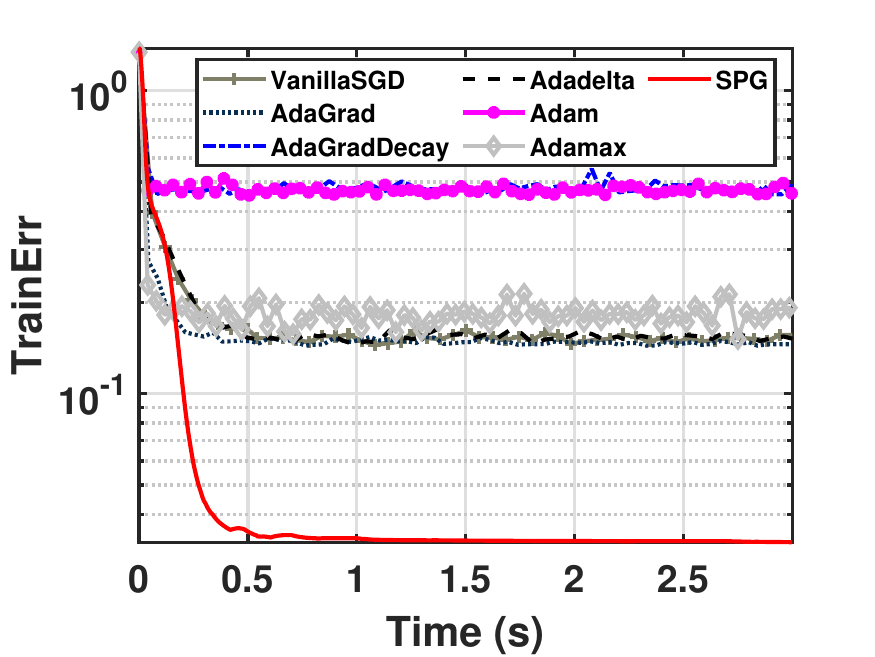}}
		\subfloat[E.g. 6]{\includegraphics[width=42mm]{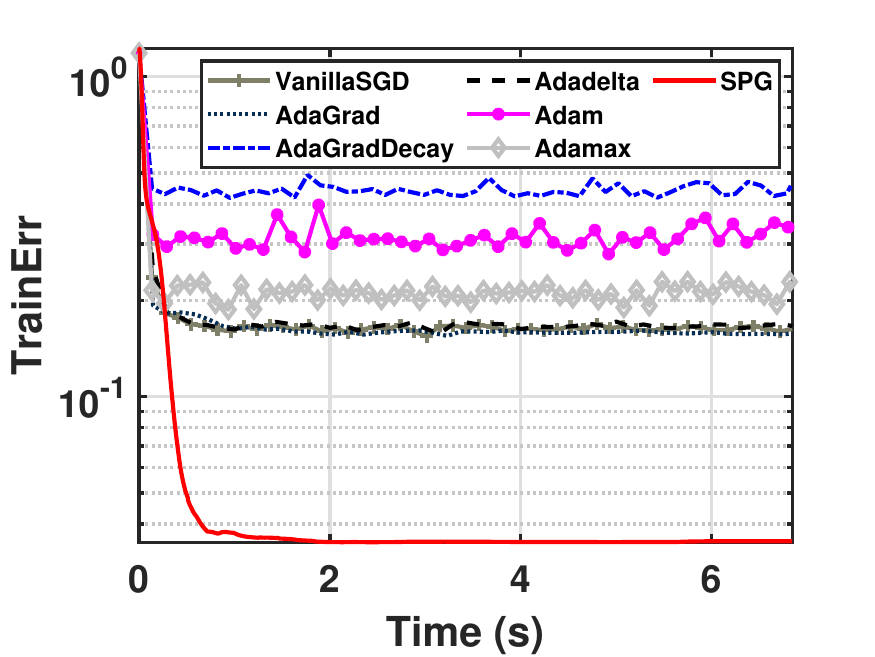}}
		\subfloat[E.g. 9]{\includegraphics[width=42mm]{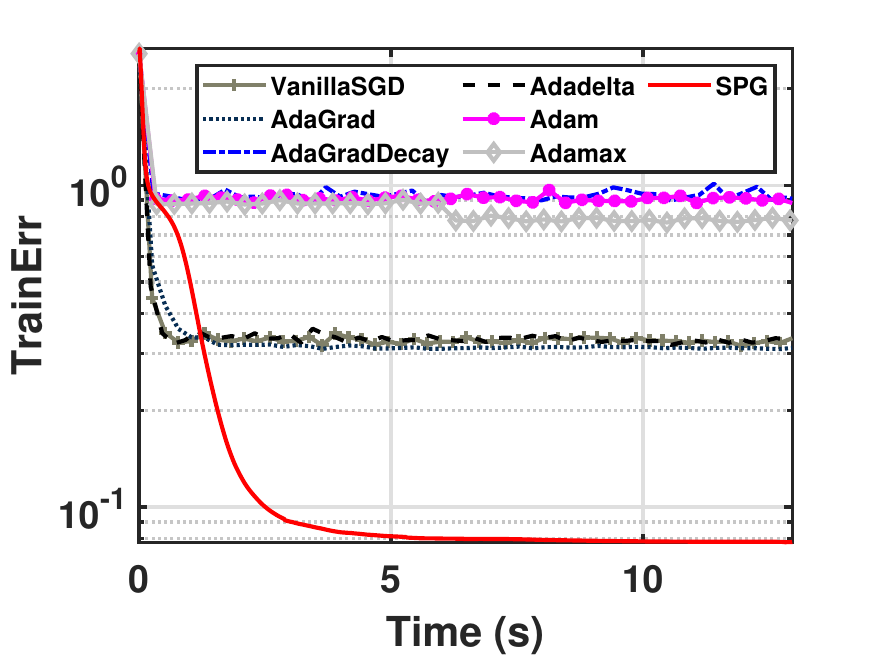}}
		\\
		\subfloat[E.g. 5]{\includegraphics[width=42mm]{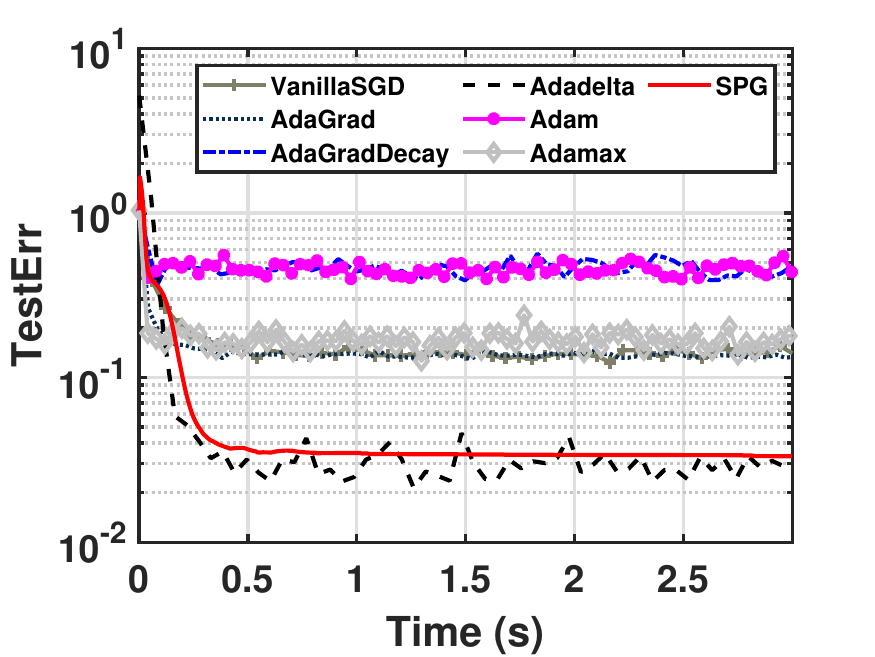}}
		\subfloat[E.g. 6]{\includegraphics[width=42mm]{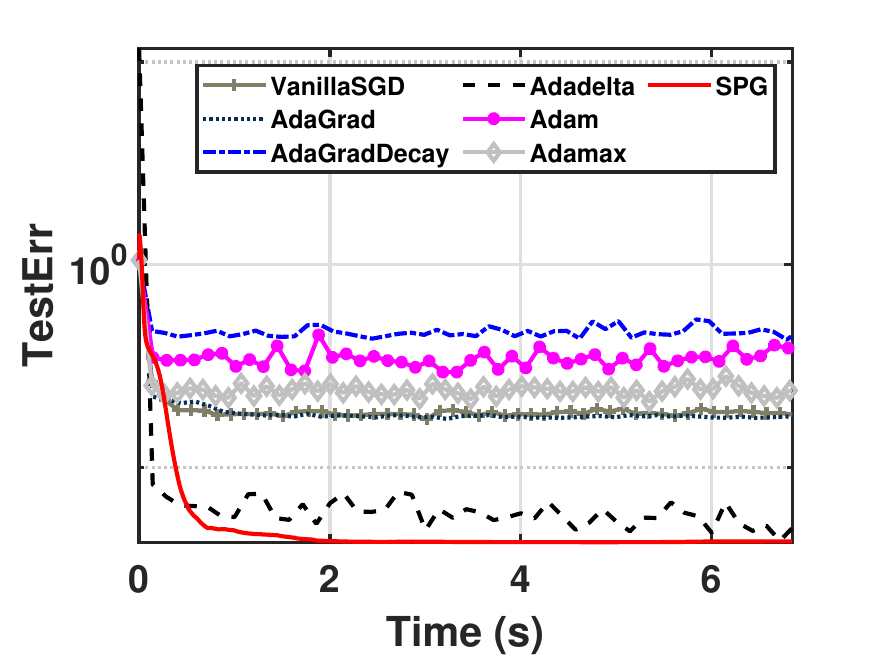}}
		\subfloat[E.g. 9]{\includegraphics[width=42mm]{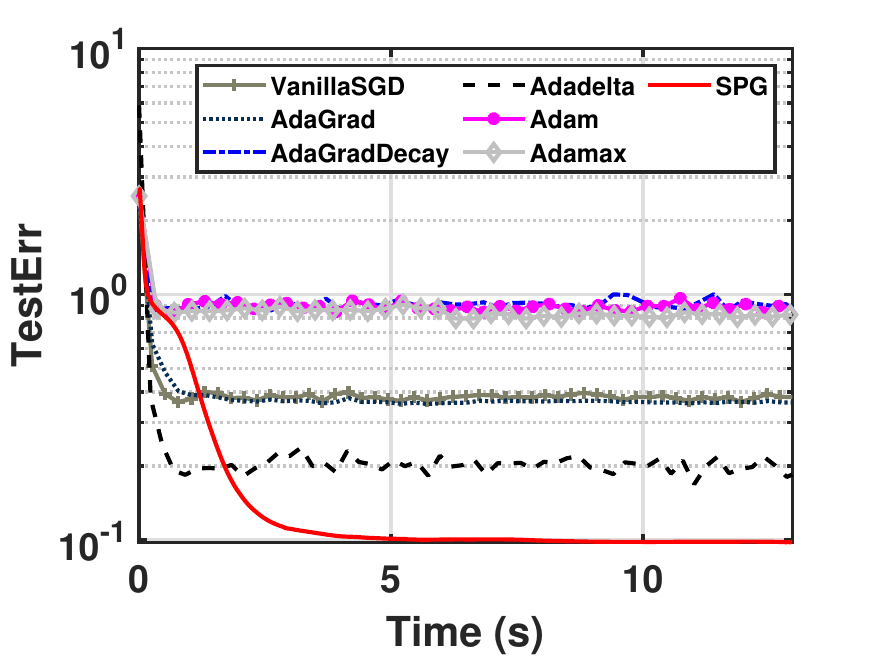}}
		\caption{Comparison among SPG and SGD-based approaches with data type 2.}\label{fig:diffalg2}
		\vspace{-5mm}
	\end{figure}
	
	We can learn from {Figures \ref{fig:diffalg1} and \ref{fig:diffalg2}} 
	that SPG reduces the training and test errors slower 
	than some other algorithms at very beginning.
	It can reach a lower residual than the others finally.
	
	It can also be observed from {these two figures} that Adadelta
	outperforms the other SGD-based approaches in the aspects of efficiency and solution quality.
	Therefore, we consider to use Adadelta as a
	pre-process to accelerate SPG. More specifically, we first run Adadelta for $1000$ epochs
	and then switch to SPG. We call the consequent hybrid algorithm
	SPG-ADA.
	In the following tests, such pre-processing will be the default setting of SPG.
	
	We select a new group of data sets randomly
	generated by data type 1 with $N=1000$, $N_{\mathrm{test}}=300$, $\epsilon_0=0.05$,
	and different combinations of $N_0$ and $N_1$.
	The iteration number of Adadelta is set as $10000$.
	We run SPG-ADA and Adadelta $100$ times and
	record the average output values in~Table \ref{tab:s2s100}.
	We can learn from~Table \ref{tab:s2s100} that SPG-ADA
	can obtain better training and test errors than Adadelta in comparable CPU time.
	
	\begin{table}\footnotesize
		\centering
		\caption{Comparison between SPG-ADA and Adadelta with $N=1000$.}\label{tab:s2s100}
		\begin{tabular}{|c|c|c|c|c|c|c|c|c|c|c|}
			\hline
			&&\multicolumn{4}{|c|}{SPG-ADA}&\multicolumn{3}{|c|}{Adadelta}\\
			\hline
			$N_0$ & $N_1$ & TrainErr & TestErr & FeaErr&  Time  & TrainErr & TestErr  & Time \\
			\hline
			5&20 & 3.297e-02 & 3.636e-02 & 1.234e-11& 8.758 & 5.518e-02 & 5.847e-02& 3.044 \\
			\hline
			5&30 & 2.974e-02 & 3.103e-02 & 5.599e-12& 10.592 & 5.470e-02 & 5.566e-02& 3.623\\
			\hline
			5&40 & 2.960e-02 & 3.200e-02 & 7.238e-12& 15.206 & 5.474e-02 & 5.632e-02& 3.786 \\
			\hline
			10&40 & 6.708e-02 & 7.727e-02 & 1.140e-11& 19.184 & 1.257e-01 & 1.341e-01& 5.590\\
			\hline
			10&60 &6.867e-02 & 7.863e-02 & 5.138e-11& 22.599 & 1.348e-01 & 1.436e-01& 6.149\\
			\hline
			10&80 & 8.105e-02 & 9.057e-02 & 8.814e-11& 25.701 & 1.364e-01 & 1.441e-01& 7.169  \\
			\hline
			20&80 & 1.824e-01 & 2.200e-01 & 3.020e-12& 33.962 & 3.766e-01 & 4.265e-01& 8.992 \\
			\hline
			20&120 & 1.135e-01 & 2.611e-01 & 3.634e-12& 38.191 & 4.051e-01 & 4.566e-01& 12.275 \\
			\hline
			20&160 & 1.946e-01 & 2.380e-01 & 2.181e-12& 72.942 & 3.746e-01 & 4.240e-01& 20.268\\
			\hline
		\end{tabular}
	\end{table}

	\subsection{Tests on MNIST}

	In this subsection, we investigate the numerical comparison among
	SPG-ADA and Adadelta in solving problems arising from the real data set MNIST.
	
	Firstly, we set $N=100$ and $N_1=500$.
	We can find the reconstruction results corresponding to the
	autoencoder solutions obtained by SPG-ADA and Adadelta
	in~Figure \ref{fig:consalg} (a)-(b), respectively.
	We can conclude that SPG-ADA can reach the comparable
	reconstruction quality as Adadelta.
	In addition, we also present the reconstruction result derived by Adam, 
	{as a failure case}. 
	Therefore, we exclude Adam in the {last numerical experiment}.
	
	\begin{figure}[htbp!]	
		\vspace{-5mm}
		\centering
		\setcounter{subfigure}{0}
		\subfloat[SPG-ADA]{\includegraphics[width=60mm]{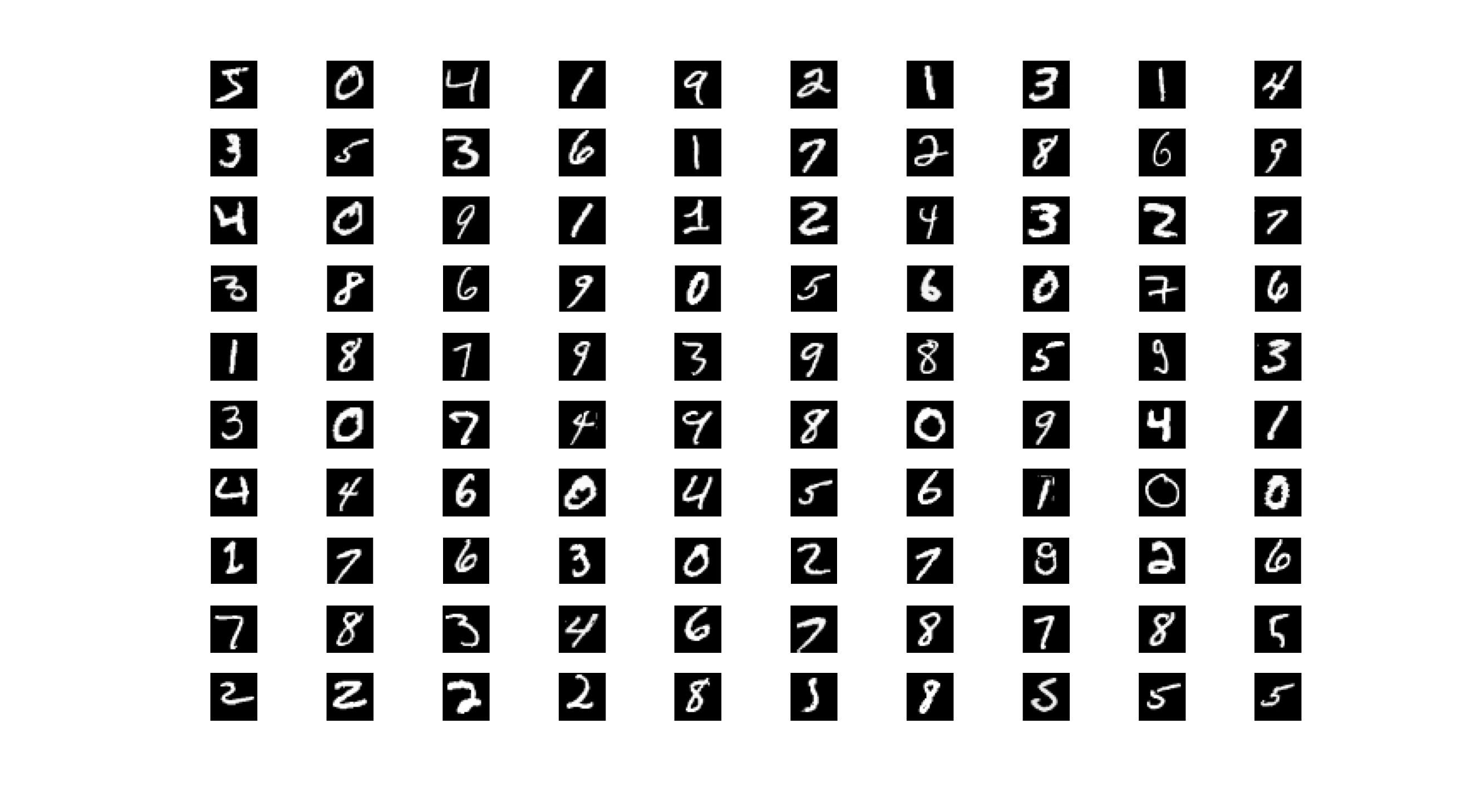}}
		\subfloat[Adadelta]{\includegraphics[width=60mm]{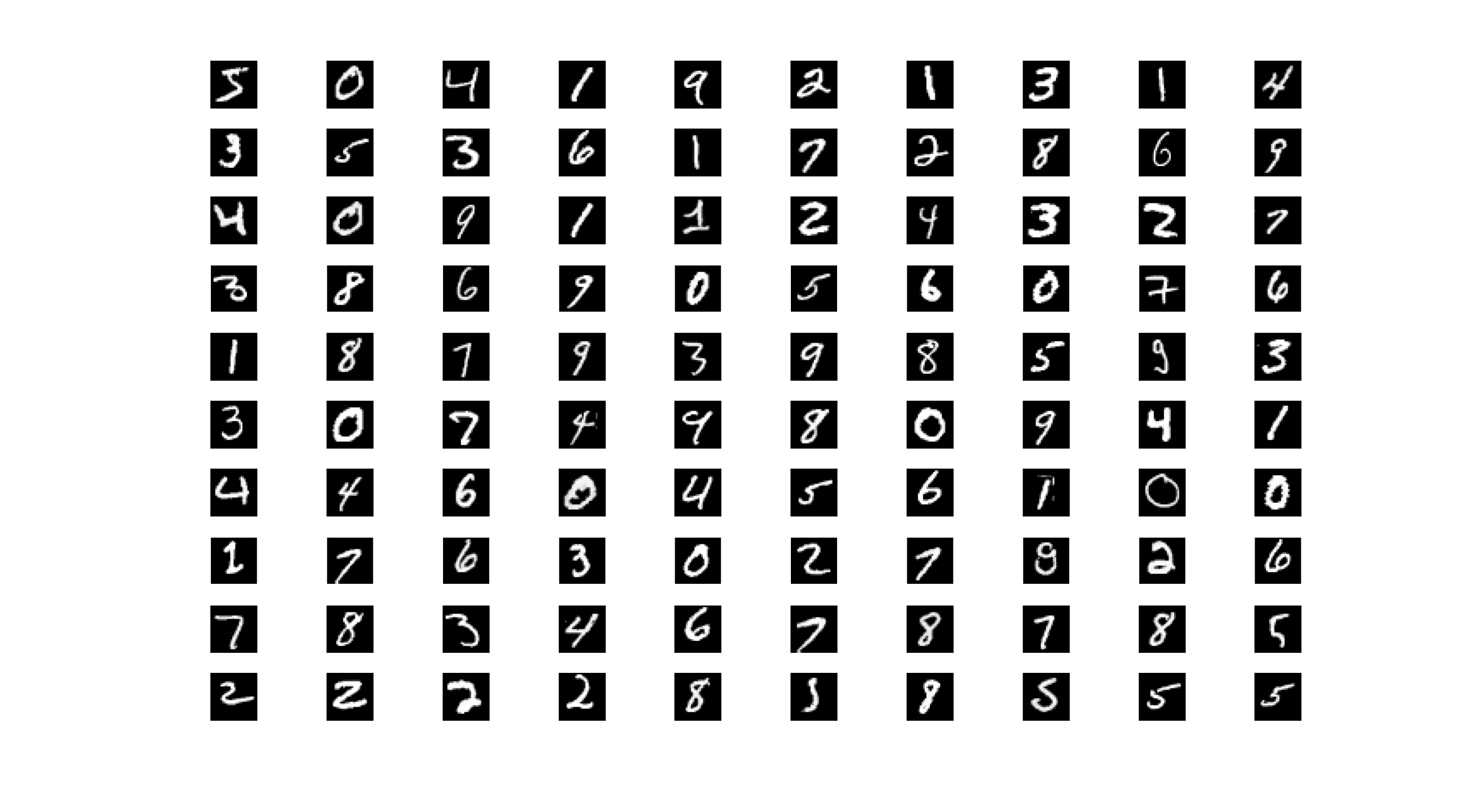}}\\
		\subfloat[Adam]{\includegraphics[width=60mm]{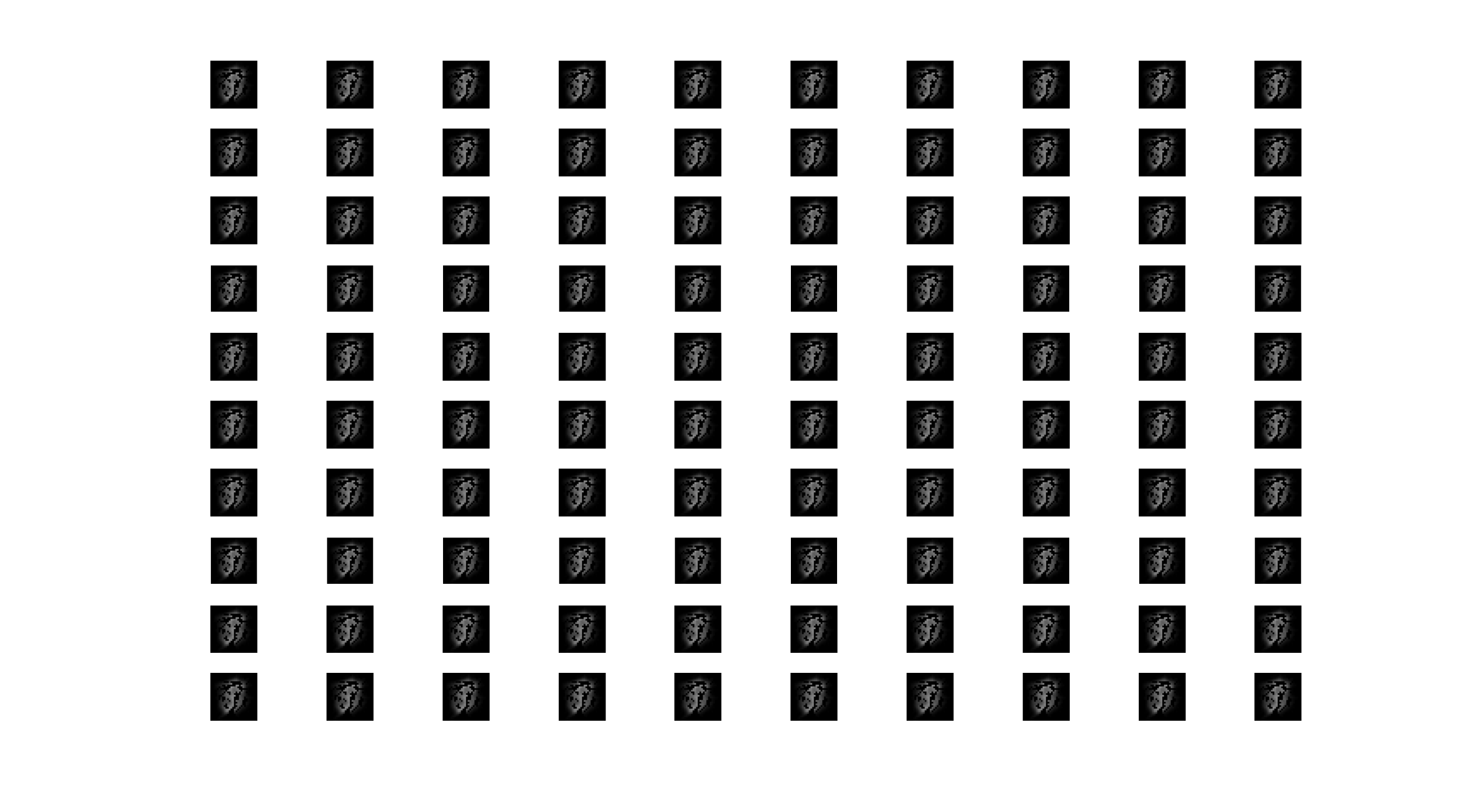}}
		\caption{Reconstruction of MNIST by SPG-ADA, Adadelta and Adam.}
		\label{fig:consalg}
		\vspace{-5mm}
	\end{figure}
	{Finally,} we demonstrate how the training and test errors
	decrease through the iterations of SPG-ADA and Adadelta.
	We select different combinations of $N$ and $N_1$.
	The results are illustrated in~Figure \ref{fig:consalg2}.
	We can learn that SPG-ADA is much more
	robust and can always find better solutions.
	\begin{figure}[htbp!]		
		\vspace{-5mm}
		\centering
		\setcounter{subfigure}{0}
		\subfloat[$N=100$, $N_1=500$]{\includegraphics[width=42mm]{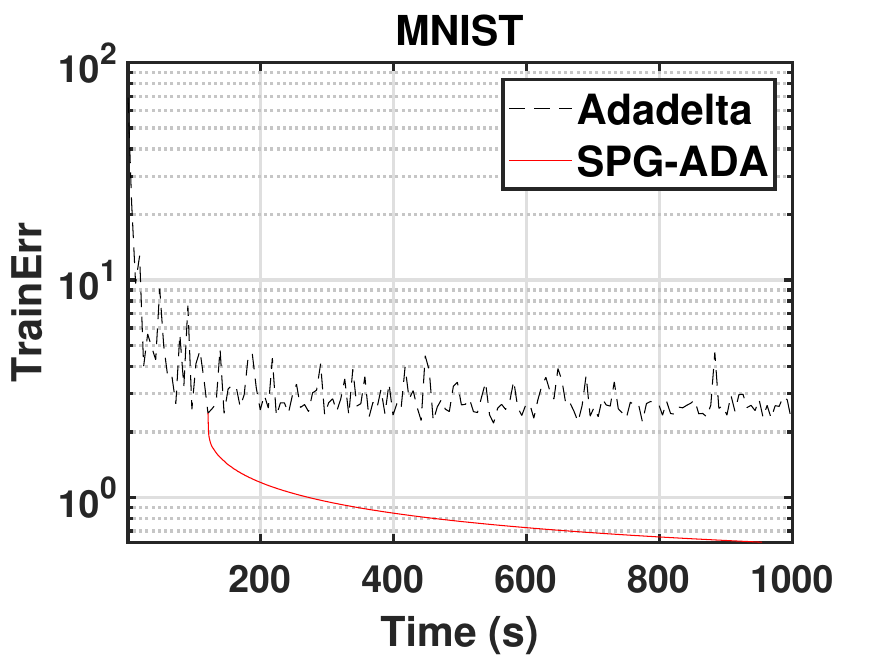}}
		\subfloat[$N=1000$, $N_1=1000$]{\includegraphics[width=42mm]{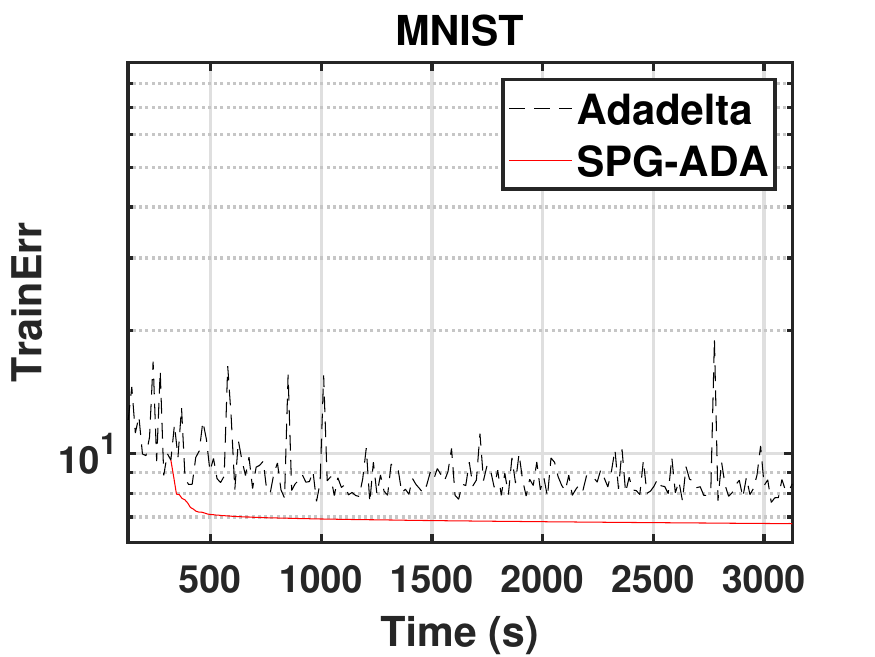}}
		\subfloat[$N=10000$, $N_1=2000$]{\includegraphics[width=42mm]{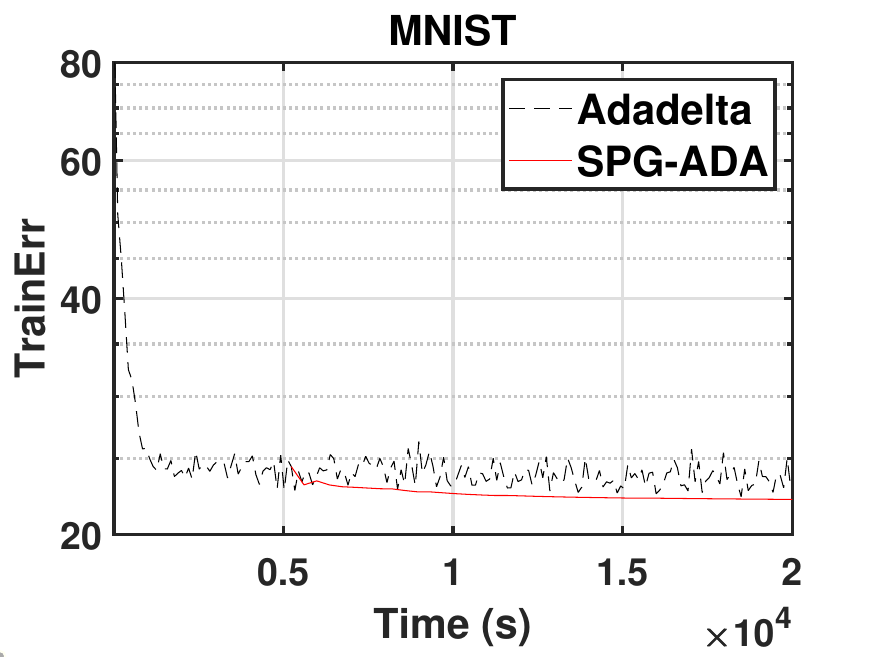}}
		\caption{Comparison between SPG-ADA and Adadelta on MNIST.}\label{fig:consalg2}
		\vspace{-5mm}
	\end{figure}
	
	\section{Conclusion}

	The regularized minimization model~\eqref{eq:reguauto2}
	using the ReLU activation function has been extensively applied for the autoencoders. However, the set of global minimizers of
	the model is generally unbounded. Existing algorithms cannot guarantee to generate bounded sequences with decreasing objective function values. In this paper,
	we propose the regularized minimization model with $l_1$-norm penalty~\eqref{eq:reguauto4} that
	has same global minimizers, local minimizer and d-stationary points with the regularized minimization model \eqref{eq:reguauto2}. Moreover, we develop the linearly constrained regularized minimization model with $l_1$-norm penalty \eqref{eq:reguauto7} which has a bounded solution set contained in the solution sets of \eqref{eq:reguauto2} and \eqref{eq:reguauto4}.
	We develop  a smoothing proximal gradient (SPG) algorithm to solve \eqref{eq:reguauto7}.
	We prove the sequence generated by the SPG algorithm is bounded and has a subsequence converging to a generalized d-stationary point of \eqref{eq:reguauto7}. We conduct comprehensive numerical experiments
	to verify the effectiveness, efficiency and robustness of the SPG algorithm.
	
	Finally, we mention that our results on the relationships among \eqref{eq:reguauto2}, \eqref{eq:reguauto4}, \eqref{eq:reguauto7} can be extended to the following three corresponding problems for training an $L$-layer DNN with ReLU activation functions, given input data $\{x_n\}_{n=1}^N$ and output data $\{y_n\}_{n=1}^N$. 
	\begin{equation}\label{eq:dnn21}
		\begin{aligned}
			\min_{z}\, & \, {\cal F}(z)+\mathcal{R}(z)\\
			\st   \, & \, u_{n,\ell}=(W_{\ell}u_{n,\ell-1}+b_{\ell})_+,  \quad n=1,\ldots,N, \, \ell=1,\ldots,L,
		\end{aligned}
	\end{equation}
	\begin{equation}\label{eq:dnn22}
		\begin{aligned}
			\min_{z}\, & \, {\cal F}(z)+\mathcal{R}(z)+\beta\sum_{n=1}^N\sum_{\ell=1}^Le\zz_{N_{\ell}}(u_{n,\ell}-(W_{\ell}u_{n,\ell-1}+b_{\ell})_+)\\
			\st  \, & \, u_{n,\ell}\geq(W_{\ell}u_{n,\ell-1}+b_{\ell})_+,  \quad n=1,\ldots,N, \, \ell=1,\ldots,L,
		\end{aligned}
	\end{equation}
	and
	\begin{equation}\label{eq:dnn23}
		\begin{aligned}
			\min_{z}\, & \, {\cal F}(z)+\mathcal{R}(z)+\beta\sum_{n=1}^N\sum_{\ell=1}^Le\zz_{N_{\ell}}(u_{n,\ell}-(W_{\ell}u_{n,\ell-1}+b_{\ell})_+)\\
			\st \, & \, u_{n,\ell}\geq (W_{\ell}u_{n,\ell-1}+b_{\ell})_+,\,\,\|b_{\ell}\|_{\infty}\le \alpha_{\ell},\quad n=1,\ldots,N, \, \ell=1,\ldots,L,
		\end{aligned}
	\end{equation}
	respectively, where $\alpha_{\ell}$ is a given constant, $u_{n,0}=x_n$, ${\cal F}(z)=\frac{1}{N}\sum^N_{n=1}\|(W_Lu_{n,L}+b_L)_+ - y_n\|^2_2$,
	${\cal R}(z)=\lambda_1\sum_{n=1}^N\sum^L_{\ell=1} e^T_{N_{\ell}} u_{n,l}
	+\lambda_2\sum^L_{\ell=1} ||W_\ell\|_F^2$, $$z=(\mathrm{vec}(W_1)\zz,\ldots,\mathrm{vec}(W_L)\zz,u_{1,1}\zz,u_{2,1}\zz,\ldots,u_{N,1}\zz,u_{1,2}\zz,\ldots,u_{N,L}\zz,b_1\zz,\ldots,b_L\zz)\zz,$$
	for all $\ell=1,2,\ldots,L$ and $n=1,2,\ldots, N$.
	{\color{black}However, the increasing number of layers results in more rapidly increasing number of variables which requires further development on the algorithm to maintain the numerical comparability to SGD-based approaches.}

	{\color{black}	\textbf{Acknowledgements}. The authors would like to thank the reviewers for their insightful comments and efforts towards improving our manuscript.}
	
	\small
	\bibliographystyle{siamplain}
	\bibliography{ref}
	
	\appendix
	\addcontentsline{toc}{chapter}{APPENDICES}

	\section{Proof of~Lemma \ref{lem:alg1}}\label{sec:bound}
	
	\begin{proof}
		
		$($a$)$
		{It follows from the updating formula \eqref{eq:mu} of $\mu$ and $L$, 
			the required relations \eqref{eq:LLL} and $\tau_1\tau_3\geq 1$ that \begin{equation}\label{eq:LLL1}
				\mu^{(k)}L^{(k)}\geq \mu^{(0)}L^{(0)} \geq \max\left\{6\lb_2N_1N_0+\frac{2}{\eta}(N_2 L_{\widetilde{\mathcal{H}}}+\lb_1N_1N),8\lb_2+L_{\nabla\widetilde{\mathcal{H}}}\right\}
			\end{equation}
			holds for all $k = 0, 1, \ldots$.}
		
		{Next, we use the} mathematical induction to prove the facts that $\{z^{(k)}\}\subset \Omega_\theta\cap {\cal Z}$ and $\{\widetilde{\mathcal{O}}(z^{(k)},\mu^{(k)})\}$ is non-increasing.
		Naturally,
		we have $z^{(0)}\in \Omega_\theta\cap {\cal Z}$. Then we suppose that $z^{(k)}\in \Omega_\theta\cap {\cal Z}$ and $\widetilde{\mathcal{O}}(z^{(l)},\mu^{(l)})\leq\widetilde{\mathcal{O}}(z^{(l-1)},\mu^{(l-1)})$ hold for all $l=1,2,\ldots,k$.
		
		We deduce from $z^{(k)}\in \Omega_\theta\cap {\cal Z}$ and the proof {of} Lemma \ref{thm:nonem} that 
		{$\|b^{(k)}\|_{\infty}\leq \alpha$ and $\|z^{(k)}-[0,0,(b^{(k)})\zz]\zz\|_{\infty}\leq\eta$.
			If $\|z^{(k+1)}-[0,0,(b^{(k)})\zz]\zz\|_{\infty}> 2\eta$,} {we immediately have $\bar{\eta}:=\|z^{(k+1)}-z^{(k)}\|_{\infty} >\eta$. Then, it holds that}
		\begin{align*}
			&\left\langle\nabla_z \widetilde{\mathcal{H}}( z^{(k)},\mu^{(k)}), z^{(k+1)}- z^{(k)}\right\rangle+\mathcal{R}( z^{(k+1)}) - \mathcal{R}( z^{(k)})+\frac{L^{(k)}}{2}\| z^{(k+1)}- z^{(k)}\|_2^2\\
			\geq& -N_2 \frac{L_{\widetilde{\mathcal{H}}}}{\mu^{(k)}}\bar{\eta}+\frac{L^{(k)}}{2}\bar{\eta}^2+\mathcal{R}( z^{(k+1)}) - \mathcal{R}( z^{(k)})
			\\\geq &  -N_2 \frac{L_{\widetilde{\mathcal{H}}}}{\mu^{(k)}}\bar{\eta}+\frac{L^{(k)}}{2}\bar{\eta}^2-\lb_1N_1N \bar{\eta}-\lb_2N_0N_1\bar{\eta}\max_{j\in\{1,2,\ldots,N_1\}}\|W_{\cdot,j}^{(k+1)}+W_{\cdot,j}^{(k)}\|_{\infty}
			\\\geq &  -N_2 \frac{L_{\widetilde{\mathcal{H}}}}{\mu^{(k)}}\bar{\eta}+\frac{L^{(k)}}{2}\bar{\eta}^2-\lb_1N_1N \bar{\eta}-\lb_2N_0N_1\bar{\eta}\left(\bar{\eta}+2\eta\right)>0,
		\end{align*}
		where the second inequality comes from the definition of $\mathcal{R}$, the third inequality {results from the relations} $\|W_{\cdot,j}^{(k+1)}+W_{\cdot,j}^{(k)}\|_{\infty}\leq \|W_{\cdot,j}^{(k+1)}-W_{\cdot,j}^{(k)}\|_{\infty}+2\|W_{\cdot,j}^{(k)}\|_{\infty}\leq \bar{\eta}+2\eta$ for all $j=1,2,\ldots,N_1$, and the last inequality comes from $0<\mu^{(k)}<1$ and~\eqref{eq:LLL1}.
		This leads to a contradiction, since $z^{(k+1)}$ is a solution of subproblem \eqref{eq:updateWBV1}. Hence, we have $\|z^{(k+1)}\|_{\infty}\leq \max\{\alpha,2\eta\}$.
		
		By the KKT condition of~\eqref{eq:updateWBV1}, there exists {a} nonnegative vector $\gamma^{{(k+1)}} \in \mathbb{R}^\nu$ such that
		\begin{equation}\label{eq:KKTsub1}
			\left\{
			\begin{aligned}
				&\nabla_{z} \widetilde{\mathcal{H}}(z^{(k)}, \mu)+\nabla \mathcal{R}(z^{(k+1)})+A\zz \gamma^{{(k+1)}}+L^{(k)}(z^{(k+1)}-z^{(k)})=0,\\
				& A z^{(k+1)} \leq c,  (\gamma^{{(k+1)}})\zz(A z^{(k+1)}-c)=0.
			\end{aligned}\right.
		\end{equation}
		
		It follows from the inequality~\eqref{eq:LLL1}, the relations \eqref{eq:KKTsub1}, $(\gamma^{{(k+1)}})\zz  (A z^{(k)}-c)\leq 0$, $0<\mu^{(k)}\leq 1$ and the definition of $L_{\nabla\widetilde{\mathcal{H}}}$ that
		\begin{equation}\label{eq:1}
			\begin{aligned}
				& \widetilde{\mathcal{O}}( z^{(k+1)},\mu^{(k)})- \widetilde{\mathcal{O}}( z^{(k)},\mu^{(k)})
				\\\leq& \left\langle\nabla_z \widetilde{\mathcal{O}}( z^{(k)},\mu^{(k)}), z^{(k+1)}- z^{(k)}\right\rangle+\frac{2\lb_2+L_{\nabla\widetilde{\mathcal{H}}}}{2\mu^{(k)}}\| z^{(k+1)}- z^{(k)}\|_2^2 \\
				=& \left\langle\nabla\mathcal{R}( z^{(k)})-\nabla\mathcal{R}( z^{(k+1)})- A\zz \gamma^{{(k+1)}}-L^{(k)}( z^{(k+1)}- z^{(k)}), z^{(k+1)}- z^{(k)}\right\rangle\\&+  \frac{2\lb_2+L_{\nabla\widetilde{\mathcal{H}}}}{2\mu^{(k)}}\| z^{(k+1)}- z^{(k)}\|_2^2\\
				\leq & \frac{2\lb_2+L_{\nabla\widetilde{\mathcal{H}}}-2\mu^{(k)}L^{(k)}}{2\mu^{(k)}}\| z^{(k+1)}- z^{(k)}\|_2^2-\left\langle A\zz \gamma^{{(k+1)}}, z^{(k+1)}- z^{(k)}\right\rangle+2\lb_2\| z^{(k+1)}- z^{(k)}\|^2_2 
				\\
				\leq& \frac{6\lb_2+L_{\nabla\widetilde{\mathcal{H}}}-2\mu^{(k)}L^{(k)}}{2\mu^{(k)}}\| z^{(k+1)}- z^{(k)}\|_2^2+(\gamma^{{(k+1)}})\zz  (A z^{(k)}-c) \leq 0.
			\end{aligned}
		\end{equation}
		
		Due to the nondecreasing property of {$\{\widetilde{\mathcal{O}}(
			z^{(k+1)},\mu)\}$ with respect to smoothing parameter $\mu$,} 
		we have
		$\widetilde{\mathcal{O}}( z^{(k+1)},\mu^{(k+1)})\leq\widetilde{\mathcal{O}}( z^{(k+1)},\mu^{(k)})$ with $\mu^{(k+1)}\leq \mu^{(k)}$. Together with {the relations \eqref{eq:obound} and \eqref{eq:1},} we arrive at $\widetilde{\mathcal{O}}( z^{(k+1)},\mu^{(k+1)})\leq\widetilde{\mathcal{O}}( z^{(k)},\mu^{(k)})$ and $\mathcal{O}(z^{(k+1)})<\theta$. Besides, it follows from the definition of $z^{(k+1)}$ that $z^{(k+1)}\in\mathcal{Z}$. Hence, we have $z^{(k+1)}\in \Omega_\theta\cap {\cal Z}$  and $\widetilde{\mathcal{O}}(z^{(l)},\mu^{(l)})\leq\widetilde{\mathcal{O}}(z^{(l-1)},\mu^{(l-1)})$ hold for all $l=1,2,\ldots,k+1$. This completes the part (a) by mathematical induction.
		
		$($b$)$ {By what we have proved in $($a$)$ and the inequality
			$\widetilde{\mathcal{O}}( z^{(k+1)},\mu^{(k)})- \widetilde{\mathcal{O}}( z^{(k)},\mu^{(k)})\geq -\tau_2\frac{\mu^{(k)}}{L^{(k)}}$, we can obtain that}
		\begin{equation*}
			-\tau_2\frac{\mu^{(k)}}{L^{(k)}}\leq\frac{6\lb_2+L_{\nabla\widetilde{\mathcal{H}}}-2\mu^{(k)}L^{(k)}}{2\mu^{(k)}}\| z^{(k+1)}- z^{(k)}\|_2^2+(\gamma^{{(k+1)}})\zz  (A z^{(k)}-c).
		\end{equation*}
		{which together with the relations 
			$(\gamma^{{(k+1)}})\zz  (A z^{(k)}-c)\leq 0$ and $\mu^{(k)}L^{(k)}>8\lb_2+L_{\nabla\widetilde{\mathcal{H}}}$ further implies that}
		\begin{align}
			&\| z^{(k+1)}- z^{(k)}\|_2^2\leq \frac{2\tau_2(\mu^{(k)})^2}{L^{(k)}(2\mu^{(k)}L^{(k)}-6\lb_2-L_{\nabla\widetilde{\mathcal{H}}})}\label{eq:kktsa1-1},\\
			&-	\tau_2\frac{(\mu^{(k)})^2}{8\lb_2+L_{\nabla\widetilde{\mathcal{H}}}}\leq (\gamma^{{(k+1)}})\zz  (A z^{(k)}-c)\leq 0.\label{eq:kktsa1-2}
		\end{align}
		It follows from the inequality~\eqref{eq:kktsa1-1} and the KKT condition~\eqref{eq:KKTsub1} 
		that
		\begin{eqnarray*}
			&&\left\|\nabla\widetilde{\mathcal{O}}( z^{(k)},\mu^{(k)})+ A\zz \gamma^{{(k+1)}}\right\|_{2}=\left\|\nabla\mathcal{R}( z^{(k)})-\nabla\mathcal{R}( z^{(k+1)})-L^{(k)}( z^{(k+1)}- z^{(k)})\right\|_{2}\notag\\
			&\leq& \left\|\nabla\mathcal{R}( z^{(k)})-\nabla\mathcal{R}( z^{(k+1)})\right\|_{2}+L^{(k)}\left\| z^{(k+1)}- z^{(k)}\right\|_{2}\notag\leq(2\lb_2+L^{(k)})\left\| z^{(k+1)}- z^{(k)}\right\|_{2}\notag\\
			&\leq& \sqrt{\frac{2\tau_2}{L^{(k)}(2\mu^{(k)}L^{(k)}-6\lb_2-L_{\nabla\widetilde{\mathcal{H}}})}}\mu^{(k)}(2\lb_2+L^{(k)})
			\,\leq\, 2\sqrt{\tau_2}(\mu^{(k)})^{1/2},\notag
		\end{eqnarray*}
		where the last inequality {results from the inequalities}  
		$0<\mu^{(k)}<1$, $\mu^{(k)}L^{(k)}>8\lb_2+L_{\nabla\widetilde{\mathcal{H}}}$ and $L^{(k)}>2\lb_2$.
		Together with~\eqref{eq:kktsa1-2}, 
		we can conclude the proof.
	\end{proof}
	
	\section{A Structured Algorithm for Solving \eqref{eq:updateWBV1}}\label{sec:alg2} 
	
	{\color{black} We notice that the subproblem~\eqref{eq:updateWBV1} is a convex quadratic programming (QP), which can be solved by any QP solvers such as
		`quadprog'~\cite{turlach2019package},
		the default QP solver in MATLAB, and `CVX'~\cite{grant2014cvx}.
		Since the subproblem of our SPG to be solved in 
		autoencoder scenario is usually large-scale but structured,
		the existing solvers are not efficient enough. Therefore, in this subsection we propose a special algorithm for subproblem ~\eqref{eq:updateWBV1}  to take the structure into account.}
	{
		We focus on this subproblem at the $k$-th iteration of SPG for any $k=0,1,...$.
		For brevity,  we will drop the superscript $(k)$ and let $(\bar{ W},\bar{ b},\bar{ V})$ 
		to denote the current iterate $(W^{(k)}, b^{(k)}, V^{(k)})$ (and similarly for $\mu^{(k)}$
		and $L^{(k)}$) in this section. In addition, we introduce a new group of variables $U=(u_1,u_2,\ldots,u_N)$ 
		subject to $u_n=Wx+b_1$ for all $n=1,...,N$. Hence
		the quadratic programming \eqref{eq:updateWBV1} can be reformulated as
		\begin{equation}\label{eq:pro1}
			\begin{aligned}
				\min_{W,b,V,U,\rho}&\,\,\left\langle g_W,W-\bar{W}\right\rangle
				+\left\langle g_{b},b-\bar{b}\right\rangle+\left\langle g_V,V-\bar{V}\right\rangle
				+\lb_1\|W\|_F^2+\lb_2\sum_{n=1}^Ne\zz v_n
				+\frac{L}{2}\|z-\bar{ z}\|_2^2\\
				\st 
				\quad& \,\,\,\,\,b\in\Omega_3, v_n\geq u_n,\,\, v_n\geq 0,\,\, u_n=Wx+b_1,\,\, \mbox{ for\, all\, }
				n=1,2,\ldots,N,
			\end{aligned}
		\end{equation}
		where $g_W=\nabla_W\widetilde{\mathcal{H}}(\bar{W},\bar{b},\bar{V},\mu),$ $g_V=\nabla_V\widetilde{\mathcal{H}}(\bar{W},\bar{b},\bar{V},\mu),$ 
		and $g_{b}=\nabla_{b}\widetilde{\mathcal{H}}(\bar{W},\bar{b},\bar{V},\mu)$
		are preset constants in this subproblem.
		
		The variables of problem \eqref{eq:pro1} can be divided into two parts $(W,b)$
		and $(V,U)$.
		We then apply the alternating direction method of multipliers (ADMM) to solve \eqref{eq:pro1}.
		By penalizing the equality constraints, we obtain the augmented Lagrange penalty function 
		\begin{equation*} 
			\begin{aligned}
				\mathcal{G}(W,b,V,U,\rho):=&\lb_1\|W\|_F^2+\lb_2\sum_{n=1}^Ne\zz v_n+\left\langle g_W,W-\bar{W}\right\rangle
				+\left\langle g_{b},b-\bar{b}\right\rangle+\left\langle g_V,V-\bar{V}\right\rangle\\
				&+\sum_{n=1}^N\left\langle\rho_n,u_n- (Wx_n+b_1) \right\rangle+\frac{1}{2}\sum_{n=1}^N\|u_n- (Wx_n+b_1)\|_2^2+\frac{L}{2}\|z-\bar{ z}\|_2^2,
			\end{aligned}
		\end{equation*}
		where $\rho=(\rho_1,\rho_2,\ldots,\rho_N)$ with $\rho_n\in\R^{N_1}$, for all $n=1,...,N$, are the
		Lagrangian multipliers associated with the equality constraints.

		At the $l$-th iteration, we first fix $W=W^{(l)}$, $b=b^{(l)},\rho=\rho^{(l)}$, and the $(V,U)$ subproblem
		can be formulated as
		\begin{equation*}\label{eq:UVsub}
			\begin{aligned}
				\min_{U,V}\,\,&\mathcal{G}(W^{(l)},b^{(l)},V,U,\rho^{(l)})\\
				\st\,\,& v_n\geq u_n,\,\, v_n\geq 0, \text{ for all }n=1,2,\ldots,N.
			\end{aligned}
		\end{equation*}
		
		Due to the separability of $v_n$ and $u_n$ for all
		$n=1,...,N$, the $(V,U)$ subproblem has also  a closed-form solution, {\color{black} which} is illustrated as follows,
		\begin{equation}\label{eq:solUVsub}
			\left\{\begin{aligned}
				&V^{(l+1)}_{j,n}=-\xi_{j,n}^{1;l},U^{(l+1)}_{j,n}=-\xi_{j,n}^{2;l},\quad\text{ if } \xi_{j,n}^{2;l}\geq \xi_{j,n}^{1;l} \text{ and }\xi_{j,n}^{1;l}\leq 0,\\
				&V^{(l+1)}_{j,n}=0,U^{(l+1)}_{j,n}=-\xi_{j,n}^{2;l},\hspace{0.36in}\text{ if } \xi_{j,n}^{2;l}\geq 0, \xi_{j,n}^{1;l} > 0,\\
				&V^{(l+1)}_{j,n}=U^{(l+1)}_{j,n}=0,\hspace{0.72in}\text{ if } \xi_{j,n}^{2;l}<0, \text{ and }L\xi_{j,n}^{1;l}+\xi_{j,n}^{2;l}> 0,\\
				&V^{(l+1)}_{j,n}=U^{(l+1)}_{j,n}=-\frac{L\xi_{j,n}^{1;l}+\xi_{j,n}^{2;l}}{L+1},\hspace{0.01in}\text{ otherwise}
			\end{aligned}\right.
		\end{equation}
		for all $j=1,2,\ldots,N_1$ and $n=1,2,\ldots,N$.
		Here, $\xi_n^{1;l}=g_{v_n}/L-\bar{ v}_n+\lb_2e/L$, $\xi_n^{2;l}=\rho^{(l)}_n-(W^{(l)}x_n+b^{(l)}_1)$,  $g_{v_n}$ is the $n$-th column of $g_V$,  
		and $\xi_{j,n}^{1;l}$ and $\xi_{j,n}^{2;l}$ are the $j$-th elements of $\xi_n^{1;l}$
		and $\xi_n^{2;l}$, respectively, for all $j=1,2,\ldots,N_1$ and $n=1,2,\ldots,N$.
		
		Secondly, we fix $V=V^{(l+1)}$, $U=U^{(l+1)},\rho=\rho^{(l)}$, and then the $(W,b)$ {\color{black} subproblem}
		can be written as
		\begin{equation*} 
			\begin{aligned}
				\min_{W,b\in\Omega_3}\,\,&\mathcal{G}(W,b,V^{(l+1)},U^{(l+1)},\rho^{(l)}).
			\end{aligned}
		\end{equation*}
		By simply calculation, we can obtain its closed-form solution as
		follows.
		\begin{equation}\label{eq:solWbsub}
			W^{(l+1)}=\widehat{W}^{(l+1)}\widetilde{I}\zz, \, \mbox{and} \quad b^{(l+1)}=\mathrm{Proj}_{\Omega_3}(\widehat{W}^{(l+1)}s_{N_0+1}, \bar{ b}_2 - g_{b_2}/L),
		\end{equation}
		where 
		\begin{equation*}
			\widehat{W}^{(l+1)}=\left(-[g_W, g_{b_1}]+L[\bar{W},\bar{b}_1]+\rho^{(l)}\widehat{X}\zz +U^{(l+1)}\widehat{X}\zz \right)\left(LI_{N_0+1}+2\lb_1\widetilde{I}\zz \widetilde{I}+\widehat{X}\widehat{X}\zz \right)^{-1},
		\end{equation*}
		$s_{N_0+1}=(0,0,\ldots,0,1)\in\R^{N_0+1}$, 
		$[g_{b_1}\zz, g_{b_2}\zz] = g_b\zz$, 
		$\widehat{X}:=(X\zz,1_N)\zz$ and 
		$\widetilde{I}=[I_{N_0},0]$.
		
		Finally, we present the framework of ADMM for solving  the subproblem  \eqref{eq:updateWBV1}. 
		
		\begin{algorithm}[H]
			\caption{A Splitting and Alternating Method for the Quadratic Programming~\eqref{eq:updateWBV1} (SAMQP)}\label{alg:activeset}
			\begin{algorithmic}[1]
				\STATE{	Initialization: set $(W^{(l)},b^{(l)},V^{(l)})=(\bar{W},\bar{b},\bar{V})$,  $\rho^{(l)}=0$, $u^{(l)}_n=W^{(l)}x_n+b^{(l)}_1$ for all $n=1,...,N$, and $l:=0$.}
				\WHILE{the stop criterion is not met}
				\STATE{Calculate $V^{(l+1)},U^{(l+1)}$ by~\eqref{eq:solUVsub};}
				\STATE{Calculate $W^{(l+1)},b^{(l+1)}$ by~\eqref{eq:solWbsub};}
				\STATE{Calculate $\rho_n^{(l+1)}=\rho_n^{(l)}+(u^{(l+1)}_n- (W^{(l+1)}x_n+b^{(l+1)}_1))$, for
					all $n=1,2,\ldots,N$;}
				\STATE{Set $l:=l+1$.}
				\ENDWHILE
			\end{algorithmic}
		\end{algorithm}
		
		Since the subproblem~\eqref{eq:updateWBV1} is strongly convex, any sequence generated by~SAMQP, a two block ADMM,  converges to a global solution of~\eqref{eq:updateWBV1}. {\color{black} 
			Furthermore, the local R-linear rate convergence
			of SAMQP 
			can be guaranteed by Boley \cite{boley2013local}. }
	
	To test the efficiency of SAMQP, we
	construct the following randomly generated test problems. We set $X=\mathrm{rand}(N_0,N)$,
	$g_W=\mathrm{rand}(N_1,N_0)$, $g_b=\mathrm{rand}(N_1+N_0,1)$, $g_V=\mathrm{rand}(N_1, N)$,
	$\bar{W} = \mathrm{randn}(N_1,N_0)/N$,  
	$\bar{V}_n=(\bar{W}x_n)_{+}$ for all $n=1,2,\ldots,N$ and
	$\bar{b}=0$.
	The problem parameters $\mu$ and $L$ are set as $0.001$ and $1$, respectively.
	In addition, the stopping criterion is set as  $$\max\left\{\|\rho^{(l+1)}-\rho^{(l)}\|_F^{2},\|U^{(l+1)}-U^{(l)}\|_F^{2}\right\}\leq 10^{-6}.$$
	
	We compare SAMQP with some existing QP solvers including the `quadprog' solver from MATLAB, the `fmincon' solver from MATLAB and the `CVX' solver \cite{grant2014cvx} for solving \eqref{eq:updateWBV1}. 
	We choose seven test problems with different sizes. We record the CPU time in seconds required
	by these solvers. The results are displayed in Table \ref{tb:1}, in which
	``--" stands for the cases that the solver runs out of memory during the iteration
	or terminates abnormally. It can be easily observed that SAMQP is the most efficient and robust one
	among these four solvers.}

\begin{table}[tbhp]
	\centering
	\scriptsize
	\caption{A comparison of CPU time for several solvers and~SAMQP.}
	\begin{tabular}{|c|c|c|c|c|c|c|c|c|c|}
		\hline
		\multicolumn{4}{|c|}{}&\multicolumn{4}{|c|}{CPU time (s)}\\
		\hline
		$N$&$N_1$ & $N_0$&$N_2$& `fmincon' & `quadprog'  & `CVX'  & SAMQP \\
		\hline
		100&5&5&535&3.502&0.707&2.031&0.099\\
		\hline
		100&10&10&1120&33.990&4.546&1.172&0.105\\
		\hline
		100&20&20&2440&674.163&39.303&1.781&0.189\\
		\hline
		100&40&40&5680&--&359.555&6.672&0.419\\
		\hline
		100&100&10&11110&--&--&7.453&0.838\\
		\hline
		1000&100&10&101110&--&--&50.781&6.056\\
		\hline
		10000&784&1000&8625784&--&--&--&189.868\\
		\hline
	\end{tabular}\label{tb:1}
\end{table}

\end{document}